\documentclass[reqno]{amsart}
\usepackage{setspace}
\usepackage[utf8]{inputenc}
\usepackage{graphicx}
\graphicspath{ {./images/} }
\usepackage[pagewise]{lineno}
\usepackage{amsmath, amsthm, amssymb, amstext}
\usepackage{enumitem}
\usepackage{amssymb}
\usepackage[foot]{amsaddr}
\usepackage[all,arc]{xy}
\usepackage{mathrsfs}
\usepackage[margin=1in]{geometry}
\usepackage{amsmath,amsthm}
\usepackage{mathtools}
\usepackage[T1]{fontenc}
\usepackage{fourier}
\usepackage{cite}
\usepackage{stmaryrd}
\usepackage{amsfonts}
\usepackage{hyperref}
\usepackage[capitalise,nameinlink]{cleveref}

\DeclareMathOperator*{\essinf}{ess\,inf}

\DeclareMathOperator*{\spp}{supp}
\newcommand*\diff{\mathop{}\!\mathrm{d}}

\newtheorem{theorem}{Theorem}
\newtheorem{remark}[theorem]{Remark}
\newtheorem{lemma}[theorem]{Lemma}
\newtheorem{proposition}[theorem]{Proposition}

\newtheorem{definition}[theorem]{Definition}
\newtheorem{definitions}[theorem]{Definitions}
\newtheorem{example}[theorem]{Example}

\usepackage{xcolor}
\hypersetup{
colorlinks,
linkcolor={blue!100!black},
citecolor={red!100!black},
urlcolor={green!100!black}
}

\newcommand{\Phix}{\ensuremath{{\Phi}}}
\renewcommand{\L}{\left}
\renewcommand{\r}{\right}

\newcommand{\norm}[1]{|\!|#1|\!|}

\newcommand{\curly}[1]{\left\{#1\right\}}

\newcommand{\round}[1]{\left(#1\right)}

\newcommand{\scal}[1]{\left\langle#1\right\rangle}

\renewcommand{\l }{\lambda }
\renewcommand{\O }{\Omega }
\newcommand{\h}{\mathcal{H}}
\newcommand{\R}{{\mathbb R}}
\newcommand{\N}{{\mathbb N}}
\newcommand{\RD}{\mathbb{R}^d}

\def\ds{\displaystyle}
 \def\W{W^{1,\Phi}(\RD)}

\def\lh{L^{\Phi}(\O)}
\def\WV{W^{1,\Phi}_V(\R^d)}

\def\a{\hat{a}(x)}
\def\D{\mathcal{D}}
\def\B{\mathcal{B}}

\DeclarePairedDelimiterX{\inp}[2]{\langle}{\rangle}{#1, #2}

\makeatletter
\def\@tocline#1#2#3#4#5#6#7{\relax
\ifnum #1>\c@tocdepth 
\else
\par \addpenalty\@secpenalty\addvspace{#2}%
\begingroup \hyphenpenalty\@M
\@ifempty{#4}{%
\@tempdima\csname r@tocindent\number#1\endcsname\relax
}{%
\@tempdima#4\relax
}%
\parindent\z@ \leftskip#3\relax \advance\leftskip\@tempdima\relax
\rightskip\@pnumwidth plus4em \parfillskip-\@pnumwidth
#5\leavevmode\hskip-\@tempdima
\ifcase #1
\or\or \hskip 1em \or \hskip 2em \else \hskip 3em \fi%
#6\nobreak\relax
\dotfill\hbox to\@pnumwidth{\@tocpagenum{#7}}\par
\nobreak
\endgroup
\fi}
\makeatother
\usepackage{hyperref}

\numberwithin{theorem}{section}
\numberwithin{equation}{section}

\usepackage{epigraph}

\date{}
\begin{document}
\setcounter{page}{1}

\vspace*{1.0cm}
\title[The concentration-compactness principle for Musielak-Orlicz spaces
and applications]
{The concentration-compactness principle for Musielak-Orlicz spaces
and applications}
\author[A.E. Bahrouni, A. Bahrouni]{ Ala Eddine Bahrouni and Anouar Bahrouni}
\maketitle
\vspace*{-0.6cm}
\begin{abstract}

This paper extends the Concentration-Compactness Principle to Musielak-Orlicz spaces, working in both bounded and unbounded domains. We show that our results include important special cases like classical Orlicz spaces, variable exponent spaces, double phase spaces, and a new type of double phase problem where the exponents depend on the solution. Using these general results with variational methods, we prove that certain quasilinear equations with critical nonlinear terms have solutions.\\
\smallskip\noindent {\bf Keywords.}
Musielak-Orlicz spaces, generalized Young functions, Concentration-compactness principle, Variational methods.\\
\smallskip\noindent  {\bf 2020 Mathematics Subject Classification.} 35A01, 35A20, 35J25,35J62.
\end{abstract}
\tableofcontents

\section{Introduction}

\subsection{Overview}

Since its introduction by P.L. Lions in 1985, the Concentration--Compactness Principle (CCP) has become one of the most influential tools in the study of nonlinear partial differential equations (PDEs) and variational problems. Lions, in his famous articles \cite{Lions1985,LIONS1985},  developed this principle to address the lack of compactness in critical problems, particularly those involving unbounded domains or critical Sobolev exponents.

Over the years, the CCP has been extended and applied to a wide range of problems driven by different types of operators, such as, classical Orlicz (see  \cite{Bonder2024,Fukagai2006}), variable exponents case (see \cite{bonder2010, Chung2021, Fu2009,Fu2010, Ho2020, Ho2021}), double phase variable exponents (see \cite{Ha2024, Ha2025}, and logarithmic double phase with variable exponents (see \cite{Arora2025}) in bounded domain. Despite extensive research on extensions of the CCP, several significant cases remain unexplored, including the logarithmic double-phase problem with variable exponents in unbounded domains and operators with exponents that depend on both the solution and its gradient in various domain scenarios. Consequently, this paper aims to bridge this gap by establishing the CCP within the framework of Musielak--Orlicz spaces in both bounded and unbounded domains. The inherent nonhomogeneity of these spaces precludes the direct application of classical analysis techniques, rendering the problem particularly challenging.

For Orlicz spaces (see \cite{ Harjulehto2019} for detailed discussions), the first extension of Lions' CCP to Orlicz spaces was achieved by Fukagai, Ito and Nakamura in \cite{Fukagai2006}. They established concentration estimates using global bounds on the Orlicz  Young function, laying the groundwork for analyzing concentration phenomena in these spaces. Later, in 2024, Bonder and Silva \cite{Bonder2024} generalized the CCP further by deriving sharper estimates for blowing-up sequences. Note that there are many CCP results in special cases of Orlicz spaces, such as classical Sobolev spaces; see \cite{tintarev2007concentration, Lions1984, Lions1985 } and references therein.

For Sobolev space with variable exponents, see \cite{DieningHarjulehtoHastoRuzicka2011, Kovacik1991 } for more details of this space, the study of the CCP in these spaces was initiated by Fu \cite{Fu2009} in 2009, focusing on bounded domains. Subsequently, in 2010, Fu and Zhang \cite{Fu2010} established the CCP in the variable exponent Sobolev space \( W^{1,p(x)} \) defined on the entire space \(\mathbb{R}^d\). However, the result in \cite{Fu2010} did not provide any information regarding the loss of mass of the sequence at infinity. Later, in 2019, Ho, Kim, and Sim \cite{Ho2019} proved the CCP in weighted Sobolev spaces with variable exponents. Moreover, they established a theorem that provides insight into the possibility of concentration at infinity. Several other papers have addressed the CCP in these spaces; see, for example, \cite{bonder2010, Chung2021, Ho2020, Ho2021} and the references therein.

\vspace{6mm}

Another important setting where the CCP  has been extensively studied is the Musielak-Orlicz Sobolev space $ W^{1,\Phi}(\O )$, which arises naturally in the analysis of double-phase problems. The associated double-phase operator takes the form
$$
\text{div}\left(|\nabla u|^{p(x)-2}\nabla u + a(x)|\nabla u|^{q(x)-2}\nabla u\right),
$$
where \( p, q \colon \Omega \to \mathbb{R} \) are variable exponents satisfying \( 1 < p(x) \leq q(x) \) for all \( x \in \Omega \), and \( a \colon \Omega \to [0, +\infty) \) is a weight function modulating the transition between two distinct elliptic regimes. The corresponding generalized Young function is given by
\begin{equation}\label{nfct=dble}
\Phi(x,t) = t^{p(x)} + a(x)t^{q(x)},  \quad \text{for all } x \in \Omega \text{ and all } t \geq 0,
\end{equation}
which captures the intrinsic interplay between the two phases.

The study of double phase problems, (with constant exponent), originated in the pioneering works of Zhikov \cite{Zhikov95,Zhikov97} on homogenization and Lavrentiev's phenomenon in the context of elasticity theory. These operators have since found significant applications in transonic flows \cite{Bahrouni2019}, quantum physics \cite{Benci2000}, and the analysis of composite materials \cite{Zhikov2011}, where the weight function \( a(\cdot) \) plays a crucial role in describing the transition between different material phases. Mathematically, double-phase operators fall under the class of functionals with \textit{nonstandard growth conditions}, introduced by Marcellini \cite{Marcellini1989, Marcellini1991}. Significant progress in their regularity theory has been made by Mingione and collaborators \cite{Baroni2015, Colombo2014}.

Returning to the CCP, recent work by Ha and Ho \cite{Ha2024,Ha2025} has established Lions' CCP in the setting of Musielak-Orlicz-Sobolev spaces $W^{1,\Phi}(\Omega)$, where
$\Phi$ is defined in \eqref{nfct=dble}, for both bounded and unbounded domains $\Omega$.  These spaces, first introduced by Crespo-Blanco-Gasi\'nski-Harjulehto-Winkert \cite{CrespoBlancoGasinskiHarjulehtoWinkert2022}, provide a broad framework that encompasses classical Sobolev and Orlicz spaces as special cases.

Very recently, Arora--Crespo--Blanco--Winkert \cite{Arora2025} established a CCP in a particular Sobolev-Musielak-Orlicz space in bounded domain, where the generalized Young function is a logarithmic double-phase function with variable exponents. This Young function is defined as
\begin{align}\label{N-function}
	\Phi(x,t) = a(x) t^{p(x)} + b(x) t^{q(x)} \log^{s(x)}(1 + t) \quad \text{for } (x, t) \in \Omega \times (0, \infty),
\end{align}
where $p, q \in C(\overline{\Omega})$ satisfy $1 < p(x), q(x) < d$ for almost all $x \in \overline{\Omega}$, $s \in L^\infty(\Omega)$ with $q(x) + s(x) \geq r > 1$, and $0 \leq a, b \in L^1(\Omega)$ such that $a(x) + b(x) \geq c > 0$ for all $x \in \overline{\Omega}$. This generalized Young function extends the earlier logarithmic double-phase function
\begin{align*}
	\Phi_{\log}(x,t) = t^{p(x)} + \mu(x) t^{q(x)} \log (e + t) \quad \text{for all } (x,t) \in \overline{\Omega} \times [0, \infty),
\end{align*}
which was deeply studied in \cite{Arora2023}. The associated energy functional is given by
\begin{align}\label{log-functional}
	u \mapsto \int_\Omega \left( \frac{|\nabla u|^{p(x)}}{p(x)} + \mu(x) \frac{|\nabla u|^{q(x)}}{q(x)} \log (e + |\nabla u|) \right) \mathrm{d}x.
\end{align}

This class of Young functions with logarithmic terms was first introduced by Baroni--Colombo--Mingione \cite{Baroni2016} to prove the local H\"older continuity of the gradient of local minimizers of the functional
\begin{align} \label{log-functional-Mingione1}
	w \mapsto \int_\Omega \left[ |D w|^p + a(x) |D w|^p \log (e + |D w|) \right] \mathrm{d}x,
\end{align}
where $1 < p < \infty$ and $0 \leq a(\cdot) \in C^{0,\alpha}(\overline{\Omega})$. Notably, the functionals in \eqref{log-functional} and \eqref{log-functional-Mingione1} coincide when $p = q$ are constant.
For further results concerning problems driven by this class of operators, we refer the reader to \cite{DeFilippis2023,Fuchs2000,Marcellini2006,Seregin1999,Fuchs2000a}, and the references therein.

\vspace{5mm}
As a final remark in this subsection, we conclude that all the CCP results discussed here pertain to specific cases of Musielak-Orlicz spaces. This naturally raises the question of identifying the minimal assumptions required to guarantee the validity of the CCP in the broader setting of generalized Musielak-Orlicz spaces.

The main goal of this paper is to bridge this gap by establishing the CCP in the more general setting of Musielak-Orlicz spaces in both bounded and unbounded domains. These spaces, which encompass Orlicz spaces and variable exponent spaces as special cases, provide a flexible and unifying framework for studying problems with nonstandard growth conditions. Our results extend the classical CCP to this broader context, offering new insights and powerful tools for analyzing concentration and compactness phenomena in Musielak-type functionals.

However, studying the CCP in Musielak--Orlicz--Sobolev spaces presents several significant challenges:

\begin{enumerate}[label=\textbf{\roman*.}, leftmargin=*, align=left]
    \item \textbf{Unknown explicit structure:}
    The generalized Young function $\Phi$ often lacks an explicit form, complicating the analysis of energy functionals. This necessitates abstract tools tailored to Musielak--Orlicz spaces rather than relying on specific Young function structures, see Subsection~\ref{subsec: techL}.

    \item \textbf{Nonhomogeneous complexity:}
    Unlike classical Sobolev spaces with uniform growth conditions, the variable growth and nonstandard coercivity in Musielak--Orlicz spaces hinder uniform estimates and compactness results. These challenges are further amplified when dealing with sequences that may lose mass at infinity, as the interplay between the nonhomogeneous structure and the lack of explicit Young function complicates the analysis of concentration phenomena.

    \item \textbf{Gap in Sobolev conjugate theory:}
    One of the principal motivations for employing the CCP in variational problems lies in the presence of critical growth in the nonlinear terms, which necessitates a careful analysis of the associated Sobolev conjugate for generalized Young functions. The study of such conjugates was initiated by Fan~\cite{Fan2012}, who introduced the concept under the assumption that $\Phi$ satisfies a Lipschitz condition in the spatial variable. This theory was significantly advanced by Cianchi and Diening~\cite{Cianchi2024}, who established a sharper formulation under weaker regularity assumptions. Their construction preserves a key property: for each fixed $x$, it coincides pointwise with the classical sharp Sobolev conjugate in Orlicz spaces presented  in~\cite{Cianchi1996}, while adapting to the more complex structure of Musielak-Orlicz spaces. Unlike the classical Orlicz-Sobolev setting, the Sobolev conjugates $\Phi_d$ and $\Phi_{d, \diamond}$ introduced in~\cite{Cianchi2024} are tailored respectively to unbounded domains (e.g., $\mathbb{R}^d$) and bounded domains. However, despite these advances, a complete theory of Sobolev conjugates for generalized Young functions remains incomplete, lacking crucial structural properties such as the $\Delta_2$-condition, which is notoriously difficult to verify. To overcome these obstacles in bounded domain, we introduce a new critical exponent, denoted $\Phi^*$, satisfying the asymptotic relations $\Phi^* \sim \Phi_{d, \diamond}$. We then establish key properties of this new Sobolev conjugate, particularly in the context of the CCP (see Subsection \ref{conjsob}). 
    This argument fails in unbounded domains due to the behavior of the Sobolev conjugate \( \Phi_d \) in a neighborhood of zero. To circumvent this difficulty, we introduce an additional assumption on $ \Phi_d $; see \eqref{HAST}.

    \item \textbf{Limitations of existing CCP techniques:}
    An additional difficulty arises from the fact that the standard CCP approach developed for Orlicz spaces \cite{Bonder2024} cannot be directly applied in our setting. This technique fundamentally relies on the Matuszewska-Orlicz function defined by
    \begin{equation}\label{eq:Matuszewska}
    M_\Phi(x,t) := \limsup_{s\to +\infty} \frac{\Phi(x,st)}{\Phi(x,s)},
    \end{equation}
    where the \textit{uniformity} of the limit in the spatial variable $x$ is essential for the argument. In the classical Orlicz case (where $\Phi$ is independent of $x$), this uniformity holds automatically. However, for certain classes of generalized Young functions, (particularly those defined in \eqref{nfct=dble} and \eqref{N-function}), this limit fails to be uniform with respect to $x$.
 This lack of uniformity prevents the direct implementation of the CCP framework used in \cite{Bonder2024}, necessitating alternative techniques to establish compactness results in these more general function spaces. We note, however, that uniformity may still hold for specific cases, such as variable exponent functions $t^{p(x)}$ under suitable regularity conditions on $p(x)$, see Theorems~\ref{CCP10} and \ref{CCP20}.
\end{enumerate}

\subsection{ Main results}
In this subsection, we present our main results concerning the CCP in the Sobolev-Musielak space \( W^{1,\Phi}_0(\Omega) \), as stated in Theorem \ref{CCP1}, and in the space \( W^{1,\Phi}_V(\mathbb{R}^d) \) defined on the whole space, as detailed in Theorems \ref{CCP2} and \ref{CCP3}, (in the next section, we provide the definitions and fundamental properties of the spaces \( W^{1,\Phi}_V(\mathbb{R}^d) \) and \( W^{1,\Phi}_0(\Omega) \)). Moreover, as one of the main contributions of this work, we establish the existence of a weak solution to a quasilinear equation with critical growth, as demonstrated in Theorem \ref{thm:exis}. This result is deeply rooted in the CCP framework developed in Theorems \ref{CCP1}, \ref{CCP2}, and \ref{CCP3}, highlighting the interplay between functional analysis and variational methods in solving nonlinear problems.

We adopt the following notation: $u_n \to u$ denotes strong convergence, while $u_n \rightharpoonup u$ and $u_n \overset{\ast}{\rightharpoonup} u$ represent weak and weak-$\ast$ convergence, respectively, as $n \to \infty$ in the relevant normed space (implied by context). The duality pairing between a space and its dual is written as $\langle \cdot, \cdot \rangle$. For a measurable subset $E \subset \RD$, $|E|$ stands for the Lebesgue measure of $E$. For Euclidean balls, $B_\epsilon(x_0) \subset \RD$ denotes the open ball centered at $x_0$ with radius $\epsilon > 0$. The notation $X \hookrightarrow Y$ indicates that the space $X$ is embedded continuously into the space $Y$, while $X \hookrightarrow\hookrightarrow Y$ means that $X$ is embedded compactly into $Y$. Constants dependent only on the problem's assumptions or data are generically labeled $C_i,\ c_i$ ($i \in \mathbb{N}$).

Let \(\Omega \subseteq \mathbb{R}^d\) be a domain. Let \( C_c(\Omega) \) be the set of all continuous functions \( u : \Omega \to \mathbb{R} \) whose support is compact, and let \( C_0(\Omega) \) be the completion of \( C_c(\Omega) \) relative to the supremum norm \( \|\cdot\|_\infty \). 
We denote by \(\mathbb{M}(\Omega)\) the space of finite non-negative Borel measures on \(\Omega\). A sequence \(\mu_n \overset{\ast}{\rightharpoonup} \mu\) in \(\mathbb{M}(\Omega)\) is defined by \(\langle \mu_n, \varphi \rangle \to \langle \mu, \varphi \rangle\), for any \(\varphi \in C_0(\Omega)\).

 Throughout the paper, we denote
$$
C_+(\overline{\Omega}) := \left\{ g \in C\left(\overline{\Omega}\right) : 1 < g^- := \inf_{x \in \overline{\Omega}} g(x) \leq g^+ := \sup_{x \in \overline{\Omega}} g(x) < \infty \right\}.
$$
In the following, the function \(\Phi : \Omega \times [0,\infty) \to [0,\infty)\) is an generalized Young functions, (see Definition \ref{dfnfct}),  such that
\begin{enumerate}[label=(H),ref=H]
    \item\label{H}
    The function $\Phi$ satisfies conditions \eqref{A1}, \eqref{A2}, and \eqref{B0}, as well as the asymptotic integrability condition:
    \begin{equation}\label{conv1}
        \int_0 \left( \frac{t}{\ds\lim_{|x| \to +\infty} \Phi(x,t)} \right)^{\frac{1}{d-1}} \, dt < \infty,
    \end{equation}
    where the condition \eqref{A1}, \eqref{A2}, and \eqref{B0} will be defined in Subsection~\ref{gnfmos}, see Definition~\ref{def: conds}.
    Moreover, we suppose that $t\mapsto \Phi(x,t)$ is increasing on $[0,+\infty[$, for a.a. $x\in \O$, and there exist two continuous positive functions \( m, \ell \in C_+(\overline{\Omega}) \) such that
    \begin{equation}\label{mar3}
        m(x) \leq \frac{\phi(x,t) \, t}{\Phi(x,t)} \leq \ell(x) < \min\left\{d, m_*^-\right\}, \quad \text{for all } x \in \Omega \text{ and } t > 0, \tag{$\Phi_0$}
    \end{equation}
    where $m_*^-= \frac{m^-d}{d-m^-}$ and $\phi(x,t) := \frac{\partial \Phi(x,t)}{\partial t}$ denotes the derivative of $\Phi$ with respect to $t$.
\end{enumerate}

Let function $V$ become a potential such that
\begin{enumerate}[label=($\mathcal{V}$),ref=$\mathcal{V}$]
    \item\label{VV} $V \in L^1_{\mathrm{loc}}(\mathbb{R}^d)$ satisfies $\operatorname{ess\,inf}_{x \in \mathbb{R}^d} V(x) > 0.$
\end{enumerate}
Under these assumptions, we obtain the continuous embedding $\WV \hookrightarrow L^{\Phi
_d}(\RD)$ for the whole space case. When $\Omega$ is a bounded Lipschitz domain, we additionally have the embedding $W^{1,\Phi}_0(\Omega) \hookrightarrow L^{\Phi^*}(\Omega)$. These embeddings yield the strictly positive Sobolev constants:
\begin{equation}\label{eq:S}
    S_1 := \inf_{\substack{v \in W^{1,\Phi}_0(\Omega) \\ v \neq 0}}
    \frac{\|v\|_{W^{1,\Phi}_0(\Omega)}}{\|v\|_{L^{\Phi^*}(\Omega)}} > 0, \quad
    S_2 := \inf_{\substack{v \in \WV \\ v \neq 0}}
    \frac{\|v\|_{\WV}}{\|v\|_{L^{\Phi_d}(\RD)}} > 0,
\end{equation}
where $\Phi^*$ and $\Phi_d$ denote the Sobolev conjugate exponents of $\Phi$ in bounded domains and in $\RD$ respectively. We emphasize that the embedding \(W^{1,\Phi}_0(\Omega) \hookrightarrow L^{\Phi^*}(\Omega)\) requires neither condition \eqref{A2} nor \eqref{conv1}, because of the boundedness of \(\Omega\). The complete definitions of conditions \eqref{A1}, \eqref{A2}, and \eqref{B0} and the detailed properties of the Musielak-Orlicz spaces $L^{\Phi_d}(\RD)$ and $L^{\Phi^*}(\Omega)$ and Sobolev spaces $W^{1,\Phi}_0(\Omega)$ and $\WV$ are provided in Section~\ref{prlmn}.

Before stating our main result, we introduce  the following functions:
\begin{equation}\label{R1}
\Phi_{\max}(x,t) := \max \left\{ t^{m(x)}, t^{\ell(x)} \right\},\quad \Phi^{\ast}_{\min}(x,t) := \min \left\{ t^{m_\ast(x)}, t^{\ell_\ast(x)} \right\},
\end{equation}
where $m_\ast(x) =\frac{m(x)d}{d-m(x)},$ and $\ell_\ast(x)=\frac{\ell(x)d}{d-\ell(x)}.$

We now begin with the CCP in a bounded domain.

\begin{theorem}\label{CCP1}
Let \(\Omega \subset \mathbb{R}^d\) be a bounded domain, and let $ \Phi $ be a generalized Young function satisfying  \eqref{H}, and let $\{u_n\}_{n\in\mathbb{N}}$ be a bounded sequence in $W^{1, \Phi}_0(\Omega)$ such that
\begin{eqnarray}
	u_n &\rightharpoonup& u \quad \text{in}\quad  W^{1,\Phi}_0(\Omega), \label{un_weak.conv}\\
	\Phi(\cdot,|\nabla u_n|)\diff x  &\overset{\ast }{\rightharpoonup }&\mu\quad \text{in } \mathbb{M}(\overline{\O}),\label{mu_n to mu*}\\
	\Phi^\ast(\cdot,|u_n|)\diff x&\overset{\ast }{\rightharpoonup }&\nu\quad \text{in } \mathbb{M}(\overline{\O}). \label{nu_n-to-nu*}
\end{eqnarray}
Then, there exist $\{x_i\}_{i\in I}\subset \bar\Omega$  of distinct points and $\{\nu_i\}_{i\in I}, \{\mu_i\}_{i\in I}\subset (0,\infty),$ where $I$ is at most countable, such that
\begin{gather}
	\nu=\Phi^\ast(\cdot,|u|)\diff x + \sum_{i\in I}\nu_i\delta_{x_i},\label{T.ccp.form.nu}\\
	\mu \geq 	\Phi(\cdot,|\nabla u|)\diff x +   \sum_{i\in I} \mu_i \delta_{x_i},\label{T.ccp.form.mu}\\
	S_1 \frac{1}{\L(\Phi_{\min}^\ast\r)^{-1}(x_i, \frac{1}{\nu_i})} \leq \frac{1}{\Phi_{\max}^{-1}(x_i, \frac{1}{\mu_i})}, \quad \forall i\in I,\label{T.ccp.nu_mu}
\end{gather}
where $\delta_{x_i}$ is the Dirac mass at $x_i$.
\end{theorem}

When the limit in \eqref{eq:Matuszewska} is uniform in  $x\in \O$ and $t\geq0$, we can strengthen the conclusion of Theorem~\ref{CCP1} regarding the estimate \eqref{T.ccp.nu_mu}. For this purpose, we introduce the following assumption
\begin{enumerate}[label=(H$_1$),ref=H$_1$]
    \item\label{H1} Let $\Phi$ be a generalized Young function satisfying \eqref{H} such that $x \mapsto \Phi(x, t)$ is continuous for each $t \geq 0$. We assume that the limits in the definition of the Matuszewska functions $M_{\Phi^*}$ and $M_{\Phi_d}$, see \eqref{eq:Matuszewska}, (associated with $\Phi^*$ and $\Phi_d$ respectively) are uniform.
\end{enumerate}

\begin{theorem}\label{CCP10}
  Under the assumptions of Theorem~\ref{CCP1} and condition~\eqref{H1}, there exist distinct points $\{x_i\}_{i \in I} \subset \overline{\Omega}$ and sequences $\{\nu_i\}_{i \in I}, \{\mu_i\}_{i \in I} \subset (0, \infty)$, where $I$ is at most countable, such that \eqref{T.ccp.form.nu} and \eqref{T.ccp.form.mu} hold. Moreover,
\begin{equation}\label{NUMU}
    S_1 \left(M_{\Phi^\ast}^{-1}\left(x_i, \frac{1}{\nu_i}\right)\right)^{-1} \leq \left(\Phi_{\max}^{-1}\left(x_i, \frac{1}{\mu_i}\right)\right)^{-1}, \quad \forall i \in I.
\end{equation}
\end{theorem}

As mentioned in the overview, we require the following assumption related to the Sobolev conjugate in the context of unbounded domains.
\begin{enumerate}[label=(H$_\ast$),ref=H$_\ast$]
    \item\label{HAST} there exist two continuous positive functions \( \kappa, \delta \in C_+(\RD) \) such that $$\ell^+ \kappa(x)\leq \delta (x),\text{ for a.a. }x \in \RD$$ and that
    \begin{equation}\label{AST}
     \Phi_{d_{\min}}(x,\l) \Phi_d(x,t)\leq  \Phi_d(x,\l t)\leq \Phi_{d_{\max}}(x,\lambda) \Phi_d(x,t), \text{ for a.a.}\, x \in \RD \text{ and for all }\l, \ t \geq 0,
    \end{equation} where $\Phi_{d_{\min}}(x,\l) =  \min \L\{\l^{\kappa (x)}, \l^{\delta (x)}\r\}$ and $\Phi_{d_{\max}}(x,\l) =  \max \L\{\l^{\kappa (x)}, \l^{\delta (x)}\r\}$.
\end{enumerate}

The following theorems establish a Lions-type CCP in $\mathbb{R}^d$.  
\begin{theorem}\label{CCP2}
Let $\Phi$ be a generalized Young function satisfying \eqref{H}, and assume that \eqref{HAST} and \eqref{VV} hold. Additionally, let $\{u_n\}_{k\in\mathbb{N}}$ be a bounded sequence in $W_V^{1, \Phi}(\RD)$ such that
\begin{eqnarray}
	u_n &\rightharpoonup& u \quad \text{in}\quad  \WV, \label{unweak.conv}\\
	\L(\Phi(\cdot,|\nabla u_n|)+  V	\Phi(\cdot,| u_n|)\r) \diff x &\overset{\ast }{\rightharpoonup }&\mu\quad \text{in }\mathbb{M}(\RD),\label{mu_nto mu*}\\
	\Phi_d(\cdot,|u_n|)\diff x&\overset{\ast }{\rightharpoonup }&\nu\quad \text{in } \mathbb{M}(\RD). \label{nu_nto-nu*}
\end{eqnarray}
Then, there exist $\{x_i\}_{i\in I}\subset \RD$ of distinct points and $\{\nu_i\}_{i\in I}, \{\mu_i\}_{i\in I}\subset (0,\infty),$ where $I$ is at most countable, such that
\begin{gather}
	\nu=\Phi
_d(\cdot,|u|)\diff x + \sum_{i\in I}\nu_i\delta_{x_i},\label{T.ccp.formnu}\\
	\mu \geq 	\L(\Phi(\cdot,|\nabla u|) + V \Phi(\cdot,|u|)\r)\diff x +   \sum_{i\in I} \mu_i \delta_{x_i},\label{T.ccp.formmu}\\
	S_2\frac{1}{\L(\Phi_{d_{\min}}\r)^{-1}(x_i, \frac{1}{\nu_i})} \leq \frac{1}{\Phi_{\max}^{-1}(x_i, \frac{1}{\mu_i})}, \quad \forall i\in I,\label{T.ccp.numu}
\end{gather}
where $\delta_{x_i}$ is the Dirac mass at $x_i$.
\end{theorem}
As in Theorem~\ref{CCP10} for bounded domains, Theorem~\ref{CCP2} can be improved for estimate \eqref{T.ccp.numu} when the limit defining $M_{\Phi_d}$ is uniform.
\begin{theorem}\label{CCP20}
  Under the assumptions of Theorem~\ref{CCP2} and condition~\eqref{H1}, there exist distinct points $\{x_i\}_{i \in I} \subset \overline{\Omega}$ and sequences $\{\nu_i\}_{i \in I}, \{\mu_i\}_{i \in I} \subset (0, \infty)$, where $I$ is at most countable, such that \eqref{T.ccp.formnu} and \eqref{T.ccp.formmu} hold. Moreover,
\begin{equation}\label{NUMU2}
    S_2 \left(M_{\Phi_d}^{-1}\left(x_i, \frac{1}{\nu_i}\right)\right)^{-1} \leq \left(\Phi_{\max}^{-1}\left(x_i, \frac{1}{\mu_i}\right)\right)^{-1}, \quad \forall i \in I.
\end{equation}
\end{theorem}

The next result elucidates the possible loss of mass at infinity.

\begin{theorem}\label{CCP3}
Let \eqref{H}, \eqref{HAST} and \eqref{VV} hold, and let $\{u_n\}_{n\in \mathbb{N}}$ be the same sequence as that in Theorem~\ref{CCP2}. Set
\begin{align}
	\mu_\infty:=\lim _{R \rightarrow \infty} \limsup _{n \rightarrow \infty} \int_{B_{R}^{c}}\Big[ \Phi(x,|\nabla u_n|) +  V(x)\Phi(x,|u_n|)\Big]\diff x \label{muinfdef}
\end{align}
and
\begin{align}
	\nu_{\infty}:=\lim _{R \rightarrow \infty} \limsup_{n \rightarrow \infty} \int_{B_{R}^{c}}\Phi_d(x,|u_n|)\diff x . \label{nuinfdef}
\end{align}
Then
\begin{align}
	\limsup_{n \to \infty}\int_{\mathbb{R}^d}\Big[\Phi(x,|\nabla u_n|) +  V(x)\Phi(x,|u_n|)\Big]\diff x
	=\mu(\mathbb{R}^d)+\mu_\infty\label{Tccpinfinity.mu}
\end{align}
and
\begin{align}
	\limsup_{n \to \infty}\int_{\mathbb{R}^d}\Phi_d(x,|u_n|)\diff x=\nu(\mathbb{R}^d)+\nu_\infty.
	\label{T.ccp.infinitynu}
\end{align}
Furthermore, consider the condition where, for every function $r \in \{m, \kappa, \ell, \delta\}$, there exists a positive constant $r_\infty \in (0, \infty)$ satisfying
$$
\lim_{|x| \to \infty} r(x) = r_\infty \text{ with } r^- \leq r_\infty \leq r^+.
$$
In this setting, the inequality
\begin{equation}\label{ala8}
	S_2 \min \left\{\nu_\infty^{\frac{1}{\kappa_{\infty}}}, \nu_\infty^{\frac{1}{\delta_{\infty}}} \right\} \leq \max \left\{\mu_\infty^{\frac{1}{m_{\infty}}},\,\mu_\infty^{\frac{1}{\ell_{\infty}}} \right\}.
\end{equation}
is valid.

\end{theorem}

\vspace{5mm}

As application of our abstract results, we aim to employ Theorems \ref{CCP2} and \ref{CCP3} to study the following equation with critical growth in $\mathbb{R}^d$:
\begin{equation}\label{prb}
	\begin{aligned}
		-\operatorname{div} &\left(\phi\L(x,|\nabla u|\r)\frac{\nabla u}{\vert \nabla u \vert}\right)
		&+ V(x) \left(\phi\L(x,|u|\r)\frac{u}{|u|}\right) =  f(x,u) + \lambda \phi_d(x,|u|)\frac{u}{|u|}, \quad x \in \mathbb{R}^d,
	\end{aligned} \tag{$\mathcal{P}_1$}
\end{equation}
where $\phi$ and $\phi_d$ denote the right derivatives of $\Phi$ and $\Phi_d$ (respectively) with respect to $t$, $V \colon \mathbb{R}^d \to \mathbb{R}$ is a measurable potential,  $\lambda > 0$ is a parameter, and $f \colon \mathbb{R}^d \times \mathbb{R} \to \mathbb{R}$ is a Carathéodory function satisfying a subcritical growth condition. The presence of critical growth $\phi_d$ in this problem introduces additional challenges, which we address using the CCPs established earlier.\\

Our final main result applies Theorems~\ref{CCP2} and~\ref{CCP3} to establish the existence of solutions to problem~\eqref{prb}. By employing these CCPs, we demonstrate the precompactness of Cerami sequences, which is a crucial step in proving existence results via variational methods based on the mountain pass theorem.

Before stating our main result, we introduce the assumptions imposed on the nonlinearity $ f $, namely:
\begin{enumerate}[label=\textnormal{(F)},ref=\textnormal{F}]
	\item\label{F}
		$f\colon \RD \times \R \rightarrow \R$ is a Carath\'{e}odory function (i.e.\,$f(\cdot, t)$ is measurable in $\RD$ for all $t\in  \R$ and $f(x, \cdot)$ is continuous in $\R$ for a.a.\,$x \in \RD$ ), such that
		\begin{enumerate}
			\item[\textnormal{(i)}] There exist two generalized Young functions $\mathcal{B}(x,t) = \ds\int_{0}^t b(x,s) \,\mathrm{d}s$ and $\mathcal{D}(x,t)= \ds\int_{0}^t d(x,s) \,\mathrm{d}s$ satisfying condition \eqref{B0}, such that $\B \ll \Phi_d$ and $\mathcal{D}\circ\mathcal{B} \ll \Phi_d$, where the symbol "$\ll$" will be precisely defined in Definition \ref{ddffd}. Moreover, we assume that
\begin{align*}
    \lvert f(x,t) \rvert \leq \a\, b(x,\lvert t \rvert) \quad \text{for all } t \in \mathbb{R} \text{ and a.e.\ } x \in \mathbb{R}^d,
\end{align*}
and
\begin{align*}
    \ell^+ \leq b^{-}  \leq \frac{b(x,t)t}{\mathcal{B}(x,t)} \leq b^{+} < m^{-}_\ast \text{ and } 1<d^- \leq \frac{td(x,t)}{\D(x,t)} \leq d^+<+\infty \quad \text{for a.e.\ } x \in \mathbb{R}^d,
\end{align*}
where $0<\hat{a}(\cdot) \in L^1(\RD)\cap L^\infty(\RD)\cap L^{\widetilde{\D}}(\RD).$

\item[\textnormal{(ii)}]
				$\displaystyle{\lim\limits_{t\rightarrow \pm\infty}\frac{F(x,t)}{ \vert t\vert^{\ell^+}}=+\infty}$ uniformly in $x \in \RD$, where $\displaystyle{F(x,t)=\int^t_0 f(x,s)\,\mathrm{d}s}$.

		\item[\textnormal{(iii)}]
				$f(x, t)=o\left(|t|^{m^{-}-1}\right)$ as $|t| \rightarrow 0$ uniformly in $x \in \RD$.
			\item[\textnormal{(iv)}]
				$\tilde{F}(x, t)=\frac{1}{\ell^{+}} f(x, t) t-F(x, t)>0$ for $|t|$ large and there exist constants $\sigma >\frac{d}{m^{-}}$, $\tilde{c}>0$ and $r_0>0$, such that
				\begin{align}\label{ex0}
					|f(x, t)|^\sigma \leq \tilde{c}|t|^{\left(m^{-}-1\right) \sigma} \tilde{F}(x, t) \quad \text{for all }(x, t) \in \RD \times \mathbb{R} \text{ with }|t| \geq r_0.
				\end{align}
		\end{enumerate}
\end{enumerate}

Our main existence result reads as follows.
\begin{theorem} \label{thm:exis} Let \eqref{F}, \eqref{H}, \eqref{HAST}, and \eqref{VV} hold. Furthermore, suppose there exists a constant $\kappa > 0$ such that
\begin{equation}\label{ineq for embed}
 \kappa t^{m^-} \leq  \Phi(x,t)  \quad \text{for all } x \in \mathbb{R}^d \text{ and } t \geq 0.
 \end{equation}
Then, there exists a constant $\lambda_0 > 0$ such that for every $\lambda \in (0, \lambda_0)$, problem \eqref{prb} admits a nontrivial weak solution.
\end{theorem}
 

In the following example, we present several functions that satisfy the assumptions given in \eqref{F}, \eqref{H}, \eqref{HAST} and \eqref{ineq for embed}. Regarding the conditions associated with the generalized Young functions \eqref{H} and \eqref{HAST}, we will examine various illustrative examples of functions that fulfill these criteria in Subsection \ref{exples}. It is important to note that the inequality \eqref{ineq for embed} does not directly affect the validity of our concentration-compactness principles (CCP) theorems. Rather, this inequality is introduced to overcome certain technical obstacles and to facilitate the application of our variational framework.

\begin{example}
  \begin{enumerate}[label=(\roman*)]
    \item To illustrate assumption \eqref{F}, consider the function \( f\colon \mathbb{R}^d \times \mathbb{R} \to \mathbb{R} \) defined by
\[
f(x,t) = \hat{a}(x) \left( |t|^{\ell^+ - 2}t \ln(1 + |t|) + \frac{1}{\ell^+} \frac{|t|^{\ell^+ - 1}t}{1 + |t|} \right),
\]
where \( \hat{a} \in L^1(\mathbb{R}^d) \cap L^\infty(\mathbb{R}^d) \cap L^{\widetilde{\D}}(\mathbb{R}^d) \), and \( \D \) is a generalized Young function as specified in \eqref{F}. The associated primitive function is given by
\[
F(x,t) = \frac{\hat{a}(x)}{\ell^+} |t|^{\ell^+} \ln(1 + |t|).
\]
It is straightforward to verify that the nonlinearity \( f \) satisfies conditions \eqref{F}(i)--(iii). To confirm condition (iv), we compute
\[
\widetilde{F}(x,t) = \hat{a}(x) \left( \frac{1}{\ell^+} t f(x,t) - F(x,t) \right) = \frac{\hat{a}(x)}{(\ell^+)^2} \frac{|t|^{\ell^+ + 1}}{1 + |t|},
\]
which is strictly positive for sufficiently large values of \( |t| \). Moreover, the function \( f \) satisfies the inequality \eqref{ex0} for
\[
\frac{d}{m^-} \leq \sigma \leq \frac{\ell^+}{\ell^+ - m^- + 1}.
\]
\item As examples of functions that satisfy the inequality \eqref{ineq for embed} and conditions  \eqref{H}, \eqref{HAST}, we cite the following:
\begin{itemize}
  \item The double-phase \(N\)-function:
\[
  \mathcal{H}(x,t) = t^p + a(x)t^{q(x)},
  \]
  where \(1< p \leq q(\cdot) < d \), and \( a(\cdot) \in L^1(\mathbb{R}^d) \).
  \item The generalized \(N\)-function:
\[
  \Psi(x,t) = a(x)\ln(e + t) \cdot t^p,
  \]
  where \( p < d \), and \( a(\cdot) \in L^\infty(\mathbb{R}^d) \) with \( a(x) \geq 1 \) for a.e. $x \in \RD$.
\end{itemize}
  \end{enumerate}
\end{example}

The condition (iv) in \eqref{F} plays a crucial role in proving the boundedness of every Cerami sequence (see Step~1 in Lemma~\ref{crmi}). This condition was first introduced by Ding and Lee~\cite{Ding-Lee-2006} for certain scalar Schrödinger equations, and was subsequently adapted by Bahrouni, Missaoui, and Rădulescu~\cite{Bahrouni-Missaoui-Radulescu-2025} for nonlinear Robin problems in Orlicz spaces. Notably, there exist nonlinearities that satisfy condition (iv) but fail to verify the classical Ambrosetti--Rabinowitz condition (cf.~\cite[p. 1277]{mao2009}). The assumption \eqref{ineq for embed} is essential in our approach to proving the existence result. It is closely related to the boundedness of the Cerami sequence, which plays a crucial role in the variational framework. This condition becomes particularly important due to the analytical difficulties encountered when studying problem \eqref{prb} in the whole space \( \mathbb{R}^d \), in contrast to the bounded domain case, where such issues are significantly easier to manage.

Recently, Bahrouni, Bahrouni, and Winkert~\cite{BAHROUNI2025104334} established the existence and multiplicity of solutions for the quasilinear problem
\begin{equation*}
    -\operatorname{div}\left(\phi(x,|\nabla u|)\frac{\nabla u}{|\nabla u|}\right)
    + V(x)\phi(x,|u|)\frac{u}{|u|} = \lambda f(x,u), \quad x \in \mathbb{R}^d,
\end{equation*}
where $\phi(x,|t|) = |t|^{p(x,|t|)-1} + a(x)|t|^{q(x,|t|)-1}$. Their approach, based on the critical point theorem in~\cite[Theorem~2.1 and Remark~2.2]{Bonanno-DAgui-2016}. The key differences between their results and ours stem from the different nature of the operators involved and the presence of critical growth terms involving $\phi^*$ in our equation \eqref{prb}. These differences introduce significant technical challenges, particularly in verifying the Cerami condition (see Lemma~\ref{crmi}), and require a new analytical approach.

\subsection{Plan of the paper}

We organize the paper as follows. In Section \ref{prlmn}, we present the necessary preliminaries. We begin with the definition and structural properties of generalized Young functions and Musielak–Orlicz–Sobolev spaces (Subsection \ref{gnfmos}), followed by a detailed discussion of Sobolev conjugates for generalized Young function and Sobolev embeddings in Musielak–Orlicz–Sobolev spaces  (Subsection \ref{conjsob}), and conclude with the functional setting relevant to our problem (Subsection \ref{funset}). Section \ref{PCC} is devoted to the proofs of the CCPs,  where we also treat some special instances. In Section \ref{APP}, we apply these results to establish the existence of solutions to problem \eqref{prb}.

\section{Preliminaries}\label{prlmn}

This section presents the fundamental concepts and framework for our analysis. We begin by introducing generalized Young functions and Musielak-Orlicz-Sobolev spaces, which provide the functional analytic setting for our work. Next, we discuss the Sobolev conjugate and corresponding embedding results that are crucial for our arguments. Finally, we establish the precise functional setting in which our main results will be developed. These foundational elements will be referenced throughout our subsequent analysis.

\subsection{Generalized Young functions and Musielak-Orlicz-Sobolev spaces}\label{gnfmos}


In this subsection, we explore definitions and properties relevant to the Musielak-Orlicz spaces and Musielak-Orlicz Sobolev spaces. We refer to the paper by Bahrouni--Bahrouni--Missaoui \cite{Bahrouni-Bahrouni-Missaoui-Radulescu-2024} and Bahrouni--Bahrouni--Winkert \cite{BAHROUNI2025104334}, Crespo-Blanco--Gasi\'{n}ski--Harjulehto--Winkert \cite{CrespoBlancoGasinskiHarjulehtoWinkert2022}, Fan \cite{Fan2012} and the monographs by Chlebicka--Gwiazda--\'{S}wierczewska-Gwiazda \cite{ChlebickaGwiazdaSwierczewskaGwiazdaWroblewskaKaminska2021}, Diening--Harjulehto--H\"{a}st\"{o}--R$\mathring{\text{u}}$\v{z}i\v{c}ka \cite{DieningHarjulehtoHastoRuzicka2011}, Harjulehto--H\"{a}st\"{o} \cite{Harjulehto2019}, Musielak \cite{Musielak-1983}.

\begin{definitions}\label{dfnfct}
 \begin{enumerate}
		\item[\textnormal{(1)}] A Young function is a function $\Phi\,:\, [0,\infty) \to [0,\infty)$ which is convex, continuous, non-constant and such that $\Phi(0)=0$.  The following representation formula holds
\begin{align*}
  \Phi(t) = \int_0^t \phi(\tau)\,d\tau,
\end{align*}
where $\phi: [0, \infty) \to [0, \infty]$ is a non-decreasing  function.
\item[\textnormal{(2)}]
A function $\Phi\,:\, \RD \times [0, \infty) \to [0,\infty)$ is called
 a generalized Young function  if
\begin{enumerate}
\item [(i)] the function $\Phi(x,\cdot)$ is a Young function for a.e.  $x \in \RD$,
\item [(ii)]  the function $\Phi(\cdot, t)$ is measurable for every~$t \geq 0$.
\end{enumerate}
\end{enumerate}
\end{definitions}
\begin{definitions}\label{def: conds}
	\begin{enumerate}
		\item[\textnormal{(1)}]
			We say that a generalized Young function $\Phi$ satisfies the $\Delta_2$-condition if there exist $C_0 > 0$ and a nonnegative function $g\in L^1(\O)$ such that
			\begin{align*}
				\Phi(x,2t)\leq C_0\Phi(x,t)+g(x)\quad\text{for a.a.\,}x\in \O \text{ and for all } t\geq 0.
			\end{align*}
		\item[\textnormal{(2)}]
			A generalized Young function $\Phi$ is said to satisfy the condition:
			\begin{enumerate}[label=$(\textnormal{A}0)$,ref=A0]
    \item\label{A0}
			if there exists $\beta \in(0,1]$ such that
					\begin{align*}
						\beta \leq \Phi^{-1}(x, 1) \leq \frac{1}{\beta}
					\end{align*}
					for a.a.\,$x \in \mathbb{R}^d$;
\end{enumerate}
\begin{enumerate}[label=$(\textnormal{A}1)$,ref=A1]
    \item\label{A1}
				if there exists $\beta \in(0,1]$ such that
					\begin{align*}
						\beta \Phi^{-1}(x, t) \leq \Phi^{-1}(y, t)
					\end{align*}
					for every $t \in\left[1, \frac{1}{|B|}\right]$, for a.a.\,$x, y \in B$, and every ball $B \subset \mathbb{R}^d$ with $|B| \leq 1$;\end{enumerate}
				\begin{enumerate}[label=$(\textnormal{A}2)$,ref=A2]
    \item\label{A2}
					if for every $s>0$ there exist $\beta \in(0,1]$ and $h \in L^1\left(\mathbb{R}^d\right) \cap L^{\infty}\left(\mathbb{R}^d\right)$ such that
					\begin{align*}
						\beta \Phi^{-1}(x, t) \leq \Phi^{-1}(y, t)
					\end{align*}
					for a.a.\,$x, y \in \mathbb{R}^d$ and for all $t \in[h(x)+h(y), s]$;\end{enumerate}
\begin{enumerate}[label=$(\textnormal{B}0)$,ref=B0]
    \item\label{B0} if there exist constants $c_1, c_2 > 0$ such that
			\begin{align*}
				c_1 < \Phi(x,1) < c_2 \quad \text{for a.a.\,} x \in \mathbb{R}^d.
			\end{align*}
			\end{enumerate}
	\end{enumerate}
\end{definitions}

The condition (\ref{A0}) is commonly referred to as the weight condition, \eqref{A1} as the local continuity condition, and \eqref{A2} as the decay condition. These conditions are crucial in the functional analysis of Musielak-Orlicz and Musielak-Orlicz-Sobolev spaces; see \cite{Harjulehto2016, Cianchi2024, Harjulehto2019 } and the references therein.
\begin{remark}\label{rem:B0A0}
By invoking \cite[Corollary 3.7.5]{Harjulehto2019}, it is clear that condition \eqref{B0} implies \eqref{A0}.
\end{remark}
Throughout the rest of this subsection, let $\Omega$ denote an open subset of $\mathbb{R}^d$.

\begin{definition}\label{ddffd}
	Let $\Phi_1$ and $\Phi_2$ be two generalized Young functions.
	\begin{enumerate}
		\item[\textnormal{(1)}]
			We say that $\Phi_1$ increases essentially slower than $\Phi_2$ near infinity and we write $\Phi_1 \ll \Phi_2$, if for any $k>0$
			\begin{align*}
				\lim _{t \rightarrow \infty} \frac{\Phi_1(x, k t)}{\Phi_2(x, t)}=0\quad \text {uniformly in}\  x \in \O.
			\end{align*}
		\item[\textnormal{(2)}]
			We say that $\Phi_1 $ is weaker than $\Phi_2$, denoted by $\Phi_1 \preceq \Phi_2$, if there exist two positive constants $C_1, C_2$ and a nonnegative function $g \in L^1(\O)$ such that
			\begin{align*}
				\Phi_1(x, t) \leq C_1 \Phi_2\left(x, C_2 t\right)+g(x)\quad \text{for a.a.\,} x\in \O \text{ and for all }  t\geq 0.
			\end{align*}
	\end{enumerate}
\end{definition}

\begin{definition}
	For any generalized Young function $\Phi$, the function $\widetilde{\Phi}\colon\O\times\mathbb{R}\to \mathbb{R}$ defined by
	\begin{align*}
		\widetilde{\Phi}(x,t):=\sup_{\tau\geq 0}\left( t\tau-\Phi(x,\tau)\right)\quad \text{for all } x\in \O \text{ and for all }  t\geq 0,
	\end{align*}
	is called the complementary function
    of   $\Phi$.

\end{definition}

In view of the definition of the complementary function $\widetilde{\Phi}$, we have the following Young type inequality:
\begin{align}\label{Yi}
	\tau\sigma\leq \Phi(x,\tau)+\widetilde{\Phi}(x,\sigma)\quad \text{for all } x\in \O \text{ and for all }  \tau,\sigma\geq0.
\end{align}

\begin{remark}
	$~$
	\begin{enumerate}
		\item[\textnormal{(1)}]
			Note that the complementary function $\widetilde{\Phi}$ is also a generalized Young function.
		\item[\textnormal{(2)}]
			If there exist $m,\ell\in \mathbb{R}$ such that
			\begin{align}\label{D2}
				1\leq m\leq \frac{\phi(x,t)t}{\Phi(x,t)}\leq \ell\quad \text{for all } x\in \O \text{ and for all }  t> 0,
			\end{align}
			then the generalized Young function $\Phi$ satisfies the $\Delta_2$-condition.
	\end{enumerate}
\end{remark}

The following lemma is taken from Bahrouni--Bahrouni--Missaoui \cite[Lemma 2.3]{Bahrouni-Bahrouni-Missaoui-Radulescu-2024}. Since the proof follows a similar argument (applied there for generalized N-functions), we omit it here.

\begin{lemma}\label{lm1}
	Let $\Phi$ be a generalized Young function. We suppose that $t\mapsto \phi(x,t)$ is continuous and increasing for a.a.\,$x\in \O$. Moreover, we assume that there exist $m,\ell\in \mathbb{R}$ such that
	\begin{align}\label{D22}
		1<m\leq \frac{\phi(x,t)t}{\Phi(x,t)}\leq \ell\quad \text{for all } x\in \O \text{ and for all }  t> 0,
	\end{align}
	then,
	\begin{align*}
		\widetilde{\Phi}(x,\phi(x,s))\leq (\ell-1)\Phi(x,s)\quad \text{for all } s\geq 0 \text{ and for all }  x\in \O,
	\end{align*}
	and
	\begin{align*}
		\frac{\ell}{\ell-1}:=\widetilde{m}\leq \frac{\widetilde{\phi}(x,s)s}{\widetilde{\Phi}(x,s)}\leq \widetilde{\ell}:=\frac{m}{m-1}\quad \text{for all } x\in \O \text{ and for all }  s> 0,
	\end{align*}
	where $\widetilde{\Phi}(x,s)=\ds\int_{0}^s \widetilde{\phi}(x,\upsilon)\,\mathrm{d}\upsilon$.
\end{lemma}

\begin{remark}\label{compl}
\begin{enumerate} [label=(\roman*)]
  \item  	Note that the condition \eqref{D22} implies that $\Phi$ and its complementary function $\widetilde{\Phi}$ satisfy the $\Delta_2$-condition, wich $g=0,$ see \cite{Mih2008}.
  \item     It is clear that inequality~\eqref{mar3} implies inequality~\eqref{D22} 
with $m$ and $\ell$ replaced by $m^-$ and $\ell^+$, respectively. 
Therefore, throughout the rest of this paper, any result proved under inequality~\eqref{D22} 
remains valid under inequality~\eqref{mar3}, 
with the necessary replacements of $m$ and $\ell$ as indicated above.

\end{enumerate}
\end{remark}

\begin{lemma}\label{xit}
  Let $\Phi$ be a generalized Young function fulfilling \eqref{D22}, then the following inequality holds:
$$
			\min \{\l^m, \l^\ell\}\Phi(x,t)\leq  \Phi(x,\l t)\leq \max \{ \l^m, \l^\ell\}\Phi(x,t), \text{ for a.a.}\, x \in \O \text{ and for all }\l, \ t \geq 0.
$$
\end{lemma}

\begin{lemma}\label{l1}
Let \( \Phi \) be a generalized  Young function satisfying \eqref{D22}.
Then, for every \( \eta > 0 \), there exists a constant \( C_\eta > 0 \), independent of \( x \), such that:
\begin{equation}\label{t+s}
\Phi(x, s + t) \leq C_\eta \Phi(x, t) + (1 + \eta)^{\ell} \Phi(x, s), \quad \text{ for all } s, t > 0, \quad \text{and for a.a. } x \in \Omega.
\end{equation}
\end{lemma}
\begin{proof}
 Let $t,\ s \geq0$ and let $\eta>0$. If \(t > \eta s\), then \(t + s \leq t\left(1 + \frac{1}{\eta}\right)\). By the monotonicity of \(\Phi(x, \cdot)\) and the \(\Delta_2\) condition, there exists a constant \(\mathbf{C} > 1\) such that:
\begin{equation}\label{t+s1}
\Phi(x, t + s) \leq \Phi\left(x, t\left(1 + \frac{1}{\eta}\right)\right) \leq \Phi(x, t 2^\kappa) \leq \mathbf{C}^\kappa \Phi(x, t), \text{ for all } x\in \O \text{ and } t\geq0,
\end{equation}
where \(\kappa = \kappa(\eta) \in \mathbb{N}\) is chosen such that \(1 + \frac{1}{\eta} \leq 2^\kappa\).

If \(t \leq \eta s\), then \(s + t \leq s(1 + \eta)\). Therefore, by Lemma \ref{xit}, we have:
\begin{equation}\label{t+s2}
\Phi(x, s + t) \leq \Phi(x, s(1 + \eta)) \leq (1 + \eta)^\ell \Phi(x, s).
\end{equation}
Combining \eqref{t+s1} and \eqref{t+s2}, we conclude \eqref{t+s}.
\end{proof}

Now, we define the two functions $M_\Phi$, $\Phi_{\max}$ and $\Phi_{\min}$, which are crucial for our CCP results.

Throughout the paper we denote
$$
C_+(\overline{\Omega}) := \left\{ g \in C\left(\overline{\Omega}\right) : 1 < g^-:=\inf _{x \in \overline{\Omega}} g(x) \leq g^+:= \sup _{x \in \overline{\Omega}} g(x) < \infty \right\}.
$$
Consider a generalized Young function $\Phi$ satisfying the existence of two functions $m(\cdot), \ell(\cdot) \in C_+(\overline{\Omega})$ such that
\begin{equation}\label{mar3}
    m(x) \leq \frac{\phi(x,t) \, t}{\Phi(x,t)} \leq \ell(x), \quad \text{for all } x \in \Omega \text{ and } t > 0, \tag{$\Phi_0$}
\end{equation}
where $\phi(x,t) = \frac{\partial \Phi(x,t)}{\partial t}$ is the derivative of $\Phi$ with respect to $t$. We denote by $\Phi_{\max}$ and $\Phi_{\min}$ functions associated with $\Phi$, defined as
\begin{equation}\label{R1}
\Phi_{\max}(x,t) := \max \left\{ t^{m(x)}, t^{\ell(x)} \right\},\ \Phi_{\min}(x,t) := \min \left\{ t^{m(x)}, t^{\ell(x)} \right\},\  M_\Phi(x,t):= \limsup_{s\to +\infty} \frac{\Phi(x,st)}{\Phi(x,s)}.
\end{equation}
The functions $\Phi_{\max}$ and $M_\Phi$ satisfie the definition of a generalized Young function, while $\Phi_{\min}$ does not, as it lacks convexity. However, $\Phi_{\min}$ belongs to the class of \emph{weak Young functions} (also known as \emph{weak $\Phi$-functions}). For further details about this function class, we refer to \cite{Harjulehto2019}.

The functions $\Phi$, $\Phi_{\min}$, $\Phi_{\max}$, and $M_{\Phi}$ satisfy the following properties. Since the proof is straightforward from definitions and assumption, we leave the details to the reader.
\begin{remark}\label{R2}
Let \(\Phi\) be a generalized Young function satisfying \eqref{mar3}. Then, we have
\begin{enumerate}
    \item[$(1)$] $\Phi_{\min}(x,s) \Phi(x,t) \leq \Phi(x,st) \leq \Phi_{\max}(x,s) \Phi(x,t) $, for all $x \in \Omega $ and all $ t, s \geq 0 $.
    \item[$(2)$] $ \Phi_{\max} $ and $\Phi_{\min}$  satisfy the \(\Delta_2\)-condition.
    \item[$(3)$] $ \Phi_{\max}(x,st) \leq \Phi_{\max}(x,s) \Phi_{\max}(x,t) $, for all $ x \in \Omega $ and all $ t, s \geq 0 $.
    \item[$(4)$] $ \Phi_{\min}(x,st) \geq \Phi_{\min}(x,s) \Phi_{\min}(x,t) $, for all $ x \in \Omega $ and all $ t, s \geq 0 $.
    \item[$(5)$] $\Phi_{\min} (x,s) M_\Phi(x,t) \leq M_\Phi(x,st) \leq \Phi_{\max} (x,s) M_\Phi(x,t), \ \text{for all } x \in \Omega \text{ and all } t, s \geq 0.$
\end{enumerate}
\end{remark}


Now, we can define the Musielak-Orlicz space. Let $\mu$ be a $\sigma$-finite and complete measure on $\overline{\O}$, and $M(\O)$ denotes the set of all $\mu$-measurable functions $u\colon\O\to\R$. Given a generalized Young function $\Phi$, the Musielak-Orlicz space $L_\mu^\Phi(\Omega)$ is defined as
\begin{align*}
	L_\mu^{\Phi}(\O):=\left\lbrace u\in M(\O)\colon  \rho_{\Phi}(\lambda u)<+\infty \text{ for some}\ \lambda>0\right\rbrace,
\end{align*}
where
\begin{align}\label{Mo}
	\rho_{\Phi}(u):= \int_{\O}\Phi(x, u)\,\mathrm{d}\mu.
\end{align}
The space $L_\mu^{\Phi}(\O)$ is endowed with the Luxemburg norm
\begin{align*}
	\Vert u\Vert_{L_\mu^{\Phi}(\O)}:=\inf\left\lbrace \lambda>0\colon  \rho_{\Phi}\left(\frac{u}{\lambda}\right)\leq 1\right\rbrace.
\end{align*}

\begin{proposition}
	Let $\Phi$ be a generalized Young function that satisfies the $\Delta_2$-condition, then
	\begin{align*}
		L_\mu^{\Phi}(\O)=\left\lbrace  u\in M(\O)\colon  \rho_{\Phi}( u)<+\infty\right\rbrace.
	\end{align*}
\end{proposition}

\begin{proposition} \label{zoo}
	Let $\Phi$ be a generalized Young function fulfilling \eqref{D22}, then the following assertions hold:
	\begin{enumerate}
		\item[\textnormal{(1)}]
            $\|u\|_{L_\mu^{\Phi}(\O)}=\lambda$ if and only if $ \rho_{\Phi}(\frac{u}{\lambda})=1$; for all $  u\in L_\mu^{\Phi}(\O)\setminus \{0\}$.
\item[\textnormal{(2)}]
$\|u\|_{L^\Phi_\mu(\O)}<1$ (resp.\,$>1$, $=1$) if and only if $ \rho_{\Phi}(u)<1$ (resp.\,$>1$, $=1$).
\item[\textnormal{(3)}]
			$\min \left\{\|u\|_{L_\mu^{\Phi}(\O)}^{m},\|u\|_{\lh}^{\ell}\right\} \leq \rho_{\Phi}(u) \leq\max \left\{\|u\|_{L_\mu^{\Phi}(\O)}^{m},\|u\|_{L_\mu^{\Phi}(\O)}^{\ell}\right\}$, for all $  u\in L_\mu^{\Phi}(\O)$.
		\item[\textnormal{(4)}]
			Let $\left\{ u_n\right\}_{n \in \mathbb{N}}\subseteq \lh $ and $  u \in L_\mu^{\Phi}(\O) $, then
			\begin{align*}
				\|u_n-u\|_{L_\mu^{\Phi}(\O) }\to0 \quad \Longleftrightarrow \quad \rho_\Phi (u_n-u) \to 0\quad \text{as } n \rightarrow +\infty.
			\end{align*}
	\end{enumerate}
\end{proposition}

As a consequence of \eqref{Yi}, we have the following result.

\begin{lemma}[H\"older's type inequality]\label{Holder}
	Let  $\Phi$ be a generalized Young function that satisfies \eqref{D22}, then
	\begin{align*}
		\left\vert \int_{\O} uv\,\mathrm{d}\mu \right\vert \leq 2 \Vert u\Vert_{L_\mu^{\Phi}(\O)}\Vert v\Vert_{L^{\widetilde{\Phi}}_\mu(\O)}\quad \text{for all } u\in L_\mu^{\Phi}(\O) \text{ and for all } v\in L^{\widetilde{\Phi}}_\mu(\O).
	\end{align*}
\end{lemma}
We have the following extension of the Brezis-Lieb Lemma to the Musielak-Orlicz spaces $L^{\Phi}_\mu(\Omega)$.
\begin{lemma}\label{L.brezis-lieb}
Let  $\Phi$ be a generalized Young function that satisfies \eqref{D22}, and let $\{u_n\}_{n\in\N}$ be a bounded sequence in $L_\mu^{\Phi}(\Omega)$
and $u_n(x)\to u(x)$ for a.a. $x\in\Omega$. Then $u\in L_\mu^{\Phi}(\Omega)$ and
\begin{equation}\label{LeBL}
	\lim_{n\to\infty}\int_{\Omega} \Big|\Phi(x,|u_n|)
	-\Phi(x,|u_n-u|)-\Phi(x,|u|)\Big|\diff \mu=0.
\end{equation}
\end{lemma}
\begin{proof}
   The proof is similar to \cite[Proposition 3.7]{Salort} and we omit here.
\end{proof}

The subsequent proposition deals with some topological properties of the Musielak-Orlicz space, see Musielak \cite[Theorem 7.7 and Theorem 8.5]{Musielak-1983}.

\begin{proposition}\label{AB}
	Let $\Phi$ be a generalized Young function.
	\begin{enumerate}
		\item[\textnormal{(i)}]
			The space $\left(L_\mu^{\Phi}(\O),\|\cdot\|_{L_\mu^{\Phi}(\O)}\right)$ is a Banach space.
		\item[\textnormal{(ii)}]
			If $\Phi$ satisfies \eqref{D2}, then $L_\mu^{\Phi}(\O)$ is a separable space.
		\item[\textnormal{(iii)}]
			If $\Phi$ satisfies \eqref{D22}, then $L_\mu^{\Phi}(\O)$ is a reflexive and separable space.
	\end{enumerate}
\end{proposition}
Note that when \(\mu\) is the Lebesgue measure, we write $ L^{\Phi}(\Omega) $ and $ \Vert \cdot \Vert_{L^{\Phi}(\Omega)} $ in place of $ L_\mu^{\Phi}(\Omega) $ and $ \Vert \cdot \Vert_{L_\mu^{\Phi}(\Omega)} $, respectively.

Now, we are ready to define the Musielak-Orlicz Sobolev space. Let $\Phi$ be a generalized Young function. The Musielak-Orlicz Sobolev space is defined as follows
\begin{align*}
	W^{1,\Phi}(\O):=\left\lbrace u\in L^{\Phi}(\O)\colon  |\nabla u| \in L^{\Phi}(\O)\right\rbrace.
\end{align*}
The space $W^{1,\Phi}(\O)$ is endowed with the norm
\begin{align*}
	\Vert u\Vert_{W^{1,\Phi}(\O)}:=\Vert u\Vert_{L^{\Phi}(\O)}+\Vert \nabla u\Vert_{L^{\Phi}(\O)}\quad \text{for all } u\in W^{1,\Phi}(\O),
\end{align*}
where $\|\nabla u\|_{L^{\Phi}(\O)} := \|\, |\nabla u|\, \|_{L^{\Phi}(\O)}$.\\
 We denote by
$W^{1,\Phi}_0(\O)$
 the completion of $C^\infty _0(\O)$ in $W^{1,\Phi}(\O)$.

\begin{remark}
	If $\Phi$  satisfies \eqref{D22}, then the spaces $W^{1,\Phi}(\O)$ and $W^{1,\Phi}_0(\O)$  are  reflexive and separable Banach spaces with respect to the norm $\Vert \cdot\Vert_{W^{1,\Phi}(\O)}$.
\end{remark}
\subsection{Sobolev conjugate and Sobolev embedding}\label{conjsob}

In the following paragraph, we focus on Sobolev embedding results in Musielak-Orlicz spaces. Various results on this topic can be found in the literature; see, e.g., \cite{CruzUribe2018, Fan2012, Harjulehto2008, Harjulehto2016, Harjulehto2019, Mizuta2023}. However, we specifically cite the results established in \cite{Cianchi2024}, as the authors proved sharper results. In particular, for each fixed \(x\), their findings coincide with the sharp Sobolev conjugate in classical Orlicz spaces; see \cite{Cianchi1996}. As mentioned in the introduction, due to the complex structure of the Sobolev conjugates defined in~\cite{Cianchi2024}, we introduce a new Sobolev conjugate associated with \( \Phi \) in bounded domains (see \eqref{def=P*}). This formulation enables us to establish several key properties of the \( \Phi \)-associated Sobolev conjugate, which are fundamental to our analysis. In the case of unbounded domains, we work with the Sobolev conjugate $\Phi_d$, defined in~\cite{Cianchi2024}, under the additional assumption given in \eqref{HAST}.

We begin with the following definitions, which can be found in \cite{Cianchi2024}.

\begin{definitions}
\begin{enumerate}
    \item The function $\Phi_{\infty} : [0, \infty) \to [0, \infty]$, associated with a generalized Young function $\Phi$ by
\begin{equation}\label{phiinf}
\Phi_\infty (t) = \limsup_{|x|\to \infty} \Phi(x,t) \qquad \text{for $t \geq 0$,}
\end{equation}
    \item The function  $\overline \Phi : \RD \times [0, \infty) \to [0, \infty]$ is given by
\begin{align}\label{july7}
      \overline \Phi(x,t)= \begin{cases} 2\Phi_0(x,t)-1 \quad &\text{if $t \geq 1$}
\\ \ds\limsup_{|x|\to \infty} \Phi_0(x,t) \quad &\text{if $0\leq t <1$,}
\end{cases} \end{align}for $x\in \RD$, where $\Phi_0  : \RD \times [0, \infty) \to [0, \infty]$ is a generalized Young function given by
\begin{align}
\Phi_0(x,t)=
 \max\big\{\Phi\big(x,\Phi^{-1}(x,1)t\big), 2t-1\big\} \quad \text{for $x\in \RD$ and $t \geq 0$.}
\end{align}
\item
The function $\widehat \Phi$ is defined as
\begin{equation}\label{july11}
      \widehat \Phi(x,t)= \begin{cases} 2\Phi_0(x,t)-1 \quad & \text{if $t \geq 1$}
\\ t \quad &\text{if $0\leq t<1$,}
\end{cases}
\end{equation}
for $x \in \RD$.
\end{enumerate}
\end{definitions}

Noting that, by invoking \cite[Proposition 4.2]{Harjulehto2016}, if $\Phi$ is a generalized Young function satisfying the conditions \eqref{A0}, \eqref{A1}, and \eqref{A2}, then
   \begin{align} \label{mar1}
L^{\Phix}(\RD) = L^{\overline {\Phi}}(\RD),
    \end{align}
up to equivalent norms, where the equivalence constants depending on the constant $\beta$ from the conditions \eqref{A0} and \eqref{A1}, as well as on the function $h$ from condition \eqref{A2}. Moreover, according to \cite[Proposition 2.8]{Cianchi2024}, one has
\begin{align}\label{mar2}
L^{\widehat \Phi }(\Omega) = L^{\Phi }(\Omega),
\end{align}
up to equivalent norms, provided that $|\Omega| < +\infty$.

%
%

Assume that $\Omega$ is an open subset of $\mathbb{R}^d$ and that $\Phi$ is a generalized Young function. The homogeneous Musielak-Orlicz-Sobolev space $V^{1,\Phix}(\Omega)$ is defined as
\begin{equation}\label{june1}
V^{1,\Phix}(\Omega) = \L\{u\in W^{1,1}_{\rm loc}(\Omega): \, |\nabla u|\in  L^{\Phix}(\Omega)\r\}.
\end{equation}
If $\Omega$ is connected and $G$ is a bounded open set with $\overline{G} \subset \Omega$, then the functional
\begin{equation} \label{nov153}
\|u\|_{L^1(G)}+ \|\nabla u\|_{L^\Phix(\Omega)}
\end{equation}
defines a norm on $V^{1,\Phix}(\Omega)$. Moreover, different choices of $G$ lead to equivalent norms.

The subspace of functions that vanish on $\partial \Omega$ is given by
\begin{equation}\label{june1vanish}
V^{1,\Phix}_0(\Omega) = \L\{u\in W^{1,1}_{\rm loc}(\Omega): \, \text{the zero extension of $u$ outside $\Omega$ belongs to $V^{1,\Phix}(\mathbb{R}^d)$} \r\}.
\end{equation}
This space is equipped with the norm
\begin{equation}\label{nov154}
\|\nabla u\|_{L^\Phix(\Omega)}.
\end{equation}

For the case $\Omega = \mathbb{R}^d$, we define the space
\begin{equation}\label{june2}
V^{1,\Phix}_d(\mathbb{R}^d) = \L\{u \in V^{1, \Phix}(\mathbb{R}^d): \, |\{|u|>t\}|<\infty\,\, \text{for every $t>0$}\r\}.
\end{equation}
Intuitively, $V^{1,\Phix}_d(\mathbb{R}^d)$ consists of functions in $V^{1,\Phix}(\mathbb{R}^d)$ that exhibit sufficient decay at infinity to ensure the validity of a homogeneous Sobolev inequality in $\mathbb{R}^d$. The norm given by \eqref{nov154}, with $\Omega = \mathbb{R}^d$, is well-defined on $V^{1,\Phix}_d(\mathbb{R}^d)$.

In contrast to the case of Orlicz spaces, where there exists a unique Sobolev conjugate for every Young function (see \cite{Cianchi1996}), the Musielak-Orlicz space has two distinct Sobolev conjugate Young functions (see \cite{Cianchi2024}). Specifically, the function $\Phi^\ast$ is associated with domains of finite measure, while $\Phi_d$ is defined in $\mathbb{R}^d$. 

\begin{definitions}
  \begin{enumerate}
    \item The Sobolev conjugate  of $\Phi$ is the generalized
 Young function  $\Phi_d$ defined as
\begin{equation}\label{sobconj}
\Phi_d (x,t) = \overline  \Phi (x,G_1^{-1}(x,t)) \qquad \text{for $x\in\RD$ and $t \geq 0$,}
\end{equation}
where  $G_1: \RD \times [0, \infty) \to [0, \infty)$ is the function given by
\begin{equation}\label{Hg}
G_1(x,t) = \bigg(\int_0^t \bigg(\frac \tau{\overline \Phi(x,\tau)}\bigg)^{\frac 1{d-1}}\, d\tau\bigg)^{\frac {1}{d'}} \quad \text{for $x \in \RD$ and $t \geq 0$.}
\end{equation}
Here, $\overline \Phi$ as in \eqref{july7}
and $d'= \frac d{d-1}$, the H\"older conjugate of $d$.
    \item The Sobolev conjugate  of $\Phi$ on  domains with finite measure can thus be defined as the generalized
 Young function  $\Phi_{d, \diamond}$ obeying
\begin{equation}\label{sobconj1}
\Phi_{d, \diamond} (x,t) = \widehat  \Phi (x,G_2^{-1}(x,t)) \qquad \text{for $x\in\RD$ and $t \geq 0$,}
\end{equation}
where  $G_2: \RD \times [0, \infty) \to [0, \infty)$ is the function given by
\begin{equation}\label{Hg1}
G_2(x,t) = \bigg(\int_0^t \bigg(\frac \tau{\widehat \Phi(x,\tau)}\bigg)^{\frac 1{d-1}}\, d\tau\bigg)^{\frac {1}{d'}} \quad \text{for $x \in \RD$ and $t \geq 0$.}
\end{equation}
  \end{enumerate}
\end{definitions}

Now, we present the embedding theorem for Musielak-Sobolev spaces, as established in \cite{Cianchi2024}.
\begin{theorem}\label{thm:main}
Assume that $\Phi $ is a  generalized Young function satisfying the
conditions
\eqref{A0}, \eqref{A1}, \eqref{A2} and the following \begin{equation}\label{conv0}
\int_0 \bigg(\frac t{\Phi^\infty(t)}\bigg)^{\frac 1{d-1}}\, dt <\infty.
\end{equation}
Then, $$V^{1,\Phi}_{d}(\RD)\hookrightarrow L^{\Phi_d}(\RD),$$ and there exists a constant  $c=c(d, \beta, h)$  such that
\begin{equation}\label{main1}
\|u\|_{L^{\Phi_d}(\RD)}\leq c \|\nabla u\|_{L^{\Phi}(\RD)}
\end{equation}
for every $u \in V^{1,\Phi}_{d}(\RD)$. Here,  $\beta$ denotes the constant from the  conditions \eqref{A0} and \eqref{A1}, and $h$ the function from condition \eqref{A2}. In particular, if $\Phi= \overline \Phi$, then the constant $c$ in inequality \eqref{main1} only depends on $d$ and~$\beta$.
\end{theorem}

\begin{theorem} \label{thm:mainzero} Let $\Omega$ be an open set in $\RD$ such that $|\Omega|<\infty$ and let $\Phi $ be a  generalized Young function satisfying the conditions \eqref{A0} and \eqref{A1}.
Then, $$V^{1,\Phi}_{0}(\Omega)\hookrightarrow L^{\Phi_{d,\diamond}}(\Omega),$$ and there exists a constant $c=c(\beta, n, |\Omega|)$  such that
\begin{equation}\label{main1zero}
\|u\|_{L^{\Phi_{d, \diamond}}(\Omega)}\leq c \|\nabla u\|_{L^{\Phi}(\Omega)}
\end{equation}
for every   $u \in V^{1,\Phi}_0(\Omega)$. Here,  $\beta$ denotes the constant appearing in conditions \eqref{A0} and \eqref{A1}.
\end{theorem}

We conclude our discussion about Musielal-Orlicz-Sobolev embeddings with a condition for compactness.
\begin{theorem}  \label{thm:maincompact}  Let $\Omega$ be an open set in $\RD$ such that $|\Omega|<\infty$ and let $\Phi $ be a  generalized Young function satisfying conditions \eqref{A0} and \eqref{A1}. Let $\vartheta$ be a generalized Young function such that $\vartheta\ll\Phi_{d, \diamond}$ and
  \begin{align}\label{nov1700}
      \int_\Omega \vartheta(x,t)\,dx < \infty \qquad \text{for $t>0$}.
  \end{align}
 (i) The embedding
\begin{equation}\label{nov1600}
 V^{1,\Phix}_0(\Omega) \hookrightarrow L^{\vartheta }(\Omega)
\end{equation}
is compact.
\\ (ii) Assume, in addition, that $\Omega$ is a bounded domain with a Lipschitz boundary $\partial \O$. Then, the embedding
\begin{equation}\label{nov1611}
 V^{1,\Phix}(\Omega) \hookrightarrow L^{\vartheta }(\Omega)
\end{equation}
is compact.
\end{theorem}

The following lemma clarifies the relationship between the classical Musielak-Orlicz-Sobolev spaces and the homogeneous Musielak-Orlicz-Sobolev spaces defined in \eqref{june1}--\eqref{june2}.
\begin{lemma}\label{lem:embeddings}
Let \(\Phi\) be a generalized Young  function satisfying \eqref{D22}. Then the following continuous embeddings hold:
  \begin{itemize}
    \item [$(1)$] $W^{1,\Phi}(\RD) \hookrightarrow V^{1,\Phi}_d(\RD).$
    \item [$(2)$] $W^{1,\Phi}(\O) \hookrightarrow V^{1,\Phi}(\O),$ and $W^{1,\Phi}_0(\O) \hookrightarrow V^{1,\Phi}_0(\O),$ where $\O$ be a bounded Lipschitz domain.
  \end{itemize}
\end{lemma}
\begin{proof}
The results follow immediately from the definitions of the respective spaces and the properties of $\Phi$.
\end{proof}
\begin{remark}\label{rem:embe}
It is clear that if we consider a generalized Young function satisfying \eqref{D22}, then, by invoking Lemma \ref{lem:embeddings}, the conclusions of Theorems \ref{thm:main}, \ref{thm:mainzero}, and \ref{thm:maincompact} remain valid with \( W^{1,\Phi}(\mathbb{R}^d) \), \( W^{1,\Phi}(\Omega) \), and \( W^{1,\Phi}_0(\Omega) \) replacing \( V^{1,\Phi}_d(\mathbb{R}^d) \), \( V^{1,\Phi}(\Omega) \), and \( V^{1,\Phi}_0(\Omega) \), respectively.
\end{remark}


Now, we introduce the new Sobolev conjugate associated with $\Phi$ in domains of finite measure, defined directly in terms of the generalized Young function $\Phi$ as follows
\begin{equation}\label{def=P*}
  \Phi ^ \ast (x,t)=\Phi (x, N^{-1} (x,t)) , \text{ for a.a. } x \in \RD \text{ and all } t \geq 0,
\end{equation}
where
\begin{equation}\label{def=N}
  N(x,t)= \L( \int_0^t \L( \tfrac \tau{\Phi (x, \tau)} \r)^{\tfrac 1{d-1}} \diff \tau \r)^{\frac{1}{d'}}, \text{ for } x \in \RD \text{ and } t\geq 0.
\end{equation}

\begin{lemma}\label{lem=conc}
    Under assumption \eqref{H}, the following hold:
    \begin{enumerate}[label=(\roman*)]
        \item The functions $N(x, \cdot)$ and $G_2(x, \cdot)$ are concave and increasing.
        \item The inverse functions $N^{-1}(x, \cdot)$ and $G_2^{-1}(x, \cdot)$ are convex.
    \end{enumerate}
\end{lemma}
\begin{proof} In the proof, our attention is directed toward verifying properties (i) and (ii) for the function $G_2(x, \cdot)$. The analysis for  $N(x, \cdot)$
 follows analogously and is comparatively straightforward.
  \begin{itemize}
    \item [(i)] We can express $G_2$ as $G_2(x,t)=\L( F_x(t) \r)^{\tfrac1{d^\prime}}$, where
    $$
    F_x(t)= \int_0^ t \L( \tfrac \tau{\widehat{\Phi}(x,\tau)} \r)^{\tfrac 1{d-1}}\diff \tau,
    $$ for $x\in \O$ and $t \geq0.$ We aim to show that $ F_x $ is increasing. From \cite[Proposition 2.8]{Cianchi2024}, it is known that the mapping $
\tau \mapsto \frac{\tau}{\widehat{\Phi}(x,\tau)}$
is continuous and  positive for all \( \tau > 0 \). Consequently, the integrand in the definition of $ F_x(t) $ is positive for $ t > 0 $, which implies that $ F_x $ is increasing. Consequently, we deduce that $ G_2(x, \cdot) $ is increasing, as it is the composition of two increasing functions. 
To study the concavity of \( G_2(x, \cdot) \), we begin by proving that \( F_x(\cdot) \) is concave. To this end, we analyze the monotonicity of the function
$$
\tau \overset{h(\tau)}{\longmapsto} \left( \frac{\tau}{\widehat{\Phi}(x, \tau)} \right)^{\tfrac{1}{d-1}},
$$
for \( \tau \geq 1 \).
To simplify the proof, we study the function
\[
\tau \overset{g(\tau)}{\longmapsto} \left( \frac{\widehat{\Phi}(x, \tau)}{\tau} \right), \quad \text{for } \tau \geq 0.
\]
We consider two cases:
\vspace{1em}
\noindent\underline{Case 1:} If \( \tau < 1 \), then \( g(\tau) = 1 \); hence, \( g \) is increasing.
\vspace{1em}
\noindent\underline{Case 2:} If \( \tau \geq 1 \), we rewrite \( g \) as
\[
\frac{\widehat{\Phi}(x, \tau)}{\tau} = 2 \max \left\{
\underbrace{\frac{\Phi\left(x, \Phi^{-1}(x,1)\tau\right)}{\tau}}_{g_1(\tau)},
\underbrace{\frac{2\tau - 1}{\tau}}_{g_2(\tau)}
\right\} - \underbrace{\frac{1}{\tau}}_{g_3(\tau)}.
\]
Invoking assumption \eqref{H}, we obtain
$$
g_1'(\tau) = \frac{\Phi^{-1}(x,1)\tau \, \phi\left(x, \Phi^{-1}(x,1)\tau\right)-\Phi\left(x, \Phi^{-1}(x,1)\tau\right)}{\tau^2} \geq 0,
$$
which implies that \( g_1 \) is increasing for \( \tau > 1 \). Furthermore, it is clear that both \( g_2 \) and \( -g_3 \) are increasing functions. Therefore, \( \max\{g_1(\tau), g_2(\tau)\} \) is increasing, and consequently, \( g(\tau) \) is increasing for \( \tau > 1 \). Invoking the continuity of \( \widehat{\Phi} \), see \cite{Harjulehto2016}, and combining both cases, we conclude that \( g(\tau) \) is increasing for all \( \tau > 0 \). This implies that the function \( \frac{\tau}{g(\tau)} \) is decreasing for all \( \tau > 0 \). Hence, the function \( h(\tau) \) is decreasing for all \( \tau \geq 0 \).
Consequently, it follows that \( F_x(\cdot) \) is concave, since it is the primitive of a continuous and decreasing function.
On the other hand, the function
$
\tau \overset{f(\tau)}{\longmapsto} \tau^{\tfrac{1}{d'}}
$
is clearly non-decreasing and concave. Therefore, the composition
$
G_2(x, \cdot) = f(F_x(\cdot))$
is concave.
    \item [(ii)] The assertion (ii) follows directly by (i).
  \end{itemize}
  This ends the proof.
\end{proof}
The lemmas presented below are instrumental in our analysis within a bounded domain.
\begin{lemma}\label{lem=p=p*}
  Let $\Phi$ be a generalized Young function satisfying the assumption \eqref{H}. Then the function $\Phi^\ast$ given by \eqref{def=P*} is a generalized Young function. Moreover,
  \begin{equation}\label{eq=p=p*}
    L^{\Phi_{d, \diamond}}(\O) =     L^{\Phi^*}(\O) ,
  \end{equation} up to equivalent norms with equivalence constants depending on the constants $\beta$ from the conditions \eqref{A0} and \eqref{A1}.
\end{lemma}
\begin{proof}
It is easy to see that $\Phi^{*}$ is a generalized Young function. We aim to prove that \eqref{eq=p=p*} holds.\\
To begin with, invoking Lemma~\ref{lem=conc}(ii), we know that the functions \( N^{-1}(x,\cdot) \) and \( G_2^{-1} (x,\cdot) \), as defined in \eqref{def=N} and \eqref{Hg1}, respectively, are concave.\\
From \cite{Cianchi2024}, we have the following inequalities:
\begin{equation}\label{oct37}
  \widehat{\Phi}(x,t) \leq \Phi\left(x, \tfrac{4}{\beta} t\right) \quad \text{if } t \geq 1,
\end{equation}
and
\begin{equation}\label{oct39}
  \Phi(x,t) \leq \widehat{\Phi}\left(x, \tfrac{t}{\beta}\right) \quad \text{if } t \geq \beta,
\end{equation}
for almost all $x \in \mathbb{R}^d$. Additionally,
\begin{equation}\label{oct38}
  \widehat{\Phi}(x,t) \leq \widehat{\Phi}(x,1) = 1 \quad \text{if } 0 \leq t < 1,
\end{equation}
and
\begin{equation}\label{oct40}
  \Phi(x,t) \leq \Phi(x,\beta) \leq 1 \quad \text{if } 0 \leq t < \beta,
\end{equation}
where the last inequality follows from property \eqref{A0} of $\Phi$.\\
For $t \geq 1$, \eqref{oct37} implies
\begin{equation}\label{juil25}
  \int_1^t \left( \frac{\tau}{\Phi\left(x, \tfrac{4}{\beta} \tau\right)} \right)^{\frac{1}{d-1}} \diff \tau \leq \int_1^t \left( \frac{\tau}{\widehat{\Phi}(x, \tau)} \right)^{\frac{1}{d-1}} \diff \tau, \quad \text{for a.a. } x \in \mathbb{R}^d.
\end{equation}
For $t \leq 1$, using \eqref{B0} and Lemma~\ref{xit}, we obtain
\begin{equation}\label{juil26}
\begin{aligned}
  \int_0^t \left( \frac{\tau}{\Phi\left(x, \tfrac{4}{\beta} \tau\right)} \right)^{\frac{1}{d-1}} \diff \tau
    &\leq \int_0^1 \left( \frac{\tau}{\Phi\left(x, \tfrac{4}{\beta} \tau\right)} \right)^{\frac{1}{d-1}} \diff \tau \\
    &\leq \int_0^1 \left( \frac{\tau}{\left(\tfrac{4}{\beta}\right)^{m(x)} \tau^{\ell} c_1} \right)^{\frac{1}{d-1}} \diff \tau \\
    &\leq \left(\tfrac{\beta}{4}\right)^{\frac{m^-}{d-1}} c_1^{-\frac{1}{d-1}} \int_0^1 \tau^{\frac{1 - \ell}{d-1}} \diff \tau \\
    &= \left(\tfrac{\beta}{4}\right)^{\frac{m^-}{d-1}} c_1^{-\frac{1}{d-1}} \frac{d-1}{d - \ell} = c_3.
\end{aligned}
\end{equation}
Thus, combining \eqref{juil25} and \eqref{juil26}, we deduce that
\begin{equation}
\begin{aligned}
  \tfrac{\beta}{4} N\left(x, \tfrac{4}{\beta} t\right)
    &= \left( \int_0^t \left( \frac{\tau}{\Phi\left(x, \tfrac{4}{\beta} \tau\right)} \right)^{\frac{1}{d-1}} \diff \tau \right)^{\frac{d-1}{d}} \\
    &\leq \left( c_3 + \int_0^t \left( \frac{\tau}{\widehat{\Phi}(x, \tau)} \right)^{\frac{1}{d-1}} \diff \tau \right)^{\frac{d-1}{d}} \\
    &\leq c_3 + \left( \int_0^t \left( \frac{\tau}{\widehat{\Phi}(x, \tau)} \right)^{\frac{1}{d-1}} \diff \tau \right)^{\frac{d-1}{d}} \\
    &= c_3 + G_2(x,t).
\end{aligned}
\end{equation}
Taking $\xi = G_2(x,t)$ and applying $N^{-1}(x, \cdot)$, we obtain, from Lemma \ref{lem=conc}, that
\begin{equation}\label{juil27}
\begin{aligned}
  \tfrac{4}{\beta} t
    &\leq N^{-1}\left(x, \tfrac{4}{\beta} c_3 + \tfrac{4}{\beta} \xi\right) \\
    &\leq \tfrac{1}{2} N^{-1}\left(x, \tfrac{8}{\beta} c_3\right) + \tfrac{1}{2} N^{-1}\left(x, \tfrac{8}{\beta} \xi\right).
\end{aligned}
\end{equation}
Hence,
\begin{equation}\label{juil28}
  G_2^{-1}(x, \xi) \leq \tfrac{\beta}{8} N^{-1}\left(x, \tfrac{8}{\beta} c_3\right) + \tfrac{\beta}{8} N^{-1}\left(x, \tfrac{8}{\beta} \xi\right).
\end{equation}
It follows, using \eqref{oct37}, \eqref{oct38}, and the convexity of $\widehat{\Phi}(x,\cdot)$, that
\begin{equation}\label{juil29}
\begin{aligned}
  \Phi_{d, \diamond}(x,t) = \widehat{\Phi}\left(x, G_2^{-1}(x,t)\right)
    &\leq \widehat{\Phi}\left(x, \tfrac{\beta}{4} N^{-1}\left(x, \tfrac{8}{\beta} c_3\right)\right) + \widehat{\Phi}\left(x, \tfrac{\beta}{4} N^{-1}\left(x, \tfrac{8}{\beta} t\right)\right) \\
    &\leq \underbrace{\Phi\left(x, N^{-1}\left(x, \tfrac{8}{\beta} t\right)\right)}_{= \Phi^{*}\left(x, \tfrac{8}{\beta} t\right)} + \underbrace{\widehat{\Phi}\left(x, \tfrac{\beta}{4} N^{-1}\left(x, \tfrac{8}{\beta} c_3\right)\right) + 1}_{= h(x)} \\
    &= \Phi^{*}\left(x, \tfrac{8}{\beta} t\right) + h(x),\ \forall t>0.
\end{aligned}
\end{equation}
On the other hand, from \eqref{oct39}, we have
\begin{equation}\label{juil47}
  \int_\beta^t \left( \frac{\tau}{\widehat{\Phi}\left(x, \tfrac{1}{\beta} \tau\right)} \right)^{\frac{1}{d-1}} \diff \tau \leq \int_\beta^t \left( \frac{\tau}{\Phi(x, \tau)} \right)^{\frac{1}{d-1}} \diff \tau, \quad \text{for } x \in \O.
\end{equation}
Additionally,
\begin{equation}\label{juil48}
  \int_0^\beta \left( \frac{\tau}{\widehat{\Phi}\left(x, \tfrac{1}{\beta} \tau\right)} \right)^{\frac{1}{d-1}} \diff \tau = \beta^{\frac{d}{d-1}}.
\end{equation}
Using \eqref{juil47} and \eqref{juil48}, we obtain
\begin{equation}
  \beta G_2\left(x, \tfrac{t}{\beta}\right) \leq N(x,t) + \underbrace{\beta^{\frac{d}{d-1}}}_{= c_\beta} , \text{ for $x\in \O$ and $t \geq0$}.
\end{equation}
Therefore, following similar steps as in \eqref{juil27} and \eqref{juil28}, we find
\begin{equation}\label{juil45}
  N^{-1}(x, \xi) \leq \tfrac{\beta}{2} G_2^{-1}(x, 2 c_\beta) + \tfrac{\beta}{2} G_2^{-1}\left(x, \tfrac{2}{\beta} \xi\right),
\end{equation} for $x\in \O$ and $t \geq0.$
Applying $\Phi(x, \cdot)$ to \eqref{juil45} and using \eqref{oct39} and \eqref{oct40}, we get
\begin{equation}\label{pp2}
\begin{aligned}
  \Phi^{*}(x,t) = \Phi\left(x, N^{-1}(x,t)\right)
    &\leq \Phi\left(x, \tfrac{\beta}{2} G_2^{-1}(x, 2 c_\beta) + \tfrac{\beta}{2} G_2^{-1}\left(x, \tfrac{2}{\beta} t\right)\right) \\
    &\leq \tfrac{1}{2} \Phi\left(x, \beta G_2^{-1}(x, 2 c_\beta)\right) + \tfrac{1}{2} \Phi\left(x, \beta G_2^{-1}\left(x, \tfrac{2}{\beta} t\right)\right) \\
    &\leq \widehat{\Phi}\left(x, G_2^{-1}\left(x, \tfrac{2}{\beta} t\right)\right) + \tfrac{1}{2} \Phi\left(x, \beta G_2^{-1}(x, 2 c_\beta)\right) + 1 \\
    &= \Phi_{d, \diamond}\left(x, \tfrac{2}{\beta} t\right) + h(x).
\end{aligned}
\end{equation}
Finally, \eqref{eq=p=p*} follows from inequalities \eqref{juil29} and \eqref{pp2}, via \cite[Theorem 2.8.1]{DieningHarjulehtoHastoRuzicka2011}. This completes the proof.
\end{proof}

The following theorem is a direct consequence of the precedent lemma.
\begin{theorem}\label{thm=newp}
   Let $\Omega$ be an open set in $\RD$ such that $|\Omega|<\infty$ and let $\Phi $ be a  generalized Young function satisfying assumption \eqref{H} and conditions \eqref{B0} and \eqref{A1}. Let $\vartheta$ be a generalized Young function such that $\vartheta\ll\Phi^*$ and
  \begin{align}\label{nov170}
      \int_\Omega \vartheta(x,t)\,dx < \infty \qquad \text{for $t>0$}.
  \end{align}
 (i) The embedding
\begin{equation}\label{nov160}
 V^{1,\Phix}_0(\Omega) \hookrightarrow L^{\vartheta }(\Omega)
\end{equation}
is compact.
\\ (ii) Assume, in addition, that $\Omega$ is a bounded domain with a Lipschitz boundary $\partial \O$. Then, the embedding
\begin{equation}\label{nov161}
 V^{1,\Phix}(\Omega) \hookrightarrow L^{\vartheta }(\Omega)
\end{equation}
is compact.
\end{theorem}

Now, we require the following Lemmas
\begin{lemma}\label{lem:relconj}
  Let $\Phi $ a generalized Young function satisfing \eqref{D22} with $\ell<d$, and satisfies the conditions \eqref{B0} and \eqref{H}, then the following hold:
  \begin{eqnarray}
    \min\left\lbrace t^{m_{*}}, t^{\ell_{*}} \right\rbrace \Phi^*(x,\xi) \leq \Phi^*(x,t\xi) \leq \max\left\lbrace t^{m_{*}}, t^{\ell_{*}} \right\rbrace \Phi^*(x,\xi), \quad \text{for all } t, \xi \geq 0 ,\text{and } x\in \O ,\label{ineq:h^ast}
  \end{eqnarray} where $r_\ast=\frac{rd}{d-r}$ for $1\leq r<d.$
\end{lemma}
\begin{proof}
By the definition of $N$ (see \eqref{def=N}), for all $t > 0$ and $\xi \geq 0$, we have
	\begin{align*}
		N(x,t\xi) = \left( \int_0^{t\xi} \left( \frac{s}{\Phi(x,s)} \right)^{\frac{1}{d-1}} \,\mathrm{d}s \right)^{\frac{d-1}{d}} = t^{\frac{d-1}{d}} \left( \int_0^\xi \left( \frac{ts}{\Phi(x,ts)} \right)^{\frac{1}{d-1}} \,\mathrm{d}s \right)^{\frac{d-1}{d}}.
	\end{align*}
	Using Lemma \ref{xit}, for all $0 < t \leq 1$ and $\xi \geq 0$, we get
	\begin{align*}
		N(x,t\xi)
		&\leq t^{\frac{d-1}{d}} \left( \int_0^\xi \left( \frac{ts}{t^{\ell} \Phi(x,s)} \right)^{\frac{1}{d-1}} \,\mathrm{d}s \right)^{\frac{d-1}{d}} = t^{\left( \frac{d-1}{d} - \frac{ \ell - 1}{d} \right)} \left( \int_0^\xi \left( \frac{s}{\Phi(x,s)} \right)^{\frac{1}{d-1}} \,\mathrm{d}s \right)^{\frac{d-1}{d}} \\
		& = t^{\frac{d - \ell}{d}} N(x,\xi)
	\end{align*}
	and
	\begin{align*}
		N(x,t\xi)
		&\geq t^{\frac{d-1}{d}} \left( \int_0^\xi \left( \frac{ts}{t^{m} \Phi(x,s)} \right)^{\frac{1}{d-1}} \,\mathrm{d}s \right)^{\frac{d-1}{d}} = t^{\left( \frac{d-1}{d} - \frac{ m - 1}{d} \right)} \left( \int_0^\xi \left( \frac{s}{\Phi(x,s)} \right)^{\frac{1}{d-1}} \,\mathrm{d}s \right)^{\frac{d-1}{d}} \\
		& = t^{\frac{d - m}{d}} N(x,\xi).
	\end{align*}
	Thus,
	\begin{align*}
		t^{\frac{d - m}{d}} N(x,\xi) \leq N(x,t\xi) \leq t^{\frac{d - \ell}{d}} N(x,\xi), \quad \text{for all } 0 \leq t \leq 1 \text{ and all } \xi \geq 0.
	\end{align*}
	Similarly, we have
	\begin{align*}
		t^{\frac{d - \ell}{d}} N(x,\xi) \leq N(x,t\xi) \leq t^{\frac{d - m}{d}} N(x,\xi), \quad \text{for all } t > 1 \text{ and all } \xi \geq 0.
	\end{align*}
	Combining these results, we obtain
	\begin{equation}\label{aml1}
		\zeta_0(t) N(x,\xi) \leq N(x,t\xi) \leq \zeta_1(t) N(x,\xi), \quad \text{for all } t, \xi \geq 0,
	\end{equation}
	where
	\begin{align*}
		\zeta_0(t) = \min\left\lbrace t^{\frac{d - m}{d}}, t^{\frac{d - \ell}{d}} \right\rbrace \quad \text{and} \quad \zeta_1(t) = \max\left\lbrace t^{\frac{d - m}{d}}, t^{\frac{d - \ell}{d}} \right\rbrace.
	\end{align*}
	Therefore, inserting $\tau = N(x,\xi)$ and $\kappa = \zeta_0(t)$ into the inequality \eqref{aml1}, i.e., $\xi = N^{-1}(x,\tau)$ and $t = \zeta_0^{-1}(\kappa)$, we get
	\begin{align*}
		\kappa \tau \leq N(x,\zeta_0^{-1}(\kappa)N^{-1}(x,\tau)).
	\end{align*}
	Since $N^{-1}$ is increasing, we infer that
	\begin{align*}
		N^{-1}(x,\kappa \tau) \leq \zeta_0^{-1}(\kappa) N^{-1}(x,\tau), \quad \text{for all } \kappa, \tau > 0.
	\end{align*}
	Similarly, putting $\tau = N(x,\xi)$ and $\kappa = \zeta_1(t)$ into the inequality \eqref{aml1}, we obtain
	\begin{align*}
		\zeta_1^{-1}(\kappa) N^{-1}(x,\tau) \leq N^{-1}(x,\kappa \tau), \quad \text{for all } \kappa, \tau > 0.
	\end{align*}
	From these results, it follows that
	\begin{align*}
		\min\left\lbrace t^{\frac{d}{d - m}}, t^{\frac{d}{d - \ell}} \right\rbrace N^{-1}(x,\xi) \leq N^{-1}(x,t\xi) \leq \max\left\lbrace t^{\frac{d}{d - m}}, t^{\frac{d}{d - \ell}} \right\rbrace N^{-1}(x,\xi), \quad \text{for all } t, \xi \geq 0.
	\end{align*}
	It follows, from Lemma \ref{xit}, that
	\begin{align*}
		\min\left\lbrace t^{m_{*}}, t^{\ell_{*}} \right\rbrace \Phi^*(x,\xi) \leq \Phi^*(x,t\xi) \leq \max\left\lbrace t^{m_{*}}, t^{\ell_{*}} \right\rbrace \Phi^*(x,\xi), \quad \text{for all } t, \xi \geq 0,
	\end{align*}
	where $\displaystyle{m_*= \frac{dm}{d - m}}$ and $\displaystyle{\ell_* = \frac{d\ell}{d - \ell}}$. Hence, we conclude \eqref{ineq:h^ast}.\\
\end{proof}

\begin{lemma}\label{lem: conjdelta2}
   Let $\Phi $ a generalized Young function satisfing \eqref{D22} with $\ell<d$, and satisfies the conditions \eqref{B0} and \eqref{H}, then the following hold:

     \begin{itemize}
       \item [$(\mathrm{i})$]
       \begin{equation}\label{raar2}
       m_*\leq \frac{\phi^*(x,t)t}{\Phi^*(x,t)}\leq  \ell _\ast, \text{ for all } x\in \O \text{ and } t\geq0,
      \end{equation}where $\Phi^*(x,t)=\ds\int_{0}^{t} \phi^* (x,s)\, \diff s$.
       \item [$(\mathrm{ii})$] The function $\Phi^*$ satisfies the $\Delta_2$-condition and it holds
			\begin{align*}
				\min\left\lbrace \Vert u \Vert_{L^{\Phi^*}(\O)}^{m_*}, \Vert u \Vert_{L^{\Phi^*}(\O)}^{\ell_*} \right\rbrace \leq \int_{\O} \Phi^*(x,u)\,\mathrm{d}x \leq \max\left\lbrace \Vert u \Vert_{L^{\Phi^*}(\O)}^{m_*}, \Vert u \Vert_{L^{\Phi^*}(\O)}^{\ell_*} \right\rbrace,
			\end{align*}
			for all $u \in L^{\Phi^*}(\O)$.
     \end{itemize}

\end{lemma}
\begin{proof}
  (i) From \eqref{ineq:h^ast} in Lemma \ref{lem:relconj}, we have
	\begin{align*}
		t^{m_*} \Phi^*(x,z) \leq \Phi^*(x,tz) \leq t^{\ell*} \Phi^*(x,z) \quad \text{for all } t > 1 \text{ and for all } z>0,
	\end{align*}
	which implies that
	\begin{align*}
		\frac{t^{m_*} - 1^{m_*}}{t - 1} \Phi^*(x,z) \leq \frac{\Phi^*(x,tz) - \Phi^*(x,z)}{t - 1} \leq \frac{t^{\ell_*} - 1^{\ell_*}}{t - 1} \Phi^*(x,z)
	\end{align*}
	for all $t > 1$ and for all $z>0$. Taking the limit as $t \to 1^+$, we deduce that
	\begin{align*}
		m_* \leq \frac{\phi^*(x,z) z}{\Phi^*(x,z)} \leq \ell_* \quad \text{for all } z > 0.
	\end{align*}
	(ii) Assertion (ii) follows directly from (i) and Remark \ref{compl} and Proposition \ref{zoo}. This completes the proof.
\end{proof}

\vspace{8mm}
\subsection{Functional setting}\label{funset}
In the current paragraph, we define and explore the weighted Musielak-Orlicz-Sobolev space $W^{1,\Phi}_V(\mathbb{R}^d) $, which is crucial in Theorems \ref{CCP2} and \ref{CCP3}, as well as in the study of Problem \ref{prb}. In the sequel, let $\Phi$ a generalized Young function satisfynig \eqref{D22},
and let $ V \in L^1_{\text{loc}}(\mathbb{R}^d) $ be a function satisfy $\essinf_{x \in \RD}V(x)>0$. We define the space $ W^{1,\Phi}_V(\mathbb{R}^d) $ as follows:

\begin{align*}
W^{1,\Phi}_V(\RD):=\left\{u\in L^{\Phi}(\RD):~ \rho_V(u)<\infty\right\},
\end{align*}
where the modular $\rho_V$ is defined as
\begin{align*}
\rho_V(u):=\int_{\RD} \Phi\left(x,|\nabla u|\right) \diff x + \int_{\RD} V(x)\Phi\left(x,|u|\right) \diff x\quad\text{for}\ u\in L^\Phi(\RD).
\end{align*}
Then, $W^{1,\Phi}_V(\RD)$ is a normed space with the norm
\begin{align}\label{norm}
\|u\|_{W^{1,\Phi}_V(\RD)}:=\inf\left\{ \tau >0: \rho_V\left(\frac{u}{\tau}\right) \leq 1\right\},
\end{align}
see, e.g. \cite[Theorem 2.1.7]{DieningHarjulehtoHastoRuzicka2011}. As Proposition~\ref{zoo}, on this space we have
\begin{equation}\label{modular-norm.XV}
	 \rho_V\left(\frac{u}{\|u\|_{W^{1,\Phi}_V(\RD)}}\right)=1, \quad \forall u \in W_V^{1,\Phi}(\RD) \setminus \{0\},
\end{equation}
and
\begin{equation}\label{m-n}
\min \left\{\|u\|_{W^{1,\Phi}_V(\RD)}^{m},\|u\|_{W^{1,\Phi}_V(\RD)}^{\ell}\right\}\leq \rho_V(u)\leq \max \left\{\|u\|_{W^{1,\Phi}_V(\RD)}^{m},\|u\|_{W^{1,\Phi}_V(\RD)}^{\ell}\right\}, \quad \forall u \in W_V^{1,\Phi}(\RD).
\end{equation}

Clearly, if \eqref{VV} holds, then $W^{1,\Phi}_V(\RD)$ is a separable reflexive Banach space. Furthermore, it holds that
\begin{equation}\label{Em}
W^{1,\Phi}_V(\RD)\hookrightarrow W^{1,\Phi}(\RD),
\end{equation}
i.e., there exists a constant $C>0$ such that
\begin{equation}\label{infV>0}
\|u\|_{W^{1,\Phi}(\RD)}\leq C\|u\|_{W^{1,\Phi}_V(\RD)},\quad \forall u\in W^{1,\Phi}_V(\RD).
\end{equation}
If $\Phi$ satisfies \eqref{A1}, \eqref{A2}, and \eqref{B0}, then by Theorems~\ref{thm:main} and~\ref{thm:maincompact} together with Remark~\ref{rem:embe}, we conclude that\begin{equation}\label{ala11}
\WV \hookrightarrow L^{\Phi_d}(\mathbb{R}^d) \quad \text{and} \quad \WV \hookrightarrow L_{\text{loc}}^{\mathcal{V}}(\mathbb{R}^d) \text{ compactly},
\end{equation} where $\mathcal{V}$ is a generalized Young function such that $\Phi \preceq \mathcal{V} \ll \Phi_d.$

Finally, as a consequence of \eqref{Em}, we obtain the following result.
\begin{theorem}\label{thm:Injcn}  $\Phi$ satisfies \eqref{A1}, \eqref{A2}, and \eqref{B0} and assume that \eqref{VV} holds.  Then, the following hold
	\begin{enumerate}
		\item[\textnormal{(i)}] For any generalized Young function $\mathcal{V}$  satisfying \eqref{B0},
			\begin{align}\label{cA}
				1< v^- \leq \frac{v(x,t)t}{\mathcal{V}(x,t)} \leq v^+ <+\infty\quad \text{for all } x \in \RD \text{ and for all } t\geq0,
			\end{align}
			with $ \mathcal{V}(x,t)=\ds\int_{0}^{t} v(x,s)\,\mathrm{d}s,$
			\begin{equation}\label{2eq60}
				\mathcal{V} \ll \Phi_d,
			\end{equation}
			where $\ll$ is defined in Definition \ref{ddffd} and
			\begin{equation}\label{mla1b}
				\lim_{|t|\rightarrow 0} \frac{\mathcal{V}(x,t)}{\Phi (x,t)}=0 \quad\text{uniformly in } x \in  \RD,
			\end{equation}
	\end{enumerate}
	the embedding
	\begin{align*}
		\WV \hookrightarrow  L^{\mathcal{V}}(\RD)
	\end{align*}
	is continuous.
\item[\textnormal{(ii)}] For $\hat{a}$, $\B$, and $\D$ as defined in \eqref{F}$(\mathrm{i})$, the embedding
\begin{equation}\label{ala11.1}
    \WV \hookrightarrow L^{\B}_{\hat{a}}(\RD)
\end{equation}
is compact, where the space
$$
L^{\B}_{\hat{a}}(\RD) :=\curly {u:\RD \to \R \text{ measurable }: \int_{\RD} \a \B(x,|u|) \diff x <\infty},
$$ is endowed with the norm
$$
\|u\|_{L^{\B}_{\hat{a}}(\RD)}:= \inf\curly{\l>0, \ \int_{\RD}^{} \a \B\L(x, \frac{|u|}{\l}\r) \ \diff x \leq1}
.$$
\end{theorem}

Before proceeding with the proof of Theorem~\ref{thm:Injcn}, we require the following technical result.
\begin{lemma}\label{lem:technical}
 Let $\B$, and $\D$ as defined in \eqref{F}$(\mathrm{i})$, and let $u\in L^{\D\circ \B}(\RD)$. Then
 $$
 \min \curly{ \|u\|^{b^+}_{L^{\D\circ \B}(\RD)}, \|u\|^{b^-}_{L^{\D\circ \B}(\RD)}} \leq \|\B(\cdot, |u|)\|_{L^{\D}(\RD)}\leq \max \curly{ \|u\|^{b^+}_{L^{\D\circ \B}(\RD)}, \|u\|^{b^-}_{L^{\D\circ \B}(\RD)}}.
 $$
\end{lemma}
\begin{proof}
 Let $u \in L^{\D \circ \B}(\RD)$, we define
$$
\beta = \begin{cases}
b^- & \text{if } \norm{u}_{L^{\D \circ \B}(\RD)} \geq 1 \\
b^+ & \text{if } \norm{u}_{L^{\D \circ \B}(\RD)} \leq 1
\end{cases}
\quad \text{and} \quad
\gamma = \begin{cases}
b^+ & \text{if } \norm{u}_{L^{\D \circ \B}(\RD)} \geq 1 \\
b^- & \text{if } \norm{u}_{L^{\D \circ \B}(\RD)} \leq 1.
\end{cases}
$$
Then, for $\lambda = \norm{u}_{L^{\D \circ \B}(\RD)}$, we derive from Lemma~\ref{xit} that
\begin{equation}\label{kakk}
\begin{aligned}
\rho_\D\left(\frac{\B(x,|u|)}{\lambda^\gamma}\right) &= \int_{\RD} \D\left(x, \frac{\B(x,|u|)}{\lambda^\gamma}\right) \diff x \\
&\leq \int_{\RD} \D\left(x, \B\left(x, \frac{|u|}{\lambda}\right)\right) \diff x = 1,
\end{aligned}
\end{equation}
and
\begin{equation}\label{kakk1}
\begin{aligned}
\rho_\D\left(\frac{\B(x,|u|)}{\lambda^\beta}\right) &= \int_{\RD} \D\left(x, \frac{\B(x,|u|)}{\lambda^\beta}\right) \diff x \\
&\geq \int_{\RD} \D\left(x, \B\left(x, \frac{|u|}{\lambda}\right)\right) \diff x = 1.
\end{aligned}
\end{equation}
Hence, by applying Proposition~\ref{zoo}(2) to \eqref{kakk} and \eqref{kakk1}, we obtain the estimate.
\end{proof}

\begin{proof}[Proof of Theorem \ref{thm:Injcn}]
 \begin{enumerate}
   \item [\textnormal{(i)}]By employing similar arguments to those in the proof of \cite[Theorem 1.2]{BAHROUNI2025104334}, we can easily show that the embedding
\begin{align*}
    \WV \hookrightarrow L^{\mathcal{V}}(\mathbb{R}^d)
\end{align*}
is continuous. Thus, statement (i) follows directly from \eqref{Em}.
   \item [\textnormal{(ii)}]  Let $u_n \rightharpoonup 0$ in $\W$, we claim to prove, for $\epsilon >0$, that
   \begin{equation}\label{ajaj}
     \int_{\RD} \a \B(x,|u_n|)\diff x \leq \epsilon.
   \end{equation}
  To prove this, let $R > 0$. By invoking \eqref{F}(i), Hölder's inequality, and Lemma \ref{lem:technical}, we obtain the following estimates:
\begin{align*}
\int_{B_R^c} \a \B(x,|u_n|) \diff x
&\leq 2 \|\hat{a}\|_{L^{\widetilde{\D}}(B_R^c)} \|\B(\cdot, |u_n|)\|_{L^{\D}(B_R^c)} \\
&\leq 2 \|\hat{a}\|_{L^{\widetilde{\D}}(B_R^c)} \max\left\{ \|u_n\|^{b^-}_{L^{\D\circ\B}(B_R^c)}, \|u_n\|^{b^+}_{L^{\D\circ\B}(B_R^c)}\right\} \\
&\leq 2 \gamma_\D \|\hat{a}\|_{L^{\widetilde{\D}}(B_R^c)} \max\left\{ \|u_n\|^{b^-}_{W^{1,\Phi}(B_R^c)}, \|u_n\|^{b^+}_{W^{1,\Phi}(B_R^c)}\right\} \\
&\leq 2 C \|\hat{a}\|_{L^{\widetilde{\D}}(B_R^c)} \to 0 \quad \text{as } R \to +\infty.
\end{align*}
Consequently, there exists $R_\epsilon > 0$ such that
$$
\int_{B_{R_\epsilon}^c} \a \B(x,|u_n|) \diff x \leq \frac{\epsilon}{2}.
$$
 \end{enumerate}
On the other hand, the compact embedding $W^{1,\Phi}(B_{R_\epsilon}) \hookrightarrow L^\B(B_{R_\epsilon})$ yields
$$
\int_{B_{R_\epsilon}} \a \B(x,|u_n|) \diff x \leq \frac{\epsilon}{2}.
$$
Consequently, we obtain the compact embedding $W^{1,\Phi}(\mathbb{R}^d) \hookrightarrow L^\B_{\hat{a}}(\mathbb{R}^d)$. Therefore, property (ii) follows from \eqref{Em}, which completes the proof.
\end{proof}
\section{ Proofs of the concentration-compactness principles and some special instances}\label{PCC}

In this section, we prove our main results concerning the  CCP in Musielak-Orlicz spaces,  considering both bounded and unbounded domains, see Subsections \ref{Bounded} and \ref{ccprd}. Additionally, we present some existing CCP results that can be derived as simple applications of our findings. Furthermore, we discuss certain unproven CCP results in specific cases of Musielak-Orlicz-Sobolev spaces, see Subsection \ref{exples}.

\subsection{Technical Lemmas}\label{subsec: techL}
Let \(\Phi\) be a generalized Young function satisfying \eqref{H}. We establish several crucial technical lemmas that will play a key role in the proofs of the concentration-compactness principles developed later in the paper. These results provide essential estimates and structural properties in the Musielak-Orlicz setting, particularly useful for addressing the lack of compactness in unbounded domains.

We begin with the Reverse Hölder inequality in Musielak-Orlicz spaces for measures.

\begin{lemma}
\label{lem:reverse_holder_musielak}

Assume that \(\Omega \) is  bounded, and let \(\{u_n\} \subset W^{1,\Phi}_0(\Omega)\) be a sequence such that \(u_n \rightharpoonup 0\) weakly in \(W^{1,\Phi}_0(\Omega)\), \(u_n \to 0\) a.e. in \(\Omega\), and
\begin{equation}\label{laled}
\Phi^\ast(x, |u_n|) \, \diff x \overset{\ast }{\rightharpoonup } \nu, \quad \Phi(x, |\nabla u_n|) \, \diff x \overset{\ast }{\rightharpoonup } \mu,
\end{equation}
 in $\mathbb{M}(\overline{\O})$. Then, for every \(\psi \in C_c^\infty(\Omega)\), the following inequality holds:
\begin{equation}\label{RH}
S_{1} \|\psi\|_{L^{\Phi^\ast_{\min}}_\nu(\O)} \leq \|\psi\|_{L^{\Phi_{\max}}_\mu(\O)},
\end{equation}
where $\Phi_{\max}$ is  defined in \eqref{R1} and $\Phi^*_{\min}(x,t):= \min \curly{t^{m(x)},t^{\ell(x)}}$.
\end{lemma}

\begin{proof}
First, consider \(\psi \in C_c^\infty(\Omega)\). By  \(W^{1,\Phi}_0(\Omega) \hookrightarrow L^{\Phi^{\ast}}(\Omega)\), we have
\begin{equation}\label{VI}
S_1 \|\psi u_n\|_{L^{\Phi^{\ast}}(\O)} \leq \|\nabla (\psi u_n)\|_{L^{\Phi}(\O)}.
\end{equation}
On the one hand, by Remark \ref{R2} and Lemma~\ref{lem:relconj}, we have

$$
\Phi^{\ast}(x, s t) \geq \Phi^{\ast}(x, t) \Phi^\ast_{\min}(x,s), \text{ for all } x\in \O \text{ and } t,\ s >0 .
$$
Then, it follows
\begin{equation}\label{prbrr}
\begin{aligned}
\int_{\O} \Phi^{\ast}\left(x, \frac{|\psi u_n|}{\lambda}\right) \, \diff x &\geq \int_{\O} \Phi^{\ast}_{\min}\left(x, \frac{|\psi| }{\lambda}\right)\Phi^\ast(x, |u_n|) \, \diff x\\
&\geq \liminf_{n\to +\infty}\int_{\O} \Phi^{\ast}_{\min}\left(x, \frac{|\psi| }{\lambda}\right)\Phi^\ast(x, |u_n|) \, \diff x\\
& = \int_{\O} \Phi^{\ast}_{\min}\left(x, \frac{|\psi| }{\lambda}\right) \, \diff \nu,\\
\end{aligned}
\end{equation}
which implies that
\begin{equation}\label{III0}
 \|\psi u_n\|_{L^{\Phi^{\ast}}(\O)} \geq \|\psi\|_{L^{\Phi^\ast_{\min}}_ \nu(\O)}.
\end{equation}
Therefore, passing to the limit inferior, as \(n  \to +\infty \) in \eqref{III0}, we obtain
$$
\liminf_{n \to \infty} \|\psi u_n\|_{L^{\Phi^{\ast}}(\O)} \geq \|\psi\|_{L^{\Phi^\ast_{\min}}_ \nu(\O)}.
$$
On the other hand, one has 
                   \begin{equation}\label{zr}
                     \nabla (\psi u_n) = \psi \nabla u_n + u_n \nabla \psi.
                   \end{equation}
Thus
$$
\|\nabla (\psi u_n)\|_{L^\Phi(\O)} - \|\psi \nabla u_n\|_{L^\Phi(\O)} \leq \|u_n \nabla \psi\|_{L^\Phi(\O)},
$$
and since \(u_n \to 0\) in \(L^{\Phi}(\Omega)\) (by compact embedding for \(\Phi \ll \Phi^{\ast}\)), \(\|u_n \nabla \psi\|_{L^\Phi(\O)} \longrightarrow 0\), as $n \to +\infty$. Thus
\begin{equation}\label{IV}
\limsup_{n \to \infty} \|\nabla (\psi u_n)\|_{L^\Phi(\O)}\leq \limsup_{n \to \infty} \|\psi \nabla u_n\|_{L^\Phi(\O)}.
\end{equation}
Next, by Remark~\ref{R2} and \eqref{laled}, we can show that
$$
\begin{aligned}
\int_\Omega \Phi\left(x, \frac{|\psi \nabla u_n|}{\lambda}\right) \, \diff x &\leq \int_\Omega \Phi_{\max}\L(x,\frac{ |\psi|}{\lambda}\r) \Phi\L(x, |\nabla u_n|\r) \, \diff x\\
& \leq \limsup_{n \to \infty}  \int_\Omega \Phi_{\max}\L(x,\frac{ |\psi|}{\lambda}\r) \Phi\L(x, |\nabla u_n|\r) \, \diff x\\
& =  \int_\Omega \Phi_{\max}\L(x, \frac{|\psi|}{\lambda}\r) \, \diff \mu,
\end{aligned} 
$$
 which proves that
\begin{equation*}
 \|\nabla(\psi  u_n)\|_{L^\Phi(\O)} \leq \|\psi\|_{L^{\Phi_{\max}} _\mu(\O)}.
\end{equation*}
It follows, by passing to the limit superior as $n \to +\infty$, that
\begin{equation}\label{V}
\limsup_{n \to \infty} \|\nabla(\psi  u_n)\|_{L^\Phi(\O)} \leq \|\psi\|_{L^{\Phi_{\max}} _\mu(\O)}.
\end{equation}

 consequently, by exploiting \eqref{III0} and \eqref{V}, we derive from \eqref{VI} that
$$
S_{1} \|\psi\|_{L^{\Phi_{\min}^\ast }_\nu(\O)} \leq \|\psi\|_{L^{\Phi_{\max}}_\mu(\O)},
$$
and  the proof is complete.
\end{proof}

The next lemma establishes the Reverse Hölder inequality in Musielak-Orlicz spaces for measures in \(\mathbb{R}^d\), which is crucial in the proofs of Theorems~\ref{CCP2} and~\ref{CCP3}.
\begin{lemma}
\label{lem:reverse_holder_musielakRD}
Let $\{u_n\} \subset \WV$ be a sequence such that $u_n \rightharpoonup 0$ weakly in $\WV$, $u_n \to 0$ a.e. in $\RD$, and
\begin{equation}\label{laledr}
\Phi_d(x, |u_n|) \, \diff x \overset{\ast }{\rightharpoonup } \bar{\nu}, \quad \Phi(x, |\nabla u_n|)+V(x)\Phi(x, |u_n|) \, \diff x \overset{\ast }{\rightharpoonup } \bar{\mu},
\end{equation}
 in $\mathbb{M}(\RD)$. Then, for every \(\psi \in C_c^\infty(\RD)\), the following inequality holds:
\begin{equation}\label{RHR}
S_{2} \|\psi\|_{L^{\Phi_{d_{\min}}}_{\bar{\nu}}(\RD)} \leq \|\psi\|_{L^{\Phi_{\max}}_{\bar{\mu}}(\RD)},
\end{equation}
where $\Phi_{\max}$ and $\Phi_{d_{\min}}$ are defined in \eqref{R1}.
\end{lemma}
\begin{proof}
  Consider \( \psi \in C_c^\infty(\RD) \). By the continuous embedding \( \WV \hookrightarrow L^{\Phi_d}(\mathbb{R}^d) \), we obtain
\begin{equation}\label{zr3}
S_2\|\psi u_n\|_{L^{\Phi_d}(\mathbb{R}^d)} \leq \|\psi u_n\|_{\WV}.
\end{equation}
\textbf{Claim }: We prove that $$  \|\psi \|_{L^{\Phi_{d_{\min}}}_ {\bar{\nu}}(\RD)} \leq \liminf_{n\to +\infty} \|\psi u_n\|_{L^{\Phi_d}(\mathbb{R}^d)} \text{ and } \limsup_{n\to +\infty}\|\psi u_n\|_{\WV} \leq \|\psi\|_{L^{\Phi_{\max}}_ {\bar{\mu}}(\RD)}.$$
Following the approach in the proof of Lemma~\ref{lem:reverse_holder_musielak}, we can easily show that
\begin{equation}\label{ala1}
  \|\psi \|_{L^{\Phi_{d_{\min}}}_ {\bar{\nu}}(\RD)} \leq \liminf_{n\to +\infty} \|\psi u_n\|_{L^{\Phi_d}(\mathbb{R}^d)}.
\end{equation}
Thus, it remains to show that
\begin{equation}\label{ala2}
\limsup_{n\to +\infty} \|\psi u_n\|_{\WV} \leq \|\psi\|_{L^{\Phi_{\max}}_ {\bar{\mu}}(\RD)}.
\end{equation}

Using \eqref{zr}, we have\begin{equation}\label{zr0}
                                          \|\nabla (u_n \psi)\|_{L^\Phi(\RD)} \leq     \|\psi\nabla (u_n )\|_{L^\Phi(\RD)} +    \|u_n  \nabla (\psi)\|_{L^\Phi(\RD)}. 
                                          \end{equation}
On the one hand, by Remark \ref{R2} and \eqref{laledr}, we have 
\begin{equation}\label{zr1}
\begin{aligned}
&  \int_{\mathbb{R}^d} \Phi\left(x,\frac{|\psi \nabla u_n|}{\lambda}\right)+V(x) \Phi\left(x,\frac{|\psi u_n|}{\lambda}\right) \diff x \\
& \leq  \int_{\mathbb{R}^d} \Phi_{\max} \left(x, \frac{|\psi |}{\lambda}\right) \left[ \Phi\left(x,| \nabla u_n|\right)+V(x) \Phi\left(x,| u_n|\right)\right] \diff x \\
& \leq \limsup_{n \to +\infty} \int_{\mathbb{R}^d} \Phi_{\max} \left(x, \frac{|\psi |}{\lambda}\right) \left[ \Phi\left(x,| \nabla u_n|\right)+V(x) \Phi\left(x,| u_n|\right)\right] \diff x \\
&=  \int_{\mathbb{R}^d} \Phi_{\max} \left(x, \frac{|\psi |}{\lambda}\right) \diff \bar{\mu}.
\end{aligned}
\end{equation}
On the other hand,  fixing \( R>0 \) such that \( \operatorname{supp}(\psi) \subset B_R \), we conclude, from Lemma~\ref{xit}, that
\begin{align}\label{ala3}
	\int_{\mathbb{R}^d} \Phi\left(x,\left|\frac{ u_n\nabla \psi }{\lambda}\right|\right) \diff x \leq \max\left\{\frac{1}{\lambda^{m^-}},\frac{1}{\lambda^{\ell^+}}\right\} \left(1+\|\,|\nabla \psi|\,\|_{L^\infty(\mathbb{R}^d)}^{\ell^+}\right) \int_{B_R} \Phi(x,|u_n|) \diff x.
\end{align}
Invoking the compact embedding \( \WV \hookrightarrow\hookrightarrow L^{\Phi}_{\text{loc}}(\mathbb{R}^d) \) and the fact that $u_n \rightharpoonup 0$ in $\WV$, we obtain
\begin{align*}
	\lim_{ n \to \infty} \int_{B_R} \Phi(x,|u_n|) \diff x =0.
\end{align*}
From this, and \eqref{ala3}, we deduce that
\begin{equation}\label{ala5}
	\lim_{n \to \infty} \int_{\mathbb{R}^d} \Phi\left(x,\left|\frac{ u_n\nabla \psi }{\lambda}\right|\right) \diff x =0.
\end{equation}
Therefore, invoking \eqref{zr0}, \eqref{zr1} and \eqref{ala5}, we infer that
\begin{equation*}
  \limsup_{n \to \infty} \|\psi u_n\|_{\WV} \leq \|\psi \|_{L^{\Phi_{\max}}_{\bar{\mu}}(\RD)},
\end{equation*}
which implies, after passing to the limit superior, that
\begin{equation}\label{ala6}
  \limsup_{n \to \infty} \|\psi u_n\|_{\WV} \leq \|\psi \|_{L^{\Phi_{\max}}_{\bar{\mu}}(\RD)}.
\end{equation}
Thus, the inequality \eqref{RHR} follows from \eqref{zr3},  \eqref{ala1} and \eqref{ala6}.
\end{proof}

\begin{lemma}\label{deltaz}
Let \( \nu \) be a nonnegative, bounded Borel measure on \( \overline{\Omega} \), where \( \Omega \subset \mathbb{R}^d \) is an open set. Let \( A \) and \( B \) be two generalized Young functions such that \( A \ll B \) and that $B$ satisfy the condition \eqref{B0}. Moreover, we assume that there exist positive constants \( b^+<\infty \) and \( b^->1 \) such that
   \begin{equation}\label{u3}
    b^-\leq \frac{b(x,t)t}{B(x,t)}\leq b^+. \tag{$\Delta_2^\prime$}
     \end{equation}

Suppose there exists a constant \( C > 0 \) such that for all \( \phi \in C_c^\infty(\Omega) \),
$$
\|\phi\|_{L^{B}_ \nu(\O)} \leq C \|\phi\|_{L^A_ \nu(\O)}.
$$
 Then there exists \( \delta > 0 \) such that for every Borel set \( U \subset \overline{\Omega} \), either \( \nu(U) = 0 \) or \( \nu(U) \geq \delta \).
\end{lemma}

\begin{proof}
We prove the lemma by contradiction, showing that if \( \nu \) could take arbitrarily small positive values, it would contradict the given inequality and conditions.

    Suppose there exists a sequence of Borel sets \( \{U_n\}_{n \in \mathbb{N}} \subset \overline{\Omega} \) such that
 $$
    \nu(U_n) = \varepsilon_n > 0 \quad \text{and} \quad \varepsilon_n \to 0 \text{ as } n \to \infty.
   $$
    Consider the characteristic function \( \chi_{U_n} \). Since \( \chi_{U_n} \) is Borel measurable and bounded, and \( \nu \) is finite, the inequality \( \|\phi\|_{L^B_ \nu(\O)} \leq C \|\phi\|_{L^A_\nu (\O)} \) extends from \( C_c^\infty(\Omega) \) to \( \chi_{U_n} \) by density, see \cite[Theorem 3.7.15 and Lemma 2.2.6]{Harjulehto2019}
    \begin{equation}\label{Q0}
    \|\chi_{U_n}\|_{L^B_\nu(\O)} \leq C \|\chi_{U_n}\|_{L^A_\nu(\O)}, \ \forall n \in \N.
    \end{equation}
Let \( \|\chi_{U_n}\|_{L^B_\nu(\O)} = \lambda_{n,B} \) and $\|\chi_{U_n}\|_{L^A_ \nu(\O)} = \lambda_{n,A}$. By Proposition \ref{zoo}(i), it is clear that
    $$
    \int_{U_n} B\left(x, \frac{1}{\lambda_{n,B}}\right) \diff \nu = 1 \text{ and } \int_{U_n} A\left(x, \frac{1}{\lambda_{n,A}}\right) \diff \nu = 1,
    $$
    and, by \eqref{Q0}, we have
      \begin{equation}\label{Q1}
    \lambda_{n,B} \leq C \lambda_{n,A}.
    \end{equation}

    Define \( t_n = \frac{1}{\lambda_{n,B}} \), so
    \begin{equation}\label{Q2}
    \int_{U_n} B(x, t_n) \diff \nu = 1, \quad \text{and} \quad \|\chi_{U_n}\|_{L^B_ \nu(\O)} = \frac{1}{t_n}.
    \end{equation}
First, we claim that \( t_n \to \infty \) as \( n \to +\infty\). To prove this, assume, for contradiction, that \( t_n \) is bounded, i.e., there exists \( T < \infty \) such that \( t_n \leq T \) for all \( n \).
By condition \eqref{B0}, assumption \eqref{u3}, and Proposition \ref{zoo}, we have
$$
B(x, t_n) \leq B(x, T) \leq B_\infty(T) B(x,1) \leq c_2 B_\infty(T), \quad \text{for all } x \in \Omega,
$$
where $B_\infty(t) := \max\{t^{b^+}, t^{b^-}\}$. Consequently,
$$
\int_{U_n} B(x, t_n) \, d\nu \leq c_2 B_\infty(T) \nu(U_n) = c_3 \varepsilon_n \to 0 \quad \text{as } n \to \infty.
$$
This contradicts \eqref{Q2}, and therefore, $t_n \to \infty$.
    For fixed \( K > 0 \), compute
    \begin{equation}\label{A-integral}
    \int_{U_n} A(x, K t_n) \diff \nu = \int_{U_n} \frac{A(x, K t_n)}{B(x, t_n)} B(x, t_n) \diff \nu.
    \end{equation}
    Since \( t_n \to \infty \) and \( A \ll B \), it follows that
$$
    \frac{A(x, K t_n)}{B(x, t_n)} \to 0 \quad \text{uniformly in } x\in \O,
$$
    so there exists \( \eta_n \to 0 \) such that
    \begin{equation}\label{eta-bound}
    \frac{A(x, K t_n)}{B(x, t_n)} \leq \eta_n, \quad \forall x \in \overline{\Omega}, \text{ for large } n.
    \end{equation}
    Substituting \eqref{eta-bound} into \eqref{A-integral}, and taking into account \eqref{Q2}, we obtain    $$
    \int_{U_n} A(x, K t_n) \diff \nu \leq \eta_n \int_{U_n} B(x, t_n) \diff \nu = \eta_n<1,
   $$ for $n$ sufficiently large.
Therefore,    $$
    \int_{U_n} A\left(x, \frac{1}{\frac{1}{Kt_n}}\right) \diff \nu = \int_{U_n} A(x, K t_n) \diff \nu < 1,
   $$
which implies, by the definition of the Luxemburg norm, that
    $$
    \|\chi_{U_n}\|_{L^A_\nu(\O)} < \frac{1}{K t_n} = \frac{1}{K} \|\chi_{U_n}\|_{L^B_ \nu(\O)}.
    $$
This fact, combined with \eqref{Q0}, implies that
$$
    \frac{1}{t_n} \leq C \cdot \frac{1}{K t_n}.
   $$
    For \( K > C \), we get
$$
    1 \leq \frac{C}{K} < 1,
$$
    which is a contradiction. This completes the proof.

\end{proof}
\begin{remark}
    It is worth noting that the conclusion of Lemma 3.3 continues to hold when extended to the whole space $\RD$.
\end{remark}
The next lemma is exactly as the end of \cite[Lemma I.2]{LIONS1985}.
\begin{lemma}\label{Lema 2}
Let $\nu$ be a non-negative bounded Borel measure on $\overline{\Omega}$. Assume that there exists $\delta>0$ such that for every Borel set $U$ we have that, $\nu(U)=0$ or $\nu(U)\geq\delta$. Then, there exist a countable index set $I$, points $\{x_i\}_{i\in I}\subset \bar\Omega$ and scalars $\{\nu_i\}_{i\in I}\in (0,\infty)$ such that
$$
\nu = \sum_{i\in I} \nu_i\delta_{x_i}.
$$
\end{lemma}
In the following lemma, we estimate the norm of  of the characteristic function.

\begin{lemma}\label{normindic}
  Assume that $\Phi$ satisfies \eqref{B0} and \eqref{mar3}, and let $\lambda$ be a non-negative bounded Borel measure on $\overline{\Omega}$. Then,
 $$ \min \left\{ \lambda(U)^{\frac{1}{m^-}}, \lambda(U)^{\frac{1}{\ell^+}} \right\} \leq \|\chi_U\|_{L^\Phi_ \lambda(\O)} \leq \max \left\{ \lambda(U)^{\frac{1}{m^-}}, \lambda(U)^{\frac{1}{\ell^+}} \right\},$$ for every a Borel set  $U \subset \overline{\Omega}$.
\end{lemma}
\begin{proof}
The proof follows similarly to that of \cite[Lemma 2.18]{BAHROUNI2025104334}.
\end{proof}

\begin{lemma}\label{lemmef}

Under the same assumptions of Lemma \ref{lem:reverse_holder_musielak}, there exist a countable index set \(I\), points \(\{x_i\}_{i \in I} \subset \overline{\Omega}\), and scalars \(\{\nu_i\}_{i \in I} \subset (0, \infty)\) such that
$$
\nu = \sum_{i \in I} \nu_i \delta_{x_i}.
$$
\end{lemma}

\begin{proof}
We start the proof by the following claim:

     \textbf{Claim:} $\nu$ is absolutely continuous with respect to $\mu$.\\
    First, observe that  \(\Phi_{\max}(x, 1) = \Phi^\ast_{\min}(x, 1) = 1\) for all \(x \in \overline{\Omega}\), wich imply that $\Phi_{\max}$ and $\Phi^\ast_{\min}$ satisfy \eqref{B0}.\\
   It follows, invoking Lemma \ref{normindic}, that
    $$
    \begin{aligned}
    &\min \left\{ \mu(U)^{\frac{1}{m^-}}, \mu(U)^{\frac{1}{\ell^+}} \right\} \leq \|\chi_U\|_{L^{\Phi_{\max}} _\mu(\O)} \leq \max \left\{ \mu(U)^{\frac{1}{m^-}}, \mu(U)^{\frac{1}{\ell^+}} \right\}, \\
    &\min \left\{ \nu(U)^{\frac{1}{m^-_*}}, \nu(U)^{\frac{1}{\ell_*^+}} \right\} \leq \|\chi_U\|_{L^{\Phi^\ast_{\min}}_ \nu(\O)} \leq \max \left\{ \nu(U)^{\frac{1}{m^-_*}}, \nu(U)^{\frac{1}{\ell_*^+}} \right\},
    \end{aligned}
    $$
    where \(m_*^- = \frac{m^-d}{d - 1}\) and \(\ell_*^+ = \frac{d\ell^+}{d - \ell^+}\) are the Sobolev conjugate exponents, since \(\ell^+ < d\). Applying the reverse Hölder inequality \eqref{RH}, we get
    $$
    S_1 \min \left\{ \nu(U)^{\frac{1}{m^-_*}}, \nu(U)^{\frac{1}{\ell_*^+}} \right\} \leq \max \left\{ \mu(U)^{\frac{1}{m^-}}, \mu(U)^{\frac{1}{\ell^+}} \right\}.
    $$
    Consequently,   if \(\mu(U) = 0\),  the right-hand side is zero, forcing \(\nu(U) = 0\) since \(S_1 > 0\). Thus, $\nu$ is absolutely continuous with respect to $\mu$. This proves the {\bf Claim.}

    By the Radon-Nikodym Theorem, there exists \(f \in L^1_\mu(\Omega)\), \(f \geq 0\), such that \(\nu = f \cdot \mu\). From the inequality above, \(f\) is bounded: if \(\nu(U) > 0\), then \(\mu(U) > 0\), and the growth ensures \(f \in L^\infty_\mu(\Omega)\). The Lebesgue decomposition of \(\mu\) with respect to \(\nu\) gives
    $$
    \mu = g \cdot \nu + \sigma,
    $$
    where \(g \in L^1_\nu(\Omega)\), \(g \geq 0\), and \(\sigma\) is a bounded positive measure singular with respect to \(\nu\).

   Next, consider the test function
  $ \varphi(g) \chi_{\{g \leq n\}} \psi,
    $
    where \(\psi \in C_c^\infty(\Omega)\) and
    $$
    \varphi(t) = \begin{cases}
    t^{\frac{1}{\ell_*^+ - m^-}} & \text{if } t < 1, \\
    t^{\frac{1}{m^-_*-\ell^+}} & \text{if } t \geq 1.
    \end{cases}
    $$
    Applying \eqref{RH}, we get
      \begin{equation}\label{Z4}
    S_1\|\varphi(g) \psi \chi_{\{g \leq n\}}\|_{L^{\Phi^\ast_{\min}}_ \nu(\O)} \leq \|\varphi(g) \psi \chi_{\{g \leq n\}}\|_{L^{\Phi_{\max}}_ \mu(\O)}.
    \end{equation}
    Since \(\mu = g \cdot \nu + \sigma\) and \(\sigma \perp \nu\), then
    \begin{equation}\label{ineqmp}
    \begin{aligned}
    \|\varphi(g) \psi \chi_{\{g \leq n\}}\|_{L^{\Phi_{\max}}_ \mu(\O)} &\leq \|\varphi(g) \psi \chi_{\{g \leq n\}}\|_{L^{\Phi_{\max}}_ {g \cdot \nu}(\O)} + \|\varphi(g) \psi \chi_{\{g \leq n\}}\|_{L^{\Phi_{\max}}_ \sigma(\O)}\\& =\|\varphi(g) \psi \chi_{\{g \leq n\}}\|_{L^{\Phi_{\max}}_ {g \cdot \nu}(\O)}.\end{aligned}
    \end{equation}
    Since \(\Phi_{\max}(x,st) \leq \Phi_{\max}(x,s) \Phi_{\max}(x,t)\), (see Remark \ref{R2}), and \(\Phi_{\max}(1) = 1\), we infer that
    \begin{equation}\label{Z1}
    \int_\Omega \Phi_{\max}\left( x,\frac{\varphi(g) \psi \chi_{\{g \leq n\}}}{\lambda} \right) g \, \diff \nu \leq \int_\Omega \Phi_{\max}\left(x, \frac{\psi}{\lambda} \right) \Phi_{\max}(x,\varphi(g)) g \chi_{\{g \leq n\}} \, \diff \nu.
    \end{equation}
 For the left-hand side of \eqref{Z4}, we start by invoking Proposition \ref{zoo}, we deduce that
   \begin{equation}\label{Z2}
    \int_\Omega \Phi^\ast_{\min}\left( x, \frac{\varphi(g) \psi \chi_{\{g \leq n\}}}{\lambda} \right) \, \diff \nu \geq \int_\Omega \Phi^\ast_{\min}\left( x, \frac{\psi}{\lambda} \right) \min\{ \varphi(g)^{\ell_*^+}, \varphi(g)^{m^-_*} \} \chi_{\{g \leq n\}} \, \diff \nu.
    \end{equation}
 Moreover, it is clear that
    $$
    \max\{ \varphi(t)^{\ell^+}, \varphi(t)^{m^-} \} t = \min\{ \varphi(t)^{\ell_*^+}, \varphi(t)^{m^-_*} \},
    $$
    it follows that \eqref{Z2} becomes
      \begin{equation}\label{Z3}
      \begin{aligned}
    \int_\Omega \Phi^\ast_{\min}\left( x, \frac{\varphi(g) \psi \chi_{\{g \leq n\}}}{\lambda} \right) \, \diff \nu &\geq \int_\Omega \Phi^\ast_{\min}\left( x, \frac{\psi}{\lambda} \right) \max\{ \varphi(g)^{\ell^+}, \varphi(g)^{m^-} \}g \chi_{\{g \leq n\}} \, \diff \nu\\
    &\geq  \int_\Omega \Phi^\ast_{\min}\left( x, \frac{\psi}{\lambda} \right) \Phi_{\max}(x,\varphi(g)) g \chi_{\{g \leq n\}} \, \diff \nu.
    \end{aligned}
    \end{equation}

    Hence, if we denote \( \nu_n = \Phi_{\max}(x,\varphi(g(x))) g(x) \chi_{\{g \leq n\}} \, \diff \nu\), exploiting \eqref{ineqmp}--\eqref{Z3}, we deduce from \eqref{Z4} that the following reserve Hölder inequality holds
      \begin{equation}\label{Z5}
     S_1  \int_\Omega \Phi^\ast_{\min}\left( x, \frac{\psi}{\lambda} \right) \diff \nu_n \leq  \int_\Omega \Phi_{\max}\left( x, \frac{\psi}{\lambda} \right) \diff \nu_n,
      \end{equation}
    which implies that
    $$
    S_1\|\psi\|_{L^{\Phi^\ast_{\min}}_ {\nu_n}(\O)} \leq \|\psi\|_{L^{\Phi_{\max}}_{\nu_n}(\O)}.
    $$

Finally, since \(\Phi_{\max} \ll \Phi^\ast_{\min}\), it follows from Lemmas \ref{deltaz} and \ref{Lema 2} that there exists \(\delta > 0\) such that \(\nu_n(U) = 0\) or \(\nu_n(U) \geq \delta\), so \(\nu_n =\ds \sum_{i \in I^k} \nu_i^n \delta_{x_i^n}\). As \(n \to \infty\), \(\nu_n \nearrow h(x) \, \diff \nu\), and since \(\nu(\overline{\Omega}) < \infty\),
    $$
    \nu = \sum_{i \in I} \nu_i \delta_{x_i},
    $$
    where \(I = \bigcup_k I^k\) is countable and \(\nu_i = \nu(\{x_i\}) > 0\).

\end{proof}
\begin{remark}\label{rem: lem4}
  In unbounded domains, such as $\mathbb{R}^d$, the above lemma still holds. It suffices to take the test function $\varphi$ in the proof as follows
  $$
    \varphi(t) = \begin{cases}
    t^{\frac{1}{\delta^+ - m^-}} & \text{if } t < 1, \\
    t^{\frac{1}{\kappa^- - \ell^+}} & \text{if } t \geq 1.
    \end{cases}
  $$
\end{remark}

\subsection{The concentration-compactness principle in bounded domain}\label{Bounded}
In this subsection, we prove Theorems~\ref{CCP1} and~\ref{CCP10}. Furthermore, we establish a crucial compactness result concerning strong convergence, stated in Lemma~\ref{lem:strong_convergence}.

\begin{proof}[\rm\bf{Proof of Theorem \ref{CCP1}}]

First, we write \(v_n = u_n - u\). Then, we can apply Lemma \ref{lemmef} to conclude that
\begin{equation}\label{u8}
\Phi^\ast(\cdot,|v_n|) \diff x \overset{\ast }{\rightharpoonup } \diff \bar{\nu} = \sum_{i \in I} \nu_i \delta_{x_i},
\end{equation}
 in the sense of measures.

Now, we use Lemma \ref{L.brezis-lieb}, to obtain
\begin{equation}\label{u9}
\lim_{n \to \infty} \left( \int_\Omega \psi \Phi^\ast(x,|u_n|) \diff x- \int_\Omega \psi \Phi^\ast(x,|v_n|) \diff x \right) = \int_\Omega \psi \Phi^\ast(x,|u|) \diff x,
\end{equation}
for any \(\psi \in C_c^\infty(\Omega)\). It follows, in combinining with \eqref{u8}, that
$$
\Phi^\ast(x,|u_n|) \diff x \overset{\ast }{\rightharpoonup } \diff \nu = \Phi^\ast(x,|u|) \diff x + \diff \bar{\nu},
$$
which proves \eqref{T.ccp.form.nu}.\\
It remains to analyze the measure \(\mu\) and to estimate the weights \(\nu_i\) and \(\mu_i\).
To this end, we consider again \(v_n = u_n - u\) and denote by \(\bar{\mu}\) the weak* limit of \(\Phi(\cdot, |\nabla v_n|) \diff x\) as \(n \to \infty\).

Let \(\psi \in C_c^\infty(\mathbb{R}^d)\) be such that \(0 \leq \psi \leq 1\), \(\psi(0) = 1\), and \(\text{supp}(\psi) \subset B_1(0)\). Now, for each \(i \in I\) and \(\varepsilon > 0\), we denote \(\psi_{\varepsilon,i}(x) := \psi((x - x_i)/\varepsilon)\).
Recall the reverse Hölder inequality, see Lemma \ref{lem:reverse_holder_musielak}, for the measures \(\bar{\nu}\) and \(\bar{\mu}\)
\begin{equation}\label{Y0}
 S_1 \|\psi_{\varepsilon,i}\|_{L^{\Phi^\ast_{\min}}_{\bar{\nu}}(\O)} \leq \|\psi_{\varepsilon,i}\|_{L^{\Phi_{\max}}_ {\bar{\mu}}(\O)}.
\end{equation}

On one hand, let $\lambda >0$, \eqref{u8} gives
\begin{equation}\label{u10}
\int_\Omega \Phi^\ast_{\min}\left(x, \frac{|\psi_{\epsilon,i}(x)|}{\lambda}\right) \diff \bar{\nu} = \Phi^\ast_{\min}\left(x_i, \frac{1}{\lambda}\right) \nu_i.
\end{equation}
The condition for the Luxemburg norm \( \|\psi_{\epsilon,i}\|_{L^{\Phi^\ast_{\min}}_{\bar{\nu}}(\O)} \) is then given by
$$
\Phi^\ast_{\min}\left(x_i, \frac{1}{\lambda}\right) \nu_i \leq 1.
$$
Solving for \( \lambda \), we obtain
$$
\lambda \geq \frac{1}{\L(\Phi^\ast_{\min}\r)^{-1}\left(x_i, \frac{1}{\nu_i}\right)}.
$$
Consequently, the norm satisfies
\begin{equation}\label{Y1}
\|\psi_{\epsilon,i}\|_{L^{\Phi^\ast_{\min}}_ {\bar{\nu}}(\O)} = \frac{1}{\L(\Phi^\ast_{\min}\r)^{-1}\left(x_i, \frac{1}{\nu_i}\right)}.
\end{equation}

On the other hand, by Proposition~\ref{zoo}, we obtain
\begin{equation}\label{Y26}
\begin{aligned}
\|\psi_{\varepsilon,i}\|_{L^{\Phi_{\max}}_ {\bar{\mu}}(\O)}& \leq \max \curly{\L(\rho_{\varepsilon}\r)^{\frac{1}{m^-_{B_\varepsilon}}}, \L(\rho_{\varepsilon}\r)^{\frac{1}{\ell^+_{B_\varepsilon}}}}\\
& \leq \max \curly{\L(\bar{\mu}(B_\varepsilon(x_i)\r)^{\frac{1}{m^-_{B_\varepsilon}}}, \L(\bar{\mu}(B_\varepsilon(x_i)\r)^{\frac{1}{\ell^+_{B_\varepsilon}}}}
\end{aligned}\end{equation} where
$$
\rho_{\varepsilon} := \int_{B_\varepsilon(x_i)} \Phi_{\max}(x, |\psi_{\varepsilon,i}|) \diff \bar{\mu}, m^-_{B_\varepsilon}:=\inf_{x\in B_\varepsilon(x_i) }m(x), \ell^+_{B_\varepsilon}:=\sup_{x\in B_\varepsilon(x_i) }\ell (x),\   \text{and }  \bar{\mu}_i := \lim_{\varepsilon \to 0} \bar{\mu}(B_\varepsilon(x_i)).
$$
It follows, from \eqref{Y26}, that
\begin{equation}\label{Y2}
  \liminf_{\varepsilon \to 0} \|\psi_{\varepsilon,i}\|_{L^{\Phi_{\max}}_ {\bar{\mu}}(\O)} \leq \max \curly{ \L(\bar{\mu}_i )\r)^{\frac{1}{m(x_i)}}, \L(\bar{\mu}_i )\r)^{\frac{1}{\ell(x_i)}}}= \frac{1}{\Phi_{\max}^{-1}\left(x_i, \frac{1}{\bar{\mu}_i} \right)}..
\end{equation}
Consequently, from \eqref{Y1} and \eqref{Y2}, we arrive, via \eqref{Y0}, at
\begin{equation}\label{Y4}
 \quad S_1 \frac{1}{\L(\Phi^\ast_{\min}\r)^{-1}\left(x_i, \frac{1}{\nu_i} \right) } \leq \frac{1}{\Phi_{\max}^{-1}\left(x_i, \frac{1}{\bar{\mu}_i} \right)} .
\end{equation}
Moreover, we conclude that the points \(\{x_i\}\) are atoms of the measure \(\bar{\mu}\). By decomposing \(\bar{\mu}\) into its atomic and non-atomic parts, it follows that
$$
\bar{\mu} \geq \sum_{i \in I} \bar{\mu}_i \delta_{x_i}.
$$
On the other hand, using Lemma \ref{l1}, we have that for any \(\delta > 0\), there exists a constant \(C_\delta\) such that
$$
\Phi(x,|\nabla v_n|) \leq (1 + \delta)^\ell \Phi(x,|\nabla u_n|) + C_\delta \Phi(x,|\nabla u|).
$$
From the last inequality, letting \( n \to +\infty \) and taking into account \eqref{mu_n to mu*}, we obtain
$$
d\bar{\mu} \leq (1 + \delta)^\ell \diff \mu + C_\delta \Phi(x,|\nabla u|) \diff x,
$$
from where it follows that
\begin{equation}\label{Y5}
\bar{\mu}_i \leq (1 + \delta)^\ell\mu_i, \quad \text{where } \mu_i := \mu(\{x_i\}).
\end{equation}
This shows that \begin{equation}\label{Y7}\mu \geq \ds\sum_{i \in I} \mu_i \delta_{x_i} = {\mu_1}.\end{equation} Thus, from \eqref{Y5} and the monotonocity of $\Phi_{\max}$,  we derive, from \eqref{Y4}, that
\begin{equation}\label{Y6}
 \quad S_1 \frac{1}{\L(\Phi^\ast_{\min}\r)^{-1}\left(x_i, \frac{1}{\nu_i} \right) } \leq \frac{1}{\Phi_{\max}^{-1}\left(x_i, \frac{1}{(1+\delta)^\ell\mu_i} \right)} .
\end{equation}
Hence, \eqref{T.ccp.nu_mu} follows since \(\delta > 0\) was taken arbitrarily.

To conclude the proof, it remains to show that
$$
\diff \mu \geq \Phi(x, |\nabla u|) \diff x + \diff \mu_1,
$$
where $\mu_1$ is the singular part of $\mu$ concentrated at the concentration points $\{x_i\}$.\\
The weak convergence $u_n \rightharpoonup u$ in $W_0^{1,\Phi}(\Omega)$ implies that $\nabla u_n \rightharpoonup \nabla u$ weakly in $L^\Phi(U)$ for every subset $U \subset \Omega$. Since the modular $\rho_\Phi(v) = \ds\int_\Omega \psi  \Phi(x, |v|) \diff x$ is convex and continuous, for any nonnegative test function $\psi \in C_c^\infty(\Omega)$ with $\psi \geq 0$, it follows from \cite[Corollary 3.9]{Brez} that $\rho_\Phi$ is weakly lower semicontinuous. Consequently, we have
$$
\int_{\Omega} \psi \Phi(x, |\nabla u|) \diff x \leq \liminf_{n \to +\infty} \int_{\Omega} \psi \Phi(x, |\nabla u_n|) \diff x = \int_{\Omega} \psi \diff \mu.
$$
It follows that
\begin{equation}\label{Y8}
\diff \mu \geq \Phi(x, |\nabla u|) \diff x.
\end{equation}
Finally, combining \eqref{Y7} and \eqref{Y8}, we derive \eqref{T.ccp.form.mu}, as $\mu_1$ and the Lebesgue measure are mutually orthogonal. This completes the proof.
\end{proof}

\begin{proof}[\rm\bf{Proof of Theorem \ref{CCP10}}]
The proof follows a similar approach to Theorem~\ref{CCP1}, with the key difference being the establishment of the following reverse Hölder inequality for measures
\begin{equation}\label{ineq=holdernew}
    S_{1} \|\psi\|_{L^{M_{\Phi^\ast}}_\nu(\Omega)} \leq \|\psi\|_{L^{\Phi_{\max}}_\mu(\Omega)}, \quad \forall \psi \in C^\infty_c(\Omega),
\end{equation}
where $\nu$ and $\mu$ are the measures defined in \eqref{laled}. By \eqref{H1}, for given \(\delta > 0\) , let \(K > 0\) be such that
$$
\Phi^{\ast}(x, s t) \geq \Phi^{\ast}(x, t) (M_{\Phi^\ast}(x, s) - \delta), \text{ for } t\geq K.
$$
Then, one has
\begin{equation}\label{prbr}
\begin{aligned}
\int_{\O} \Phi^{\ast}\left(x, \frac{|\psi u_n|}{\lambda}\right) \, \diff x &\geq \int_{\{|u_n| \geq K\}} \Phi^{\ast}\left(x, \frac{|\psi| |u_n|}{\lambda}\right) \, \diff x \\&\geq\int_{\{|u_n| \geq K\}} \L(M_{\Phi^\ast}\L(x, \frac{|\psi|}{\lambda}\r) - \delta\r) \Phi^{\ast}(x, u_n) \, \diff x\\
&=
\int_\Omega \L(M_{\Phi^\ast}\L(x, \frac{|\psi|}{\lambda}\r) - \delta\r) \Phi^{\ast}(x, u_n) \, \diff x\\& \ \ \ \  - \int_{\{|u_n| < K\}} \L(M_{\Phi^\ast}\L(x, \frac{|\psi|}{\lambda}\r) - \delta\r) \Phi^{\ast}(x, u_n) \, \diff x.
\end{aligned}
\end{equation}
Since $\delta >0$ is arbitrary, the inequality \eqref{prbr} becomes
\begin{equation}\label{prbr1}
\begin{aligned}
\int_{\O} \Phi^{\ast}\left(x, \frac{|\psi u_n|}{\lambda}\right) \, \diff x &\geq\int_\Omega M_{\Phi^\ast}\L(x, \frac{|\psi|}{\lambda}\r) \Phi^{\ast}(x, u_n) \, \diff x\\& \ \ \ \  - \int_{\{|u_n| < K\}} M_{\Phi^\ast}\L(x, \frac{|\psi|}{\lambda}\r)\Phi^{\ast}(x, u_n) \, \diff x.
\end{aligned}
\end{equation}
For the first integral on the right-hand side, since \(\Phi^{\ast}(x, |u_n|) \, \diff x \overset{\ast }{\rightharpoonup } \nu \)  and $M_{\Phi^\ast}\L(\cdot, \frac{|\psi|}{\lambda}\r) $ is continuous and bounded (as \(\psi\) has compact support), we have
    \begin{equation}\label{I}
    \lim_{n \to \infty} \int_\Omega M_{\Phi^\ast}\L(x,  \frac{|\psi|}{\lambda}\r) \Phi^{\ast}(x, u_n) \, \diff x = \int_\Omega M_{\Phi^\ast}\L(x,  \frac{|\psi|}{\lambda}\r) \, \diff \nu.
    \end{equation}
  For the second integral, ass \(u_n \to 0\) a.e. $x \in \O$, \(\Phi^{\ast}(x, |u_n|) \to 0\) a.e. on \(\{|u_n| < K\}\), and by dominated convergence theorem, we infer that
    \begin{equation}\label{II}
    \lim_{n \to \infty} \int_{\{|u_n| < K\}} M_{\Phi^\ast}\L(x,  \frac{|\psi|}{\lambda}\r) \Phi^{\ast}(x, u_n) \, \diff x = 0.
    \end{equation}

Passing to the limit inferior, as $ n \to +\infty $, in \eqref{prbr} and taking into account \eqref{I} and \eqref{II}, we obtain
$$
\liminf_{n \to \infty} \int_\Omega \Phi^{\ast}\left(x, \frac{|\psi u_n|}{\lambda}\right) \, \diff x \geq \int_\Omega M_{\Phi^\ast}\L(x,  \frac{|\psi|}{\lambda}\r) \, \diff \nu,
$$
which implies
\begin{equation}\label{III}
\liminf_{n \to \infty} \|\psi u_n\|_{L^{\Phi^{\ast}}(\O)} \geq \|\psi\|_{L^{M_{\Phi^\ast}}_ \nu(\O)}.
\end{equation}
Therefore, the inequality \eqref{ineq=holdernew} follows by combining \eqref{V}, \eqref{III}, and \eqref{VI}. The rest of the proof is left to the reader.\end{proof}

The following lemma is a crucial consequence of Theorem~\ref{CCP1}.
\begin{lemma}\label{lem:strong_convergence}
Under the assumptions of Theorem~\ref{CCP1}, let $K \subset \O \setminus \{x_i\}_{i \in I}$ be a compact set, where $\{x_i\}_{i \in I} \subset \overline{\Omega}$ is the at-most-countable set of distinct points given in Theorem~\ref{CCP1}. Then
\begin{equation}\label{eq:strong_convergence}
u_n \to u \quad \text{strongly in } L^{\Phi^*}(K).
\end{equation}
\end{lemma}

\begin{proof}
Let \( K \subset \O \setminus \{x_i\}_{i \in I} \) be a compact set, where \(\{x_i\}_{i \in I}\) is the at most countable set of concentration points of the measure \( \nu \), as defined in Theorem \ref{CCP1}. We need to show that
$$
\int_K \Phi^*(x,|u_n - u|) \, dx \to 0 \quad \text{as } n \to \infty,
$$
where \( \Phi^* \) is the Sobolev conjugate of \(\Phi\), and \( u_n \rightharpoonup u \) weakly in \( W^{1,\Phi}_0(\O) \), with \( u_n \to u \) a.e. in \(\O\), and \( \Phi^*(x,|u_n|) \diff x \overset{\ast }{\rightharpoonup } \nu \)  in $\mathbb{M}(\overline{\O})$.

Let $\delta = \operatorname{dist}(K, \{x_{i}\}_{i \in I}) > 0$. Choose $R > 0$ such that $K \subset B_{R}(0)$, and define
$$
A_{\varepsilon} = \{x \in B_{R}(0) \mid \operatorname{dist}(x, K) < \varepsilon\}
$$
for $0 < \varepsilon < \delta$. Take a function $\chi_{\varepsilon} \in C^{\infty}_{0}(\O)$ satisfying
$$
0 \leq \chi_{\varepsilon} \leq 1, \quad \text{and} \quad \chi_{\varepsilon}(x) =
\begin{cases}
1 & \text{if } x \in A_{\varepsilon/2}, \\
0 & \text{if } x \in \O \setminus A_{\varepsilon}.
\end{cases}
$$

Since
$$
K \subset A_{\varepsilon/2} \subset A_{\varepsilon} \subset (\O \setminus \{x_{i}\}_{i \in I}) \cap B_{R}(0)
$$
holds for $0 < \varepsilon < \delta$, we deduce that
$$
\int_{K} \Phi^*(x, |u_n(x)|) \,\diff x \leq \int_{A_{\varepsilon}} \chi_{\varepsilon}(x) \Phi^*(x, |u_n(x)|) \,\diff x = \int_{\O} \chi_{\varepsilon}(x) \Phi^*(x, |u_n(x)|) \,\diff x.
$$
From \eqref{T.ccp.form.nu}, it follows that
$$
\limsup_{n \to \infty} \int_{K} \Phi^*(x, |u_n(x)|) \,\diff x \leq \int_{\O} \chi_{\varepsilon}(x) \, \diff \nu = \int_{\O} \chi_{\varepsilon}(x) \Phi^*(x, |u(x)|) \,\diff x \leq \int_{A_{\varepsilon}} \Phi^*(x, |u(x)|) \,\diff x.
$$
Letting $\varepsilon \to +0$, by Lebesgue's convergence theorem, we obtain
$$
\limsup_{n \to \infty} \int_{K} \Phi^*(x, |u_n(x)|) \,\diff x \leq \int_{K} \Phi^*(x, |u(x)|) \,\diff x.
$$
On the other hand, Fatou's lemma implies
$$
\int_{K} \Phi^*(x, |u(x)|) \,\diff x \leq \liminf_{n \to \infty} \int_{K} \Phi^*(x, |u_n(x)|) \,\diff x.
$$
Thus, we conclude that
\begin{equation}\label{u4}
\lim_{n \to \infty} \int_{K} \Phi^*(x, |u_n(x)|) \,\diff x = \int_{K} \Phi^*(x, |u(x)|) \,\diff x.
\end{equation}
Therefore, by Lemma~\ref{L.brezis-lieb}, we conclude the proof of the desired result.
\end{proof}

\subsection{The concentration-compactness principle in $\R^d$}\label{ccprd}
In this subsection, we establish our concentration-compactness results in the entire space $\mathbb{R}^d$. We employ a conjugate function introduced in the recent paper \cite{Cianchi2024}, which was also mentioned in the introduction. During our analysis, we encountered several obstacles in studying the properties of this conjugate. Nevertheless, we are able to prove an equivalence, in the sense of $"\sim"$, between the function $\Phi_d$ and another generalized Young function, similar to the equivalence previously established for $\Phi_{d,\diamond}$.

\begin{proof}[\rm\bf{Proof of Theorem \ref{CCP2}}]
Proceeding as in the proof of Theorem~\ref{CCP1}.
Set \( v_n = u_n - u \) for \(n \in \mathbb{N} \). Then,
\begin{eqnarray}
	v_n &\rightharpoonup& 0 \quad \text{in} \quad  \WV, \label{W0}\\
\Phi(\cdot,|\nabla v_n|)+  V	\Phi(\cdot,| v_n|) &\overset{\ast }{\rightharpoonup }&\bar{\mu}\quad \text{in } \mathbb{M}(\RD),\label{W1}\\
	\Phi_d(\cdot,|v_n|)&\overset{\ast }{\rightharpoonup }&\bar{\nu}\quad \text{in } \mathbb{M}(\RD). \label{W2}
\end{eqnarray}
\vspace{6mm}
Invoking  Lemmas~\ref{lem:reverse_holder_musielakRD}--\ref{lemmef} and Remark \ref{rem: lem4}, we conclude that
\begin{equation}\label{nu.1}
\Phi_d(x,|v_n|)\,\diff x \overset{\ast }{\rightharpoonup } \diff \bar{\nu} = \sum_{i\in I} \nu_i \delta_{x_i},
\end{equation}
weakly-* in the sense of measures.\\
Now, we use Lemma~\ref{L.brezis-lieb} to obtain
$$
\lim_{n\to\infty} \left(\int_{\mathbb{R}^d} \psi \Phi_d(x,|u_n|) \,dx - \int_{\mathbb{R}^d} \psi \Phi_d(x,|v_n|) \,dx\right) = \int_{\mathbb{R}^d} \psi \Phi_d(x,|u|) \,dx,
$$
for any \( \psi \in C^\infty_c(\mathbb{R}^d) \), from which the representation
$$
\Phi_d(x,|u_n|)\,\diff x \rightharpoonup d\nu = \Phi_d(x,|u|)\,\diff x + d\bar{\nu}
$$
follows. Then, we obtain \eqref{T.ccp.formnu}. \\
Similarly, following the proof of Theorem~\ref{CCP1}, we can establish the estimate \eqref{T.ccp.numu}.

To complete the proof, we now establish \eqref{T.ccp.formmu}. For any nonnegative function \(\psi \in C_0(\mathbb{R}^d)\), the mapping
$$
u \mapsto \int_{\mathbb{R}^d} \psi(x) \Big[\Phi(x,|\nabla u|) + V(x) \Phi(x,| u|)\Big] \diff x
$$
is convex and differentiable on \(\WV\). Consequently, it is weakly lower semicontinuous, which implies that
$$
\begin{aligned}
	&\int_{\mathbb{R}^d} \psi(x) \Big[\Phi(x,|\nabla u|) + V(x) \Phi(x,| u|)\Big] \diff x\\
	&\leq \liminf_{n\to  \infty } \int_{\mathbb{R}^d } \psi(x) \Big[\Phi(x,|\nabla u_n|) + V(x) \Phi(x,|u_n|)\Big] \diff x.
\end{aligned}
$$
By \eqref{mu_nto mu*}, we obtain
$$
\int_{\mathbb{R}^d} \psi(x) \Big[\Phi(x,|\nabla u|) + V(x) \Phi(x,| u|)\Big] \diff x \leq \int_{\mathbb{R}^d} \psi \diff \mu.
$$
Thus, it follows that
\begin{equation}\label{X0}
\mu \geq \Phi(\cdot,|\nabla u|) \diff x +  V \Phi(\cdot,| u|)\diff x.
\end{equation}
Following the some argument used in the proof of Theorem \ref{CCP1}, we get
\begin{equation}\label{X1}
\mu \geq \mu_1 := \sum_{i\in I}^{} \delta_{x_i} \mu_i.
\end{equation}
Then, since the measures \(\mu_1 \) and the Lebesgue measure are mutually singular, the desired identity \eqref{T.ccp.formmu} follows from \eqref{X0} and \eqref{X1}. This completes the proof.
\end{proof}
\begin{proof}[\rm\bf{Proof of Theorem \ref{CCP20}}]
The proof of Theorem~\ref{CCP20} follows the same approach as Theorem~\ref{CCP10}.
\end{proof}

\begin{proof}[\rm\bf{Proof of Theorem \ref{CCP3}}]
Let $\psi \in C^\infty(\mathbb{R}^d)$ be such that $0 \leq \psi \leq 1$, $|\nabla \psi| \leq 2$ in $\mathbb{R}^d$, $\psi \equiv 0$ on $B_{1}$, and $\psi \equiv 1$ on $B_2^c$. For each $R > 0$, define $\varphi_R(x) := \psi\left(\frac{x}{R}\right)$ for $x \in \mathbb{R}^d$.
We define
$$
\Psi_n(x) := \Phi(x, |\nabla u_n|) + V(x) \Phi(x, |u_n|) \quad \text{for } x \in \mathbb{R}^d.
$$
Then, we decompose
\begin{align}
	\int_{\mathbb{R}^d} \Psi_n(x) \, \mathrm{d}x = \int_{\mathbb{R}^d} \varphi_R^{m(x)} \Psi_n(x) \, \mathrm{d}x + \int_{\mathbb{R}^d} (1 - \varphi_R^{m(x)}) \Psi_n(x) \, \mathrm{d}x. \label{T.ccp.inf.decompose1}
\end{align}
By estimating
\begin{align*}
	\int_{B^c_{2R}} \Psi_n(x) \, \mathrm{d}x \leq \int_{\mathbb{R}^d} \varphi_R^{m(x)} \Psi_n(x) \, \mathrm{d}x \leq \int_{B^c_R} \Psi_n(x) \, \mathrm{d}x,
\end{align*}
we obtain
\begin{equation}\label{T.ccp.inf.mu_inf}
	\lim_{R \to \infty} \limsup_{n \to \infty} \int_{\mathbb{R}^d} \varphi_R^{m(x)} \Psi_n(x) \, \mathrm{d}x = \mu_\infty.
\end{equation}
On the other hand, \eqref{mu_nto mu*} and the fact that $1 - \varphi_R^{m(x)} \in C_c(\mathbb{R}^d)$ give
\begin{equation}\label{T.ccp.infinity.varphiR}
	\lim_{n \to \infty} \int_{\mathbb{R}^d} \left(1 - \varphi_R^{m(x)}\right) \Psi_n(x) \, \mathrm{d}x = \int_{\mathbb{R}^d} \left(1 - \varphi_R^{m(x)}\right) \, \mathrm{d}\mu.
\end{equation}
Note that $\lim\limits_{R \to \infty} \int_{\mathbb{R}^d} \varphi_R^{m(x)} \, \mathrm{d}\mu = 0$ in view of the Lebesgue dominated convergence theorem. Thus, \eqref{T.ccp.infinity.varphiR} yields
\begin{equation} \label{T.ccp.mu(Rd)}
	\lim_{R \to \infty} \lim_{n \to \infty} \int_{\mathbb{R}^d} \left(1 - \varphi_R^{m(x)}\right) \Psi_n(x) \, \mathrm{d}x = \mu(\mathbb{R}^d).
\end{equation}
Consequently, combining \eqref{T.ccp.inf.decompose1}, \eqref{T.ccp.inf.mu_inf} with \eqref{T.ccp.mu(Rd)}, we obtain \eqref{Tccpinfinity.mu}.

In the same manner, we decompose
\begin{equation}\label{T.ccp.inf.decompose2}
	\int_{\mathbb{R}^d} \Phi_d(x, |u_n|) \, \mathrm{d}x = \int_{\mathbb{R}^d} \varphi_R^{\delta(x)} \Phi_d(x, |u_n|) \, \mathrm{d}x + \int_{\mathbb{R}^d} \left(1 - \varphi_R^{\delta(x)}\right) \Phi_d(x, |u_n|) \, \mathrm{d}x.
\end{equation}
Similar arguments to those leading to \eqref{T.ccp.inf.mu_inf} and \eqref{T.ccp.mu(Rd)} give
\begin{equation}\label{T.ccp.inf.nu_inf}
	\lim_{R \to \infty} \limsup_{n \to \infty} \int_{\mathbb{R}^d} \varphi_R^{\delta(x)} \Phi_d(x, |u_n|) \, \mathrm{d}x = \nu_\infty
\end{equation}
and
\begin{equation}\label{T.ccp.inf.nu(Rd)}
	\lim_{R \to \infty} \lim_{n \to \infty} \int_{\mathbb{R}^d} \left(1 - \varphi_R^{\delta(x)}\right) \Phi_d(x, |u_n|) \, \mathrm{d}x = \nu(\mathbb{R}^d).
\end{equation}
Making use of \eqref{T.ccp.inf.decompose2}--\eqref{T.ccp.inf.nu(Rd)}, we obtain \eqref{T.ccp.infinitynu}.

Finally, we claim \eqref{ala8}. Let $\epsilon \in (0, m_\infty)$ be arbitrary and fixed. Then, we find $R_0 > 0$ such that
\begin{equation}\label{T.ccp.inf.h}
	|r(x) - r_\infty| < \epsilon, \quad \forall x \in B^c_{R_0}
\end{equation}
for each $r \in \{\kappa, \delta, m, \ell\}$. If, we denote
$$
r^-_R:=\inf_{x \in B_R^c} r(x) \text{ and }r^+_R:= \sup_{x\in B_R^c}r(x),
$$ then, from \eqref{T.ccp.inf.h}, we have
\begin{equation}\label{X2}
  r_\infty- \epsilon \leq  r^-_R \leq r^+_R \leq  r_\infty+ \epsilon, \ \forall R >R_0.
\end{equation}
 Obviously, $\varphi_R u_n \in \WV$, then by \eqref{ala11} we infer that
\begin{align} \label{T.ccp.inf.S.varphiR}
	S_2 \|\varphi_R u_n\|_{L^{\Phi_d}(\RD)} \leq \|\varphi_R u_n\|_{\WV}.
\end{align}
For $R > R_0$, using \eqref{AST}, \eqref{T.ccp.inf.h}, and  \eqref{X2}, we get
$$\begin{aligned}
	\|\varphi_R u_n\|_{L^{\Phi_d}(\RD)}&\geq \min\left\{\left(\int_{B_R^c} \varphi_R^{\delta(x)} \Phi_d(x, |u_n|) \, \mathrm{d}x\right)^{\frac{1}{\kappa^-_R }}, \left(\int_{B_R^c} \varphi_R^{\delta(x)} \Phi_d(x, |u_n|) \, \mathrm{d}x\right)^{\frac{1}{\delta^+_R}}\right\}\\
& \geq \min\left\{\left(\int_{B_R^c} \varphi_R^{\delta(x)} \Phi_d(x, |u_n|) \, \mathrm{d}x\right)^{\frac{1}{\kappa_\infty - \epsilon}}, \left(\int_{B_R^c} \varphi_R^{\delta(x)} \Phi_d(x, |u_n|) \, \mathrm{d}x\right)^{\frac{1}{\delta_\infty + \epsilon}}\right\} .
\end{aligned}$$
From this and \eqref{T.ccp.inf.nu_inf}, we obtain
\begin{equation} \label{nu.inf1}
	\liminf_{R \to \infty} \limsup_{n \to \infty} \|\varphi_R u_n\|_{L^{\Phi_d}(\RD)} \geq \min\biggl\{\nu_\infty^{\frac{1}{\kappa_\infty - \epsilon}}, \nu_\infty^{\frac{1}{\delta_\infty + \epsilon}} \biggr\}.
\end{equation}
On the other hand, for $R > R_0$, from \eqref{m-n} and \eqref{X2} it holds that
\begin{align} \label{T.ccp.inf.varphiR_un.1}
	\|\varphi_R u_n\|_{\WV}
	&\leq \max\bigg\{\left(\rho_{n,R}\right)^{\frac{1}{m_R^-}}, \left(\rho_{n,R}\right)^{\frac{1}{\ell_R^+}}\bigg\},\nonumber\\
&\leq \max\bigg\{\left(\rho_{n,R}\right)^{\frac{1}{m_\infty - \epsilon}}, \left(\rho_{n,R}\right)^{\frac{1}{\ell_\infty - \epsilon}}\bigg\}
\end{align}
with $\rho_{n,R} := \displaystyle \int_{B_R^c} \Big[ \Phi(x, |\nabla(\varphi_R u_n)|) + V(x) \Phi(x, |\varphi_R u_n|) \Big] \, \mathrm{d}x$.
Using Lemma~\ref{l1} again, we find $C_\epsilon > 1$ independent of $n, R$ such that
\begin{align*}
	\rho_{n,R} &= \int_{B_R^c} \Big[\Phi(x, |u_n \nabla \varphi_R + \varphi_R \nabla u_n|) + V(x) \Phi(x, |\varphi_R u_n|)\Big] \, \mathrm{d}x \\
	&\leq (1 + \epsilon)^\ell \int_{B_R^c} \varphi^{m(x)}_R \Big[ \Phi(x, |\nabla u_n|) + V(x) \Phi(x, |u_n|)\Big] \, \mathrm{d}x + C_\epsilon \int_{B_R^c} \Phi(x, |u_n \nabla \varphi_R|) \, \mathrm{d}x,
\end{align*}
i.e.,
\begin{align}\label{JkR}
	\rho_{n,R} \leq (1 + \epsilon)^\ell \int_{B_R^c} \varphi^{m(x)}_R \Psi_n(x) \, \mathrm{d}x + C_\epsilon \int_{B_R^c} \Phi(x, |u_n \nabla \varphi_R|) \, \mathrm{d}x.
\end{align}
Invoking \eqref{ala11} again, one has
$$
\lim_{n\to\infty} \int_{B_R^c}^{} \Phi(x, |u_n \nabla \varphi_R|) \, \mathrm{d}x = \lim_{n\to\infty} \int_{B_{2R}\backslash  \overline{B_R}}^{}\Phi(x, |u_n \nabla \varphi_R|) \, \mathrm{d}x = \int_{B_{2R}\backslash  \overline{B_R}}^{}\Phi(x, |u \nabla \varphi_R|) \, \mathrm{d}x ,
$$
and employing Proposition \ref{zoo}, we infer that
\begin{equation}\label{ala13}
\begin{aligned}
\int_{B_{2R}\backslash \overline{B_R}}^{}\Phi(x, |u \nabla \varphi_R|) \, \mathrm{d}x &\leq \int_{B_{2R}\backslash  \overline{B_R}}^{}\max\L\{|\nabla \varphi_R|^{\ell(x)},|\nabla \varphi_R|^{m(x)} \r\}\Phi(x, |u |) \, \mathrm{d}x\\  & \leq \max\L\{  \L(\frac{2}{R}\r)^{\ell^+},  \L(\frac{2}{R}\r)^{m^-}        \r\} \int_{B_{2R}\backslash \bar{B_R}}\Phi(x, |u |) \, \mathrm{d}x.\end{aligned}
\end{equation}
Thus
\begin{equation}\label{raar7}
\lim_{R\to +\infty}\lim_{n\to +\infty} \int_{B_R^c}^{} \Phi(x, |u_n \nabla \varphi_R|) \, \mathrm{d}x=0.
\end{equation}

Using \eqref{T.ccp.inf.mu_inf} and \eqref{raar7}, we derive from \eqref{JkR} that
\begin{equation*}
	\limsup_{R \to \infty} \limsup_{n \to \infty} \rho_{n,R} \leq (1 + \epsilon)^\ell \mu_\infty.
\end{equation*}
By this and \eqref{T.ccp.inf.varphiR_un.1}, we obtain
\begin{equation} \label{mu.inf1}
	\limsup_{R \to \infty} \limsup_{n \to \infty} \|\varphi_R u_n\|_{\WV} \leq (1 + \epsilon)^{\frac{\ell}{m_\infty - \epsilon}} \max \left\{\mu_\infty^{\frac{1}{m_\infty -\epsilon}}, \mu_\infty^{\frac{1}{\ell_\infty - \epsilon}} \right\}.
\end{equation}
Therefore, from \eqref{T.ccp.inf.S.varphiR}, \eqref{nu.inf1}, and \eqref{mu.inf1}, we derive
\begin{equation*}
	S_2 \min\biggl\{\nu_\infty^{\frac{1}{\delta_\infty + \epsilon}}, \nu_\infty^{\frac{\ell}{\kappa_\infty - \epsilon}} \biggr\} \leq (1 + \epsilon)^{\frac{1}{m_\infty - \epsilon}} \max \left\{\mu_\infty^{\frac{1}{m_\infty - \epsilon}}, \mu_\infty^{\frac{1}{\ell_\infty - \epsilon}} \right\}.
\end{equation*}
Hence, \eqref{ala8} follows since $\epsilon \in (0, m_\infty)$ was taken arbitrarily. The proof is complete.
\end{proof}

\subsection{Special instances}\label{exples}

In this final subsection, we validate our results, specifically Theorems \ref{CCP1}--\ref{CCP3}, by applying them to some special and widely used generalized Young functions. This subsection is divided into two paragraphs. In the first, we recall some generalized Young functions that the CCP studied in their functional spaces. We prove that these results become direct consequences of our results, or that our results are sharper, in the sense that we establish an optimal inequality of measures. In the second paragraph, we deliver some results of the CCP in certain different spaces that have not been explored until now.

Before presenting the examples, it is important to note that when studying problems in a bounded domain \(\Omega \subset \mathbb{R}^d\), only the behavior of generalized Young functions for large \(t\) is relevant for the embedding inequalities.
Specifically, if \(\Phi\) satisfies the \(\Delta_2\)-condition and \eqref{A0}, then for large \(t\), we have the equivalence
$$
\Phi^\ast \approx \Phi_d.
$$

Thus, for \(t \geq 0\), we may take
$$
\Phi^\ast \approx (\Phi_{\ast})_{\{\text{for large } t\}}.
$$
\subsubsection{Some existing CCP Results}
In this subsection, we focus on the measure estimates proven in some existing CCP results and compare them with our estimates \eqref{T.ccp.nu_mu},\eqref{NUMU}, \eqref{T.ccp.numu}, \eqref{NUMU2}, and \eqref{ala8}. These estimates make our results stand out, as they exhibit a level of sharpness and generality not found in other CCP theorems. As for the results on measure convergence, they are similar across all works on the CCP.

\vspace{6mm}

   {\underline{\it Classical Orlicz Young function}.}
If $\Phi$ is independent of $x$, that is,
$
\Phi(x,t) = \Phi(t),
$
where $\Phi$ is a classical Young function, then the Musielak space reduces to the Orlicz space. In this case, assumptions \eqref{A0}, \eqref{A1}, \eqref{A2}, and \eqref{B0} become unnecessary. Moreover, the Sobolev conjugate simplifies significantly and becomes independent of the domain $\Omega$ (whether bounded or unbounded). For further details, we refer to \cite{Cianchi1996}, where this property is explicitly discussed.

If we assume that there exist two positive constants \( m, \ell \) such that
\begin{equation}\label{ala20}
1 <m \leq \frac{t \phi(t)}{\Phi(t)} \leq \ell < d,
\end{equation}
then our Theorem \ref{CCP10} coincides with \cite[Theorem 4.3]{Bonder2024}.

\vspace{6mm}

{\it \underline{Variable exponents}. }
One of the most well-known examples of generalized Young functions that depend non-trivially on the \( x \)-variable is provided by variable exponents. These functions are defined as
\begin{equation}\label{var0}
\Phi(x,t) = t^{p(x)} \quad \text{for } x \in \mathbb{R}^d \text{ and } t \geq 0,
\end{equation}
where
$$
p : \mathbb{R}^d \to [1, \infty).
$$
Such functions form the foundation of variable exponent Lebesgue spaces, whose systematic study was first introduced in the seminal work \cite{Kovacik1991}. For the conditions \eqref{A0}, \eqref{A1}, and \eqref{A2}, we refer to \cite[Section 7.1]{Harjulehto2019}, where the authors provide sufficient assumptions on $p$ to ensure that the variable exponent Young function satisfies these conditions. Furthermore, it is easy to see that $\Phi$ fulfils \eqref{B0}.\\
 When \( p(x) < d \), from \cite[Example 3.10]{Cianchi2024}, we have
\begin{equation}\label{conj: expva}
\Phi_d(x,t) \approx  \begin{cases}
                                                     t^{\frac{dp_\infty}{d - p_\infty}}, & \mbox{for small } t  \\
                                                     t^{\frac{dp(x)}{d - p(x)}}, & \mbox{for large }t,
                                                   \end{cases}
\end{equation} where $p_\infty=\ds\lim_{|x|\to +\infty} p(x).$
where the implicit constants depend only on $ d$, $ p^- $, and $ p^+ $. We may assume, without loss of generality, that
$$
\Phi_d(x,t) = t^{\frac{dp(x)}{d - p(x)}} \quad \text{for all } x \in \mathbb{R}^d.
$$
This form of the Sobolev conjugate is particularly useful in the context of variable exponents, as discussed in \cite[Section 8.3]{DieningHarjulehtoHastoRuzicka2011}.

Returning to our results, it is evident that \( m(x) = \ell(x) = p(x) \) for all \( x \in \mathbb{R}^d \), where \( m \) and \( \ell \) are defined in \eqref{mar3}. Consequently, it follows that
$$
M_\Phi(x,t)=\Phi_{\max}(x,t) =\Phi_{\min}(x,t)= \Phi(x,t) = t^{p(x)},
$$
for all \( x \in \mathbb{R}^d \) and \( t \geq 0 \). Then, the estimations \eqref{T.ccp.nu_mu} and \eqref{NUMU} become
$$
S_\Phi\nu_i^{\frac{1}{p^\ast(x_i)}} \leq \mu_i^{\frac{1}{p(x_i)}}.
$$
This result aligns with the findings in \cite[Theorem 1.1]{bonder2010}. Therefore, our Theorem \ref{CCP1} coincides with the case of variable exponents in bounded domains.
For the entire space \( \mathbb{R}^d \), the results of Ho, Kim, and Sim in \cite[Theorems 3.3 \& 3.4]{Ho2019} align with our Theorems \ref{CCP2} and \ref{CCP3}. There are other works addressing the CCP in the context of variable exponent spaces, whose results can be shown to align with ours. For further details, see \cite{Fu2009, Fu2010}.

\vspace{6mm}

{\it \underline{Double-phase with variable exponents}.} Now, we turn our attention to the new double-phase Young function with variable exponents, defined as
$$
\mathcal{H}(x,t) = t^{p(x)} + a(x) t^{q(x)}, \quad \forall x \in \mathbb{R}^d \text{ and for } t \geq 0,
$$
where $1<p(x)<d$ and $p(x)\leq q(x)$ and $a\in L^\infty(\RD)$.
The functional analysis of this Musielak Young function was introduced in the work of \cite{CrespoBlancoGasinskiHarjulehtoWinkert2022}, where the authors established fundamental results regarding this Young function and its associated Musielak and Sobolev spaces. As shown in \cite{CrespoBlancoGasinskiHarjulehtoWinkert2022}, the function \( \mathcal{H} \) satisfies conditions \eqref{A0}, \eqref{A1}, and \eqref{A2} (under certain assumptions on \( p \), \( q \), and \( a \)).
For the Sobolev conjugate, it is clear that if \( a(x) = 0 \), then \( \mathcal{H}(x,t) = \Phi(x,t) = t^{p(x)} \), which we discussed in the previous paragraph. Now, we assume \( a(x) \neq 0 \). In this case, the conjugate depends on whether \( q(x) \) is greater than, equal to, or less than \( d \). As we are focusing on the case of variable exponents, we restrict our attention to \( q(x) < d \). Under this assumption, and using the fact that $ \h^\infty(t) \approx t^{p_\infty} $ for $ t \leq 1 $, we can prove, using the some argument in \cite[Example 3.11]{Cianchi2024}, that
\begin{equation}\label{conj: doubphas}
\h_d(x,t) \approx  \begin{cases}
                                                     t^{\frac{dp_\infty}{d - p_\infty}}, & \mbox{for small } t  \\
                                                     t^{\frac{dp(x)}{d - p(x)}} + a(x)^{\frac{d}{d - q(x)}} t^{\frac{dq(x)}{d - q(x)}}, & \text{for } x \in \mathbb{R}^d \text{ and } t \text{ large}.
                                                     \end{cases}\end{equation}

For further details, we refer to the papers \cite{Ha2024, Ha2025, Ho2023}, where the authors study embedding results for double-phase problems with variable exponents.\\
Through a straightforward calculation, we find that \( m(x) = p(x) \),\( \ell(x) = q(x) \), $\kappa (x)=\frac{d p(x)}{d-p(x)}$ and $\delta (x)=\frac{d q(x)}{d-q(x)}$. Consequently, we have
$$
\mathcal{H}_{\max}(x,t) = \max \left\{ t^{p(x)}, t^{q(x)} \right\} \quad \text{for } x \in \mathbb{R}^d \text{ and } t \geq 0,
$$
and
$$
\begin{aligned}
\h_{d_{\min}}(x,t) &= \min \left\{ t^{p^\ast(x)}, t^{q^\ast(x)} \right\} \quad \text{for } x \in \mathbb{R}^d \text{ and } t \geq 0.
\end{aligned}
$$
It follows that
$$
\mathcal{H}_\infty^{-1}(x,t) = \min \left\{ t^{\frac{1}{p(x)}}, t^{\frac{1}{q(x)}} \right\} \quad \text{for } x \in \mathbb{R}^d \text{ and } t \geq 0,
$$
and
$$
\h_{d_{\min}}^{-1}(x,t) = \max \left\{ t^{\frac{1}{p^\ast(x)}}, t^{\frac{1}{q^\ast(x)}} \right\} \quad \text{for } x \in \mathbb{R}^d \text{ and } t \geq 0.
$$
This implies
$$
\frac{1}{\h_{d_{\min}}^{-1}(x,t)} = \min \left\{ t^{-\frac{1}{p^\ast(x)}}, t^{-\frac{1}{q^\ast(x)}} \right\}.
$$
Hence, our estimates \eqref{T.ccp.nu_mu} and \eqref{T.ccp.numu} yield
$$
S_{\mathcal{H}} \min\left\{ \nu_i^{\frac{1}{p^\ast(x_i)}}, \nu_i^{\frac{1}{q^\ast(x_i)}} \right\}= S_{\mathcal{H}} \frac{1}{\h_{d_{\min}}^{-1}\left(x_i, \frac{1}{\nu_i}\right)} \leq \frac{1}{\h_{\max}^{-1}\left(x_i, \frac{1}{\mu_i}\right)} = \max \left\{ \mu_i^{\frac{1}{p(x_i)}}, \mu_i^{\frac{1}{q(x_i)}} \right\} \quad \forall i \in I.
$$
This demonstrates that our CCP results coincide exactly with those established in \cite[Theorem 2.1]{Ha2024} and \cite[Theorem 3.1]{Ha2025} under equivalent assumptions.

\vspace{6mm}

{\it \underline{Logarithmic double-phase with variable exponents in bounded domain}.}
 Now, we introduce a very recently proposed special generalized Young function, namely the \textit{logarithmic double-phase function with variable exponents}, defined in a bounded domain as follows. Let $\h \colon \overline{\Omega} \times [0, \infty) \to [0, \infty)$ be given by
\begin{equation}\label{ala2011}
\h(x, t) = t^{p(x)} + a(x) t^{q(x)} \log(e + t).
\end{equation}
In \cite{Arora2023}, the authors prove that \( \mathcal{H} \) satisfies conditions \eqref{A0}, \eqref{A1}, and \eqref{A2} under certain assumptions.\\
When considering the Sobolev conjugate, the case \( a(x) = 0 \) is straightforward, leading to \( \h(x,t) = \Phi(x,t) = t^{p(x)} \), as previously discussed. However, when \( a(x) \neq 0 \), the behavior of the conjugate depends on how \( q(x) \) compares to \( d \). In this work, we focus on the situation where \( q(x) < d \), which is particularly relevant for variable exponents. Under this condition, and proceeding like \cite[Example 3.11]{Cianchi2024}, we derive the following asymptotic expression:

\begin{equation}\label{ala22}
\h^\ast(x,t) \approx t^{\frac{dp(x)}{d - p(x)}} + \L(a(x)\log(e+t)\r)^{\frac{d}{d - q(x)}} t^{\frac{dq(x)}{d - q(x)}}, \quad \text{for } x \in \overline{\Omega} \text{ and } t \geq 0.
\end{equation}
This coincides with the Sobolev conjugate discussed in \cite[Section 3]{Arora2025}. Furthermore, invoking \cite[Lemma 3.3]{Arora2023}, we prove that
$$
m(x) = p(x) \quad \text{and} \quad \ell(x) = q(x) + \kappa,
$$
where \(\kappa = \frac{e}{e + t_0}\), with \(t_0\) being the only positive solution of the equation
$$
t_0 = e \log(e + t_0).
$$
As a result, we obtain the following expressions:
$$
\mathcal{H}_{\max}(x,t) = \max \left\{ t^{p(x)}, t^{q(x)+\kappa} \right\} \quad \text{for } x \in \O\text{ and } t \geq 0,
$$
and
$$
\begin{aligned}
\h^\ast_{\min}(x,t) &= \min \left\{ t^{p^\ast(x)}, t^{(q(x)+\kappa)^\ast} \right\} \quad \text{for } x \in \mathbb{R}^d \text{ and } t \geq 0.
\end{aligned}
$$
From this, we derive
\begin{equation}\label{V0}
\mathcal{H}_{\max}^{-1}(x,t) = \min \left\{ t^{\frac{1}{p(x)}}, t^{\frac{1}{q(x)+\kappa}} \right\} \quad \text{for } x \in \mathbb{R}^d \text{ and } t \geq 0,
\end{equation}
and
\begin{equation}\label{V1}
\L(\h^\ast_{\min}\r)^{-1}(x,t) = \max \left\{ t^{\frac{1}{p^\ast(x)}}, t^{\frac{1}{(q(x)+\kappa)^\ast}} \right\} \quad \text{for } x \in \mathbb{R}^d \text{ and } t \geq 0.
\end{equation}
This leads to the inequality
$$
\frac{1}{\L(\h^{\ast}_{\min}\r)^{-1}(x,t)} = \min \left\{ t^{-\frac{1}{p^\ast(x)}}, t^{-\frac{1}{(q(x)+\kappa)^\ast}} \right\}.
$$
On the one hand, Theorem \ref{CCP1} coincides with \cite[Theorem 4.6]{Arora2025}. On the other hand, using our estimates \eqref{T.ccp.nu_mu} and \eqref{T.ccp.numu}, along with \eqref{V0} and \eqref{V1}, we obtain
$$
S_{\h} \min\left\{ \nu_i^{\frac{1}{p^\ast(x_i)}}, \nu_i^{\frac{1}{(q(x_i)+\kappa)^\ast}} \right\} = S_{\h} \frac{1}{\L(\h^\ast_{\min}\r)^{-1}\left(x_i, \frac{1}{\nu_i}\right)} \leq \frac{1}{\h_{\max}^{-1}\left(x_i, \frac{1}{\mu_i}\right)} = \max \left\{ \mu_i^{\frac{1}{p(x_i)}}, \mu_i^{\frac{1}{q(x_i)+\kappa}} \right\} \quad \forall i \in I.
$$
This demonstrates that our result (Theorem \ref{CCP1}) coincides with those established in \cite[Theorem 4.6]{Arora2025}.\vspace{6mm}
\subsubsection{Exploring CCP in Novel Settings}
In this subsection, we present some CCP results in previously unstudied functional spaces. In the sequel, we assume that condition \eqref{VV} holds.

\vspace{5mm}
\underline{ First}, we begin by applying Theorems \ref{CCP2} and \ref{CCP3} to a weighted Orlicz-Sobolev space defined on the entire space $\mathbb{R}^d$. Let \( \Phi \) be a classical Orlicz function satisfying \eqref{ala20}, and let \( \Phi_d \) be as defined in \cite{Cianchi1996}. Let \( \{u_n\}_{n \in \mathbb{N}} \subset W^{1,\Phi}_V(\mathbb{R}^d) \) be a sequence such that \( u_n \rightharpoonup u \) weakly in \( W^{1,\Phi}_V(\mathbb{R}^d) \). Then, Theorems \ref{CCP2} and \ref{CCP3} still hold in the Orlicz-Sobolev space \( W^{1,\Phi}_V(\mathbb{R}^d) \).

\vspace{5mm}

\underline{ Second}, we consider the \textit{logarithmic double-phase variable exponents} defined on the entire space $\mathbb{R}^d$. Let $\mathcal{H}$ and $\mathcal{H}_\ast$ be as defined in \eqref{ala2011} and \eqref{ala22}, respectively. Consider a sequence \( \{u_n\}_{n \in \mathbb{N}} \subset W^{1,\mathcal{H}}_V(\mathbb{R}^d) \) such that \( u_n \rightharpoonup u \) weakly in \( W^{1,\mathcal{H}}_V(\mathbb{R}^d) \). In this setting, Theorems \ref{CCP2} and \ref{CCP3} remain valid in the Musielak-Orlicz-Sobolev space \( W^{1,\mathcal{H}}_V(\mathbb{R}^d) \) with $\kappa (x)= \frac{p(x)d}{d-p(x)}$ and $\delta (x)= \frac{(q(x)+\kappa)d}{d-q(x)-\kappa}$.

\vspace{5mm}

\underline{ Next},  we consider an important example: the Musielak-Orlicz-Sobolev space with variable exponent, which arises from the well-known \textit{Zygmund (or interpolation) spaces} associated with the function
\begin{equation}\label{zyg-func}
\Phi(x,t) := t^{p(x)}\log(e + t), \quad \text{for all } x \in \mathbb{R}^d, \, t \geq 0.
\end{equation}
Under suitable assumptions on the exponent \( p \), this function satisfies conditions \eqref{A0}, \eqref{A1}, and \eqref{A2} (see \cite[Chapter 7]{Harjulehto2019}). Moreover, it is straightforward to verify that \(\Phi\) fulfills \eqref{B0}.

Assuming \( p(x) < d - 1 \), the conjugate function \(\Phi_d\) satisfies the asymptotic estimate
\begin{equation}\label{conj-zyg}
\Phi_d(x,t) \approx \begin{cases}
t^{\frac{d p_\infty}{d - p_\infty}} \log^{\frac{d}{d - p_\infty}}(e + t), & \text{if } 0 \leq t < 1, \\
t^{\frac{d p(x)}{d - p(x)}} \log^{\frac{d}{d - p(x)}}(e + t), & \text{if } t \geq 1,
\end{cases}
\end{equation}
where \( p_\infty = \lim_{|x| \to +\infty} p(x) \), and the implicit constants depend only on \( d \), \( p^- \), and \( p^+ \). Without loss of generality, we may assume that
$$
\Phi^\ast(x,t) = \Phi_d(x,t) \approx \L( t^{p(x)} \log(e + t) \r)^{\frac{d}{d - p(x)}} \quad \text{for all } x \in \mathbb{R}^d.
$$

We may take \( m(x) = p(x) \) and \( \ell(x) = p(x) + 1 \). A distinctive feature of this generalized \( N \)-function is that the limit \eqref{eq:Matuszewska} associated with \(\Phi\) is uniform, which allows us to define the Matuszewska-Orlicz function related to \(\Phi_d\). Consequently, we obtain:
$$
\Phi^\ast_{\min}(x,t) = \Phi_{d_{\min}}(x,t) =\min \curly{ t^{p_\ast(x)}, t^{(p(x)+1)_\ast} }, \quad \Phi_{\max}(x,t) = \max\curly{ t^{p(x)}, t^{p(x)+1} },
$$
and
$$
M_{\Phi_d}(x,t) = t^{\frac{d p(x)}{d - p(x)}}, \quad \text{for all } x \in \mathbb{R}^d, \, t \geq 0.
$$

Consider a sequence \( \{u_n\}_{n \in \mathbb{N}} \subset W^{1,\Phi}(\Omega) \) (respectively, \( \subset W^{1,\Phi}_V(\mathbb{R}^d) \)) such that \( u_n \rightharpoonup u \) weakly in \( W^{1,\Phi}(\Omega) \) (respectively, in \( W^{1,\Phi}_V(\mathbb{R}^d) \)). In this setting, Theorems \ref{CCP1} and \ref{CCP10} (respectively, Theorems \ref{CCP2}, \ref{CCP20}, and \ref{CCP3}) remain valid in the Musielak-Orlicz-Sobolev space \( W^{1,\Phi}(\Omega) \) (respectively, in \( W^{1,\Phi}_V(\mathbb{R}^d) \)).

Furthermore, this example highlights the sharpness of the estimates \eqref{NUMU} and \eqref{NUMU2} in Theorems \ref{CCP10} and \ref{CCP20} compared to \eqref{T.ccp.nu_mu} and \eqref{T.ccp.numu} in Theorem \ref{CCP1}. Specifically, we have:
$$
S_1 \frac{1}{\L( \Phi_{\min}^\ast \r)^{-1}\L( x_i, \frac{1}{\nu_i} \r)} \leq S_1 \frac{1}{ M_{\Phi^\ast}^{-1}\L( x_i, \frac{1}{\nu_i} \r)} \leq \frac{1}{\Phi_{\max}^{-1}\L( x_i, \frac{1}{\mu_i} \r)}, \quad \forall i \in I.
$$
\vspace{5mm}

\underline{ Finally}, we focus on a very recent special type of generalized Young function related to the double-phase problem with variable exponents, where the exponents depend on the solution, introduced in \cite{Bahrouni-Bahrouni-Missaoui-Radulescu-2024}. This generalized Young function is defined as
     \begin{equation}\label{ala2012}
       \h(x,t)=\int_{0}^{t} h(x,s)sds,
     \end{equation}
     where
     $$h(x,s):=\begin{cases}
                 s^{p(x,s)-2}+a(x) s^{q(x,s)-2}, & \mbox{if } s>0 \\
                 0, & \mbox{if } s=0.
               \end{cases}$$
     Using some assumptions introduced in \cite{Bahrouni-Bahrouni-Missaoui-Radulescu-2024}, along with \cite[Proposition 3.10]{Bahrouni-Bahrouni-Missaoui-Radulescu-2024} and \cite[Proposition 2.15]{BAHROUNI2025104334}, the generalized Young function $\mathcal{H}$ satisfies conditions \eqref{A0}, \eqref{A1}, and \eqref{A2} in both bounded and unbounded domains, under certain assumptions on the exponents $p$, $q$, and the weight $a$. Furthermore, by invoking \cite[Theorem 1.2]{BAHROUNI2025104334}, we know that $\mathcal{H}$ satisfies \eqref{conv0}. Consequently, we can define the two Sobolev conjugates $\mathcal{H}_\ast$ and $\mathcal{H}^\ast$. This allows us to apply Theorem \ref{CCP1} to the space $W^{1,\mathcal{H}}_0(\Omega)$ and Theorems \ref{CCP2} and \ref{CCP3} to $W^{1,\mathcal{H}}_V(\mathbb{R}^d)$. Note that, this Young function is characterized by the lack of an explicit form, which complicates calculations and makes it challenging to compute the Matuszewska-Orlicz functions associated with $\mathcal{H}^\ast$ and $\mathcal{H}_\ast$. As a result, the estimates \eqref{T.ccp.nu_mu} and \eqref{T.ccp.numu} become more intricate. However, this does not prevent us from providing an alternative estimate given by
$$
S_\h \min\L\{ \nu_i^{\frac{1}{p^-_\ast}}, \nu_i^{\frac{1}{q^+_\ast}} \r\} \leq\max\L\{ \mu_i^{\frac{1}{p^-}}, \mu_i^{\frac{1}{q^+}} \r\}, \quad \forall i\in I.
$$

\section{Application: The existence of solutions}\label{APP}
In this section, we employ variational methods to establish the existence of solutions to problem~\eqref{prb}. The concentration-compactness principles (CCPs) derived in Subsection~\ref{ccprd} will play a pivotal role in our analysis.

\vskip 6mm
Recall that the problem under consideration is the following one
\begin{equation*}
	\begin{aligned}
		-\operatorname{div} &\left(\frac{\phi\L(x,|\nabla u|\r)\nabla u}{\vert \nabla u \vert}\right)
		&+ V(x) \left(\frac{\phi\L(x,|u|\r)u}{|u|}\right) =  f(x,u) + \lambda \frac{ \phi_d(x,|u|)}{|u|}u, \quad x \in \mathbb{R}^d.
	\end{aligned} \tag{$\mathcal{P}_1$}
\end{equation*}

First, we recall the definition of the Cerami condition, which will be needed.
\begin{definition} \label{dcrmi}
	Let $X$ be a Banach space, and denote by $X^\ast$ its topological dual space. Given $L \in C^1(X)$, we say that $L$ satisfies the Cerami-condition, $\textnormal{(C)}_c$-condition for short, if every $\textnormal{(C)}_c$-sequence $\left\{u_n\right\}_{n \in \mathbb{N}} \subseteq X$ such that
	\begin{enumerate}
		\item[\textnormal{(C$_1$)}]
			$L\left(u_n\right) \to c$ , as $n\to +\infty$,
		\item[\textnormal{(C$_2$)}]
			$ \left(1+\left\|u_n\right\|_X\right) L^{\prime}\left(u_n\right) \rightarrow 0$ in $X^*$ as $n \rightarrow \infty$,
		\end{enumerate}
		admits a strongly convergent subsequence in $X$.
\end{definition}

\subsection{Some properties of the energy functional} \label{Prty}
In this subsection, we work within the assumptions of Theorem~\ref{thm:exis}. We begin by introducing the variational operator associated with problem~\eqref{prb} and its corresponding energy functional. Let $\L( \WV \r)^\ast $ be the  dual space of $\WV$ with its duality pairing denoted by $\langle\cdot, \cdot\rangle$. We say that $u \in \WV$ is a weak solution of problem \eqref{prb}, if
\begin{align*}
	&\int_{\RD}^{}\left( \left(\phi\L(x,|\nabla u|\r)\frac{\nabla u}{\vert \nabla u \vert}\right)\nabla v + V(x)\left(\phi\L(x,|u|\r)\frac{u}{|u|}\right)v\right)\mathrm{d}x= \int_{\mathbb{R}^d} f(x, u) v \mathrm{d}x-\l\int_{\mathbb{R}^d} \phi_d(x, u) v\,\mathrm{d}x
\end{align*}
for all $v \in \WV$. We define the functionals $I, K\colon  \WV \rightarrow \mathbb{R}$ by
\begin{align*}
	I(u)= \int_{\RD} \L(\Phi \L(x, |\nabla u| \r)\,\mathrm{d}x + V(x) \Phi \L(x, |u| \r)\,\r)\mathrm{d}x
	\quad\text{and}\quad
	K(u)=\int_{\mathbb{R}^d}\L( F(x, u)\,\mathrm{d}x -\l\Phi_d(x, u)\r) \diff x .
\end{align*}

We are now prepared to examine the key properties of the functionals $I$ and $K$.

\begin{theorem}\label{op2}
	Let hypotheses \eqref{H} and \eqref{VV}  be satisfied. Then the functional $I$ is well-defined and of class $C^1$ with
	\begin{equation}\label{GFD}
		\begin{aligned}
			\scal{I'(u), v}
			&=  \int_{\RD}^{}\left( \left(\phi\L(x,|\nabla u|\r)\frac{\nabla u}{\vert \nabla u \vert}\right)\nabla v + V(x)\left(\phi\L(x,|u|\r)\frac{u}{|u|}\right)v\right)\mathrm{d}x,\ \forall u,\ v \in \WV.
		\end{aligned}
	\end{equation}
	Moreover, the operator $I'$ has the following properties:
	\begin{enumerate}
		\item[\textnormal{(i)}]
			The operator $I'\colon \WV \rightarrow \left(\WV\right)^*$ is continuous, bounded, and strictly monotone.
		\item[\textnormal{(ii)}]
			The operator $I'$ fulfills the $\left(\mathrm{S}_{+}\right)$-property, i.e.,
			\begin{align*}
				u_n \rightharpoonup u \text { in } \WV \quad \text{and} \quad \limsup _{n \rightarrow \infty}\left\langle I'\left(u_n\right), u_n-u\right\rangle \leq 0,
			\end{align*}
			imply $u_n \rightarrow u$ in $\WV$.
		\item[\textnormal{(iii)}]
			The operator $I'$  is a homeomorphism.
		\item[\textnormal{(iv)}]
			The operator $I'$ is strongly coercive, that is,
			\begin{align*}
				\lim_{\|u\|_{\WV}\rightarrow+\infty} \dfrac{\langle I' (u), u\rangle}{\|u\|_{\WV}}\to+\infty.
			\end{align*}
	\end{enumerate}
\end{theorem}

\begin{proof}
The derivation of formula \eqref{GFD} follows a reasoning analogous to that of \cite[Proposition 3.16]{Bahrouni-Bahrouni-Missaoui-Radulescu-2024}. The remaining part of the proof employs similar arguments as in \cite[Proposition 3.17]{Bahrouni-Bahrouni-Missaoui-Radulescu-2024}.
\end{proof}

\begin{proposition}\label{Ku}
	Let \eqref{F} , \eqref{H}, and \eqref{VV} be satisfied. Then, the following hold:
	\begin{enumerate}
		\item[\textnormal{(i)}]
			The functional $K\colon \WV \rightarrow \mathbb{R}$  is of class $C^1$ with
			\begin{align*} \left\langle K^{\prime}(u), v \right\rangle = \int_{\mathbb{R}^d}\L( f(x, u)  -\l\phi_d(x,u)  \r)v\,\mathrm{d}x
			\end{align*}
			for all $u, v \in \WV$.
		\item[\textnormal{(ii)}]
			The functional \begin{equation}\label{eq: fonctionalJ}
			                 J_\l= I- K \end{equation} is of class $C^1$ with
			\begin{align*}
				\left\langle J_\l^{\prime}(u), v \right\rangle=\left\langle I^{\prime}(u), v \right\rangle-\left\langle K^{\prime}(u), v \right\rangle\quad\text{for all } u,v \in \WV.
			\end{align*}
	\end{enumerate}
\end{proposition}

\begin{remark}
	From Proposition \ref{Ku}, it follows that the solutions of \eqref{prb} correspond to the critical points of the Euler-Lagrange energy functional $J_\lambda$.
\end{remark}

First, we present the following lemma that will be used in the proof of the main existence result.
\begin{lemma}\label{crmi}
	Let the assumptions \eqref{H}, \eqref{HAST}, \eqref{VV} and \eqref{F} be satisfied. Then, the functional $J_\lambda$, as defined in \eqref{eq: fonctionalJ}, satisfies the $\textnormal{(C)}_c$ -condition for all $\lambda>0$, with $c\in \R$ satisfying
\begin{equation}\label{PS1.c.cond}
	c<\round{\frac{\kappa^-}{\ell^+}-1}\frac{m^-}{\delta^+}\min\left\{S^{\tau_1},S^{\tau_2}\right\}\min \left\{\l^{-\sigma_1},\l^{-\sigma_2}\right\}-\widetilde{C},
\end{equation}
where $S$, $\tau_1$, $\tau_2$, $\sigma_1$, $\sigma_2$, and $\widetilde{C}$ will be defined in the proof.
\end{lemma}
\begin{proof} We employ the methodologies used in the proofs of \cite[Lemma 4.6]{BAHROUNI2025104334} and \cite[Lemma 4.3]{Ha2025}.
	Let $\left\{u_n\right\}_{n \in \mathbb{N}} \subseteq \WV$ be a sequence such that $\left(\mathrm{C}_1\right)$ and $\left(\mathrm{C}_2\right)$ from Definition \ref{dcrmi} hold. We divide the proof into two steps.\vspace{0.2cm}\\
	\textbf{Step 1.} We prove that $\curly{u_n}_{n \in \N}$ is bounded in $\WV$.

	First, from $\left(\mathrm{C}_1\right)$ we have that there exists a constant $M>0$ such that for all $n \in \mathbb{N}$ one has $\left|J_\lambda\left(u_n\right)\right| \leq M$, so
	\begin{align*}
		\left|\int_{\RD}\left(\Phi\L(x, |\nabla u_n| \r)+V(x)\Phi\L(x, |u_n|\r)\right)\,\mathrm{d}x- \int_{\RD} \L(F\left(x, u_n\right) +\lambda\Phi_d (x,|u_n|)\r)\,\mathrm{d}x\right| \leq M,
	\end{align*}
	which implies that
	\begin{align}\label{cr0}
		I(u_n)- \int_{\RD} \L(F\left(x, u_n\right) +\l\Phi_d (x,|u_n|)\r)\,\mathrm{d}x \leq M\quad\text{for all }n \in \mathbb{N}.
	\end{align}
	Besides, from $\left(\mathrm{C}_2\right)$, there exists $\left\{\varepsilon_n\right\}_{n \in \mathbb{N}}$ with $\varepsilon_n \rightarrow 0^{+}$such that
	\begin{align}\label{eq0001}
		\left|\left\langle J_\lambda^{\prime}\left(u_n\right), v\right\rangle\right| \leq \frac{\varepsilon_n\|v\|_{\WV}}{1+\left\|u_n\right\|_{\WV}} \quad \text {for all } n \in \mathbb{N} \text { and for all } v \in \WV.
	\end{align}
	Then, choosing $v=u_n$, one has
	\begin{align*}
		\left|\int_{\RD}\left(\phi(x,|\nabla u_n|)|\nabla u_n|+V(x)\phi(x,|u_n|)|u_n|\right)\,\mathrm{d}x- \int_{\RD}\L( f\left(x, u_n\right) u_n +\lambda \phi_d(x,|u_n|)|u_n|\r)\,\mathrm{d}x\right|\leq \varepsilon_n,
	\end{align*}
	which, multiplied by $\frac{-1}{\ell^+}$, leads to
	\begin{align*}
		-\frac{1}{\ell^+}\int_{\RD}\left(\phi(x,|\nabla u_n|)|\nabla u_n|+V(x)\phi(x,|u_n|)|u_n|\right)\,\mathrm{d}x+\frac{1}{\ell^+} \int_{\RD}\L( f\left(x, u_n\right) u_n +\lambda \phi_d(x,|u_n|)|u_n|\r)\,\mathrm{d}x \leq c_1,
	\end{align*}
	for some $c_1>0$ and for all $n \in \mathbb{N}$. Hence, invoking \eqref{H}, we conclude that
	\begin{align}\label{cr1}
	-I(u_n)+\frac{1}{\ell^+}\int_{\RD} \L(f\left(x, u_n\right)u_n + \l \Phi_d(x,|u_n|)|u_n|\r),\mathrm{d}x \leq c_1\quad\text{for all }n \in \mathbb{N}.
	\end{align}
	Adding \eqref{cr0} and \eqref{cr1} it follows, from \eqref{raar2}, that
	\begin{align}\label{zaar}
		C& \geq \int_{\RD}\left[ \frac{1}{\ell^+}f(x,u_n)u_n-F(x,u_n)\right]\,\mathrm{d}x + \l\int_{\RD}\left[ \frac{1}{\ell^+} \phi_d(x,|u_n|)|u_n|- \Phi_d(x,|u_n|) \r] \diff x \nonumber \\&  \geq \int_{\RD}\widetilde{F}(x,u_n)\,\mathrm{d}x +\l \int_{\RD}\L(\frac{1}{\ell^+}-\frac{1}{\kappa^-}\r)  \phi_d(x,|u_n|)|u_n|\,\mathrm{d}x,
	\end{align}
	for all $n \in \N$ with some constant $C > 0$. Since $\ell^+\leq \kappa^-$, it follows that
\begin{equation}\label{eq0002}
		C \geq \int_{\RD}\left[ \frac{1}{\ell^+}f(x,u_n)u_n-F(x,u_n)\right]\,\mathrm{d}x = \int_{\RD} \widetilde{F}(x,u_n)\diff x,\end{equation} for all $n \in \N$.\vspace{0.2cm}\\
	Now, arguing by contradiction, we assume that $\Vert u_n\Vert_{\WV}\to+\infty$. Then $\Vert u_n\Vert_{\WV} \geq 1$ for $n$ large enough. Let $v_n=\dfrac{u_n}{\Vert u_n\Vert_{\WV}}\in \WV $, so $\Vert v_n\Vert_{\WV}=1$ and, up to subsequence, we can assume that
	\begin{align*}
		v_n\rightharpoonup v\quad \text{in } \WV
		\quad\text{and}\quad
		v_n(x)\rightarrow v(x)\quad \text{a.e.\,in } \RD.
	\end{align*}
	Note that, exploiting  \eqref{m-n} and \eqref{H}, we find, for $n$ large enough, that
	\begin{align*}
		\langle J_{\l}^{'}(u_n),u_n\rangle
		& = \int_{\RD}\L(\phi(x,\vert \nabla u_n\vert)|\nabla u_n|+V(x)\phi(x, |u_n|) | u_n|\r)\,\mathrm{d}x   \\&~~~~~~ \ \ -\int_{\RD}\L(f(x,u_n)u_n+\l\phi_d(x,|u_n|)|u_n|\r)\,\mathrm{d}x\\
		& \geq m^-I(u_n) -\int_{\RD}\L(f(x,u_n)u_n+\l\phi_d(x,|u_n|)|u_n|\r)\,\mathrm{d}x\\
& \geq \Vert u_n\Vert_{\WV}^{m^-}-\int_{\RD}\L(f(x,u_n)u_n+\l\phi_d(x,|u_n|)|u_n|\r)\,\mathrm{d}x,\\
	\end{align*}
	since $\Vert u_n\Vert_{\WV} \geq 1$. Thus
	\begin{align}\label{eq0004}
		\frac{\langle J_{\l}^{'}(u_n),u_n\rangle}{\Vert u_n\Vert_{\WV}^{m^-}}\geq 1-\int_{\RD}\frac{f(x,u_n)}{\Vert u_n\Vert_{\WV}^{m^-}}u_n\,\mathrm{d}x-\int_{\RD}\frac{\l\phi_d(x,|u_n|)}{\Vert u_n\Vert_{\WV}^{m^-}}|u_n|\,\mathrm{d}x.
	\end{align}
	From $(\ref{eq0001})$ and $(\ref{eq0004})$, it follows that
	\begin{align}\label{eq0005}
		\limsup\limits_{n\rightarrow+\infty}\L[\int_{\RD}\frac{f(x,u_n)}{\Vert u_n\Vert_{\WV}^{m^-}}u_n\,\mathrm{d}x+\l\int_{\RD}\frac{\phi_d(x,|u_n|)}{\Vert u_n\Vert_{\WV}^{m^-}}|u_n|\,\mathrm{d}x\r]\geq 1.
	\end{align}
	For $r\geq 0$, we set
	\begin{align*}
		\mathfrak{F}(r):=\inf\left\lbrace \tilde{F}(x,s)\colon  x\in\RD \text{ and } s\in \mathbb{R} \text{ with }  s \geq r\right\rbrace.
	\end{align*}
	By \eqref{F}(ii)-(iv), we have
	\begin{align}\label{fr0}
		\mathfrak{F}(r)>0\quad \text{for all } r \text{ large}
		\quad\text{and}\quad
		\mathfrak{F}(r)\rightarrow +\infty\quad \text{as }  r\rightarrow+\infty.
	\end{align}
	For $0\leq a<b\leq +\infty$, let
	\begin{align*}
		A_n(a,b)&:=\left\lbrace x\in \RD\colon  a\leq \vert u_n(x)\vert <b\right\rbrace,\\
		c_a^b&:=\inf\left\lbrace \frac{\tilde{F}(x,s)}{\vert s\vert^{m^-}}\colon  x\in\RD \text{ and } s\in \mathbb{R}\setminus\{0\} \text{ with } a\leq \vert s\vert<b\right\rbrace.
	\end{align*}
	Note that
	\begin{equation}\label{eq0006}
		\tilde{F}(x,u_n)\geq c_a^b\vert u_n\vert^{m^-}\quad \text{for all } x\in A_n(a,b).
	\end{equation}
	It follows, from $(\ref{eq0002})$, that
	\begin{equation}\label{eq0007}
		\begin{aligned}
			C & \geq \int_{\RD}\tilde{F}(x,u_n)\,\mathrm{d}x\\
			& =\int_{A_n(0,a)}\tilde{F}(x,u_n)\,\mathrm{d}x+\int_{A_n(a,b)}\tilde{F}(x,u_n)\,\mathrm{d}x+\int_{A_n(b,+\infty)}\tilde{F}(x,u_n)\,\mathrm{d}x\\
			& \geq\int_{A_n(0,a)}\tilde{F}(x,u_n)\,\mathrm{d}x+c_a^b\int_{A_n(a,b)}\vert u_n\vert^{m^-}\,\mathrm{d}x+\mathfrak{F}(b)\vert A_n(b,+\infty)\vert
		\end{aligned}
	\end{equation}
	for $b$ large enough. \\
Utilizing \eqref{ineq for embed} and \eqref{Em} , we infer, based on \cite[Theorem 3.2.6]{Harjulehto2019} and \cite[Corollary 9.10]{Brez}, that
\begin{equation}\label{new inj}
\WV \hookrightarrow \W \hookrightarrow W^{1,m^-}(\RD) \hookrightarrow L^{r}(\RD),\end{equation} where $r\in[m^-,m^-_\ast].$
So, we get $\gamma_r>0$ such that
\begin{equation}\label{TZTM}
\Vert v_n\Vert^r_{L^r(\RD)}\leq \gamma_r\Vert v_n\Vert^r_{\WV}=\gamma_3 \text{ with } m^-\leq r\leq m^-_*. \end{equation}
 Let $0<\varepsilon<\frac{1}{6}$. By assumption \eqref{F}(iii), there exists $a_\varepsilon>0$ such that
	\begin{align}\label{eq0008}
		\vert f(x,s)\vert \leq \frac{\varepsilon}{6\gamma_{m^-}}\vert s\vert^{m^--1}\quad\text{for all } \vert s\vert \leq a_\varepsilon.
	\end{align}
	Thus, from  \eqref{TZTM} and \eqref{eq0008}, we obtain
	\begin{equation}\label{eq0009}
		\begin{aligned}
			\int_{A_n(0,a_\varepsilon)}\frac{f(x,u_n)}{\Vert u_n\Vert^{m^-}_{\WV}}u_n\,\mathrm{d}x
			& \leq \frac{\varepsilon}{6\gamma_{m^-}} \int_{A_n(0,a_\varepsilon)}\frac{\vert u_n\vert^{m^-}}{\Vert u_n\Vert_{\WV}^{m^-}}\,\mathrm{d}x\\
			& \leq \frac{\varepsilon}{6\gamma_{m^-}} \int_{A_n(0,a_\varepsilon)}\vert v_n\vert^{m^-}\,\mathrm{d}x\\
			& \leq \frac{\varepsilon}{6\gamma_{m^-}} \gamma_{m^-} \Vert v_n\Vert_{\WV}^{m^-}\\
			& = \frac{\varepsilon}{6}\quad \text{for all } n\in \mathbb{N}.
		\end{aligned}
	\end{equation}
	Now, exploiting \eqref{eq0006}, \eqref{eq0007} and assumption \eqref{F}(i), we see that
	\begin{equation}\label{ezez}
		C'\geq \int_{A_n(b,+\infty)}\tilde{F}(x,u_n)\,\mathrm{d}x\geq \mathfrak{F}(b)\vert A_n(b,+\infty)\vert.
	\end{equation}
	It follows, using \eqref{fr0}, that
	\begin{align}\label{eq00010}
		\vert A_n(b,+\infty)\vert\rightarrow 0\quad \text{as } b\rightarrow+\infty \text{ uniformly in } n.
	\end{align}
	Set $\displaystyle{\sigma'=\frac{\sigma}{\sigma-1}}$ where $\sigma$ is defined in \eqref{F}(iv). Since $\sigma>\frac{d}{m^-}$, one sees that $m^-\sigma^{'}\in(m^-,k^-)$.

	Let $\tau\in (m^-\sigma^{'},k^-)$. Using \eqref{TZTM}, H\"older's inequality and $(\ref{eq00010})$, for $b$ large, we find
	\begin{equation}\label{eq00020}
		\begin{aligned}
			\left( \int_{A_n(b,+\infty)}\vert v_n\vert ^{m^-\sigma^{'}}\,\mathrm{d}x \right)^{\frac{1}{\sigma^{'}}}
			& \leq \vert A_n(b,+\infty)\vert ^{\frac{\tau-m^-\sigma^{'}} {\tau\sigma^{'}}}\left( \int_{A_n(b,+\infty)}\vert v_n\vert ^{m^-\sigma^{'}\frac{\tau}{m^-\sigma^{'}}}\,\mathrm{d}x\right)^{\frac{m^-}{\tau}} \\
			& \leq \vert A_n(b,+\infty)\vert ^{\frac{\tau-m^-\sigma^{'}}{\tau\sigma^{'}}}\left( \int_{A_n(b,+\infty)}\vert v_n\vert ^{\tau}\,\mathrm{d}x\right)^{\frac{m^-}{\tau}} \\
			& \leq \vert A_n(b,+\infty)\vert ^{\frac{\tau-m^-\sigma^{'}}{\tau\sigma^{'}}}\gamma_{\tau}\Vert v_n\Vert_{\WV}^{m^-}\\
			& = \vert A_n(b,+\infty)\vert ^{\frac{\tau-m^-\sigma^{'}}{\tau\sigma^{\prime}}}\gamma_{\tau}\\
			& \leq \frac{\varepsilon}{6\L(\tilde{c}C'\r)^{\frac{1}{\sigma}}}\quad \text{uniformly in } n.
		\end{aligned}
	\end{equation}
	By \eqref{F}(iv), H\"older's inequality, $(\ref{ezez})$ and $(\ref{eq00020})$, we can choose $b_\varepsilon\geq r_0$ large so that
	\begin{equation}\label{eq00011}
		\begin{aligned}
			\int_{A_n(b_\varepsilon,+\infty)}\L|\frac{f(x,u_n)}{\Vert u_n\Vert_{\WV}^{m^-}}u_n\r|\,\mathrm{d}x
			& \leq\int_{A_n(b_\varepsilon,+\infty)}\frac{|f(x,u_n)|}{\vert u_n\vert^{m^--1}}\vert v_n\vert ^{m^-}\,\mathrm{d}x\\
			& \leq \left( \int_{A_n(b_\varepsilon,+\infty)}\left\lvert\frac{f(x,u_n)}{\vert u_n\vert^{m^--1}}\right\rvert^{\sigma}\,\mathrm{d}x\right)^{\frac{1}{\sigma}} \left( \int_{A_n(b_\varepsilon,+\infty)}\vert v_n\vert ^{m^-\sigma^{'}}\,\mathrm{d}x\right)^{\frac{1}{\sigma^{'}}} \\
			& \leq \left( \tilde{c}\int_{A_n(b_\varepsilon,+\infty)}\tilde{F}(x,u_n)\,\mathrm{d}x\right)^{\frac{1}{\sigma}} \left( \int_{A_n(b_\varepsilon,+\infty)}\vert v_n\vert ^{m^-\sigma^{'}}\,\mathrm{d}x\right)^{\frac{1}{\sigma^{'}}} \\
			& \leq \frac{\varepsilon}{6}\quad \text{uniformly in } n.
		\end{aligned}
	\end{equation}
	Next, from $(\ref{eq0007})$, we have
	\begin{equation}\label{eq00014}
		\begin{aligned}
			\int_{A_n(a_\varepsilon,b_\varepsilon)}\vert v_n\vert^{m^-}\,\mathrm{d}x
			& =\frac{1}{\Vert u_n\Vert_{\WV}^{m^-}}\int_{A_n(a_\varepsilon,b_\varepsilon)}\vert u_n\vert^{m^-}\,\mathrm{d}x\\
			&\leq \frac{C}{c_{a_{\epsilon}}^{b_{\epsilon}} \Vert u_n\Vert_{\WV}^{m^-}}\to 0\quad \text{as }  n\to +\infty.
		\end{aligned}
	\end{equation}
	Since $\displaystyle{\frac{f(x,s)}{ \vert s\vert^{m^--1}}}$ is continuous on $a_\varepsilon\leq \vert s\vert\leq b_\varepsilon$,  there exists $c>0$ depending on $a_\varepsilon$ and $b_\varepsilon$ and independent from $n$, such that
	\begin{align}\label{eq00012}
		\vert f(x,u_n)\vert \leq c\vert u_n\vert^{m^--1}\quad \text{for all } x\in A_n(a_{\epsilon},b_{\epsilon}).
	\end{align}
	Thus, using $(\ref{eq00014})$ and $(\ref{eq00012})$, we can choose $n_0$ large enough such that
	\begin{equation}\label{eq00013}
		\begin{aligned}
			\int_{A_n(a_\varepsilon,b_\varepsilon)}\L|\frac{f(x,u_n)}{\Vert u_n\Vert_{\WV}^{m^-}}u_n \r|\,\mathrm{d}x
			& \leq\int_{A_n(a_\varepsilon,b_\varepsilon)}\frac{|f(x,u_n)|}{\vert u_n\vert^{m^--1}}\vert v_n\vert^{m^-}\,\mathrm{d}x\\
			& \leq c\int_{A_n(a_\varepsilon,b_\varepsilon)}\vert v_n\vert^{m^-}\,\mathrm{d}x\\
			& \leq c\frac{C'}{c_{a_\varepsilon}^{b_\varepsilon} \Vert u_n\Vert_{\WV}^{m^-}}\\
			& \leq \frac{\varepsilon}{6}\quad \text{for all } n\geq n_0.
		\end{aligned}
	\end{equation}
	Combining \eqref{eq0009}, \eqref{eq00011} and \eqref{eq00013}, we find that
	\begin{equation}\label{zaar3}
		\int_{\RD}\frac{f(x,u_n)}{\Vert u_n\Vert_{\WV}^{m^-}}u_n\,\mathrm{d}x\leq \frac{\varepsilon}{2}\quad \text{for all } n\geq n_0
	\end{equation}
On the other hand, from \eqref{F}(i), \eqref{zaar} and \eqref{eq0007}, we deduce that there exists a constant $C=C(m^-,\ell^+,d)>0$ such that
\begin{equation}\label{zaar0}
  C\geq \int_{\RD}^{} \phi_d(x,|u_n|)|u_n|dx,
\end{equation} for all $n \in \N.$
Therefore, it holds that
$$
\begin{aligned}
\int_{\RD}\frac{\phi_d(x,|u_n|)}{\Vert u_n\Vert_{\WV}^{m^-}}|u_n|\,\mathrm{d}x &= \frac{1}{\Vert u_n\Vert_{\WV}^{m^-}}\int_{\RD}^{} \phi_d(x,|u_n|)|u_n|dx\\
&\leq \frac{C'}{\Vert u_n\Vert_{\WV}^{m^-}} \longrightarrow 0 \text{ as } n\to +\infty.\end{aligned}
$$
Hence, it follows that
\begin{equation}\label{zaar1}
\int_{\RD}\frac{\phi_d(x,|u_n|)}{\Vert u_n\Vert_{\WV}^{m^-}}|u_n|\,\mathrm{d}x \leq \frac{\varepsilon}{2} \text{ for all } n \geq n_0.
\end{equation}
Gathering \eqref{zaar3} and \eqref{zaar1}, we obtain that
$$
\int_{\RD}\frac{f(x,u_n)}{\Vert u_n\Vert_{\WV}^{m^-}}u_n\,\mathrm{d}x+\int_{\RD}\frac{\phi_d(x,|u_n|)}{\Vert u_n\Vert_{\WV}^{m^-}}|u_n|\,\mathrm{d}x \leq \varepsilon,
$$

	which contradicts to \eqref{eq0005}. Therefore, $\lbrace u_n\rbrace_{n\in\mathbb{N}}$ is bounded in $\WV$.\vspace{0.2cm}\\
\textbf{Step 2.} $u_n \to u $ in $\WV$ as $n\rightarrow +\infty$ up to a subsequence.

Since the sequence $\{u_n\}_{n \in \mathbb{N}} \subset \WV$ is bounded by Step~1, we may apply Theorems~\ref{CCP2} and~\ref{CCP3}. Consequently, up to a subsequence, there exist a  most
countable index set $I$, a family of distinct points $\{x_i\}_{i \in I} \subset \overline{\Omega}$, and positive numbers $\{\nu_i\}_{i \in I}, \{\mu_i\}_{i \in I} \subset (0, \infty)$ such that

	\begin{gather}
	u_n(x) \to u(x) \quad \text{a.a.} \ \ x \in\mathbb{R}^d,  \label{CVG1}\\
	u_n \rightharpoonup u  \quad \text{in} \ \WV, \label{CVG2}\\
	\Phi(\cdot,|\nabla u_n|)+  V \Phi(\cdot,|u_n|) \overset{\ast }{\rightharpoonup }\mu \geq \Phi(\cdot,|\nabla u|)+ V \Phi(\cdot,|u|)+ \sum_{i\in I} \mu_i \delta_{x_i} \ \text{in}\  \mathbb{M}(\mathbb{R}^d),
	\label{CVG3}\\
	\Phi_d(\cdot,|u_n|)\overset{\ast }{\rightharpoonup }\nu=\Phi_d(\cdot,|u|) + \sum_{i\in I}\nu_i\delta_{x_i} \ \text{in}\ \mathbb{M}(\mathbb{R}^d),\label{CVG4}\\
	S_2 \frac{1}{\L(\Phi_{d_{\min}}\r)^{-1}(x_i, \frac{1}{\nu_i})} \leq \frac{1}{\Phi_{\max}^{-1}(x_i, \frac{1}{\mu_i})}, \quad \forall i\in I,
	\label{CVG44}
\end{gather}
and
\begin{gather}
	\limsup_{n \to \infty}\int_{\mathbb{R}^d} \Big[\Phi(x,|\nabla u_n|)+ V(x) \Phi(x,|u_n|) \Big] \, \diff x = \mu(\mathbb{R}^d)+\mu_\infty,
	\label{CVG5}\\
	\underset{n\to\infty}{\limsup}\int_{\mathbb{R}^d}\Phi_d(x,|u_n|)\diff x = \nu(\mathbb{R}^d)+\nu_\infty,
	\label{CVG6}\\
S_2 \min \left\{\nu_\infty^{\frac{1}{\kappa_{\infty}}}, \nu_\infty^{\frac{1}{\delta_{\infty}}} \right\} \leq \max \left\{\mu_\infty^{\frac{1}{m_{\infty}}},\,\mu_\infty^{\frac{1}{\ell_{\infty}}} \right\}.\label{CVG7}
\end{gather}

\textbf{Claim 1:} $I=\emptyset$.  Suppose contrarily that there is $i \in I$. Let $\psi \in C_c^\infty(\RD)$ be such that $0 \leq \psi \leq 1$, $|\nabla \psi| \leq 2$ in $\RD$, $\psi \equiv 1$ on $B_{1/2}$ and $\spp (\psi) \subset B_1$. For each $x_i\in\RD$ and $\delta>0$, define $\psi_{i,\delta}(x):=\psi\left(\frac{x-x_i}{\delta}\right)$ for $x \in \RD$.
For any $n\in\N$, it is clear that $\psi_{i,\delta}u_n \in \WV$. Hence, invoking  \eqref{raar2}, \eqref{H}, and \eqref{F}(i), we get

\begin{align}
	\notag
		\langle J_\l'(u_n) ,\psi_{i,\delta}u_n \rangle & =\int_{\RD}\L[ \phi(x,|\nabla u_n|) \frac{\nabla u_n}{|\nabla u_n|} \nabla (u_n\psi_{i,\delta}) +V(x) \phi(x,|u_n|)|u_n| \psi_{i,\delta}\r] \diff x \\ &\ \ \ - \int_{\RD}^{}\L[ f(x,u_n)u_n  +\l\phi_d(x,|u_n|)|u_n|\r] \psi_{i,\delta} \diff x \label{raar} \\
&=\int_{\RD}\psi_{i,\delta} \L[\phi(x,|\nabla u_n|) |\nabla u_n|  +V(x) \phi(x,|u_n|)|u_n| \r]\diff x \notag  \\ & \ \ + \int_{\RD} \phi(x,|\nabla u_n|) \frac{\nabla u_n}{|\nabla u_n|} \nabla (\psi_{i,\delta}) u_n \diff x- \int_{\RD}^{} \L[f(x,u_n)u_n  + \l \phi_d(x,|u_n|)|u_n|\r] \psi_{i,\delta} \diff x \notag \\
&\geq m^- \int_{\RD}\psi_{i,\delta} \L[\Phi(x,|\nabla u_n|)+V(x) \Phi(x,|u_n|)\r]\diff x \notag \\&\ \  -  \int_{\RD}^{} \L[c_{\hat{a}} b^+\B(x,|u_n|)  + \l\delta^+\Phi_d(x,|u_n|)\r] \psi_{i,\delta} \diff x + \int_{\RD} \phi(x,|\nabla u_n|) \frac{\nabla u_n}{|\nabla u_n|} \nabla (\psi_{i,\delta}) u_n \diff x , \notag
\end{align}
where $c_{\hat{a}}:= \|\hat{a}\|_{L^\infty(\RD)}.$

Let $\epsilon>0$ be arbitrary. Due to \eqref{Yi} and Lemma \ref{lm1}, one has
\begin{align} \label{raar.b}
\int_{\RD} \L| \phi(x,|\nabla u_n|) \frac{\nabla u_n}{|\nabla u_n|} \nabla (\psi_{i,\delta}) u_n \r|\diff x  &\leq \int_{\RD} \phi(x,|\nabla u_n|) \vert\nabla (\psi_{i,\delta})\vert \vert u_n \vert \diff x \notag\\
& = \int_{\RD} \epsilon \phi(x,|\nabla u_n|) \frac{1}{\epsilon} \L|\nabla (\psi_{i,\delta})u_n \r| \diff x \\
	&\leq \epsilon \int_{\RD} \widetilde{\Phi}(x,\phi(x,|\nabla u_n|))\diff x + C_\epsilon \int_{\RD} \Phi(x,|\nabla \psi_{i,\delta} u_n|)\diff x\notag\\
	&\leq \epsilon(\ell^+-1) \int_{\RD}\Phi(x,|\nabla u_n|)\diff x + C_\epsilon \int_{\RD} \Phi(x,|\nabla \psi_{i,\delta} u_n|)\diff x\notag\\
	&\leq \epsilon M + C_\epsilon \int_{\RD} \Phi(x,|\nabla \psi_{i,\delta} u_n|)\diff x,\notag
\end{align}
where $C_\epsilon>0$ is independent of $n$ and $\delta$, and
\begin{equation} \label{M}
	M:=\sup_{n \in \mathbb{N}} \ell^+\int_{\RD}  \left[\Phi(x,|\nabla u_n|)+ V(x) \Phi(x,|u_n|)\right] \diff x \in (0,\infty)
\end{equation}
due to Step 1. Using \eqref{raar.b} along with the fact that $\spp (\psi_{i,\delta}) \subset B_\delta(x_i)$, we can deduce, from \eqref{raar}, that
\begin{align}\label{PS1.est1}
	\notag
	&m^-\int_{\RD} \psi_{i,\delta} \Big[\Phi(x,|\nabla u_n|)+ V(x) \Phi(x,|u_n|) \Big] \diff x   \leq \langle J_\l'(u_n),\psi_{i,\delta} u_n \rangle + \l \delta^+\int_{\RD} \psi_{i,\delta} \Phi_d(x,|u_n|) \diff x  \\
	& \hspace{2cm}  + c_{\hat{a}} b^+ \int_{B_\delta(x_i)} \psi_{i,\delta} \B(x,|u_n|) \diff x+\epsilon M +C_\epsilon \int_{B_\delta(x_i)} \Phi(x,| \nabla \psi_{i,\delta} u_n|) \diff x .
\end{align}
Clearly, $\{\psi_{i,\delta}u_n\}_{n \in \mathbb{N}}$ is bounded in $\WV$. Then, the inequality \eqref{eq0001} implies that
\begin{equation} \label{PS1.lim1}
	\lim_{n \to \infty} \langle J'(u_n),\psi_{i,\delta} u_n \rangle = 0.
\end{equation}
In view of \eqref{ala11}, it follows, from \eqref{F}(i) and \eqref{CVG2}, that
\begin{equation}  \label{PS.c.lim2}
	\lim_{n \to \infty} \int_{B_\delta(x_i)} \psi_{i,\delta}  \B(x,|u_n|) \diff x = 	\int_{B_\delta(x_i)} \psi_{i,\delta}  \B(x,|u|) \diff x.
\end{equation}
Again by \eqref{ala11}, we have
\begin{equation} \label{PS1.lim3}
	\lim_{n \to \infty} \int_{B_\delta(x_i)} \Phi(x,| \nabla \psi_{i,\delta} u_n|) \diff x =  \int_{B_\delta(x_i)} \Phi(x,| \nabla \psi_{i,\delta} u|) \diff x .
\end{equation}
Passing to the limit as $n \to \infty$ in \eqref{PS1.est1} and using \eqref{CVG3}--\eqref{CVG4} together with \eqref{PS1.lim1}--\eqref{PS1.lim3}, we obtain
\begin{equation}\label{PS1.mu.nu.lim}
	m^-\int_{\RD} \psi_{i,\delta}  \diff \mu \leq \l \delta^+\int_{\RD}\psi_{i,\delta} \diff \nu+ c_{\hat{a}}b^+\int_{B_\delta(x_i)} \psi_{i,\delta} \B(x,|u|) \diff x+\epsilon M + C_\epsilon\int_{B_\delta(x_i)} \Phi(x,| \nabla \psi_{i,\delta} u|) \diff x .
\end{equation}
The fact that $\B(\cdot,|u|)\in L^1(\RD)$ gives
\begin{equation}\label{PS1.lim4}
	\lim_{\delta \to 0^+} \int_{B_\delta(x_i)} \psi_{i,\delta} \B(x,|u|) \diff x = 0.
\end{equation}

On the other hand, by applying Proposition~\ref{zoo}, 
and performing the change of variables $y = \tfrac{x - x_i}{\delta}$, we obtain
\begin{equation}\label{azs1}\begin{aligned}
\int_{B_\delta(x_i)} \Phi(x,| \nabla \psi_{i,\delta} u|) \diff x &=  \int_{B(1)}\Phi \L( \delta y+x_i, \frac{| \nabla \psi(y) |}{\delta} |u | \r)  \delta^d  \diff y \\
& \leq \int_{B(1)} \delta ^{d-\ell^+ }\Phi \L( \delta y+x_i, | \nabla \psi(y) | |u | \r)    \diff y.
\end{aligned}
\end{equation}  
Therefore, invoking~\eqref{B0} and Proposition~\ref{zoo}, 
and using~\eqref{azs1} together with the dominated convergence theorem, we conclude that
\begin{equation}\label{PS1.lim5}
	\lim_{\delta \to 0^+} \int_{B_\delta(x_i)} \Phi(x,| \nabla \psi_{i,\delta} u|) \diff x =0.
\end{equation}
By letting $\delta \to 0^+$ in \eqref{PS1.mu.nu.lim} and using \eqref{PS1.lim4} and \eqref{PS1.lim5}, we deduce that taht
\begin{equation}\label{4.m-n}
	m^-\mu_i \leq \delta^+ \l \nu_i + \epsilon M, \text{ with } \mu_i=\mu\L(\curly{x_i}\r) \text{ and } \nu_i=\nu\L(\curly{x_i}\r).
\end{equation}
This leads to
\begin{equation} \label{PL4.3-mn}
	\mu_i \leq \frac{\delta^+}{m^-}\l \nu_i,
\end{equation}
since $\epsilon>0$ was chosen arbitrarily. From \eqref{CVG44} and  \eqref{PL4.3-mn},  we obtain
\begin{equation*}  \label{PS1.est2}
	\begin{aligned} S\min\left\{(\l^{-1}\mu_i)^{\frac{1}{\kappa(x_i)}},(\l^{-1}\mu_i)^{\frac{1}{\delta(x_i)}} \right\} &\leq  \frac{S_2}{\Phi_{d_{\min}}^{-1}(x_i, \frac{1}{\nu_i})}\\ & \leq \frac{1}{\Phi_{\max}^{-1}(x_i, \frac{1}{\mu_i})}= \max\left\{ \mu_i^{\frac{1}{m(x_i)}}, \mu_i^{\frac{1}{\ell(x_i)}} \right \},\end{aligned}
\end{equation*} with $S= S_2\L(\frac{m^-}{\delta^+}\r)^{\frac{1}{\kappa^-}}$.
This yields to
\begin{equation} \label{PS1.mu_i}
	\mu_i \geq S^{\frac{\eta_i\beta_i}{\beta_i- \eta_i}}\l^{-\frac{\eta_i}{\beta_i-\eta_i}},
\end{equation}
where $\eta_i \in \{m(x_i),\ell(x_i)\}$ and $\beta_i \in \{\kappa(x_i),\delta(x_i)\}$. It is not difficult to see that
\begin{equation*}
	\left(\frac{m\delta}{\delta-m}\right)^-\leq \frac{m(x_i)\delta(x_i)}{\delta(x_i)-m(x_i)}\leq \frac{\eta_i\beta_i}{\beta_i-\eta_i}\leq \frac{\ell(x_i)\kappa(x_i)}{\kappa(x_i)-\ell(x_i)}\leq \left(\frac{\ell \kappa }{\kappa-\ell}\right)^+
\end{equation*}
and
\begin{equation*}
	\left(\frac{m(x)}{\delta(x)-m(x)}\right)^-\leq\frac{m(x_i)}{\delta(x_i)-m(x_i)}\leq \frac{\eta_i}{\beta_i-\eta_i}\leq \frac{\ell(x_i)}{\kappa(x_i)-\ell(x_i)}\leq \left(\frac{\ell(x)}{\kappa(x)-\ell(x)}\right)^+.
\end{equation*}
The last two inequalities jointly with \eqref{PS1.mu_i} and \eqref{4.m-n} imply
\begin{equation} \label{PS1.mu_i.2}
	\l \nu_i \frac{\delta^+}{m^-} \geq	\mu_i \geq \min\left\{S^{\tau_1},S^{\tau_2}\right\}\min \left\{\l^{-\sigma_1},\l^{-\sigma_2}\right\},
\end{equation}
where
\begin{equation}\label{PL4.3-0}
\begin{cases}
\tau_1 := \left(\dfrac{m(x)\delta(x)}{\delta(x) - m(x)}\right)^-, \\
\tau_2 := \left(\dfrac{\ell(x)\kappa(x)}{\kappa(x) - \ell(x)}\right)^+, \\
\sigma_1 := \left(\dfrac{m(x)}{\delta(x) - m(x)}\right)^-, \\
\sigma_2 := \left(\dfrac{\ell(x)}{\kappa(x) - \ell(x)}\right)^+.
\end{cases}
\end{equation}

On the other hand, it follows from  \eqref{HAST}, \eqref{eq0007}, and \eqref{CVG6} that
\begin{align}\label{PS1.c.nu}
	\notag c &=\lim_{n\to\infty}\left[J(u_n)-\frac{1}{\ell^+}\langle J'(u_n) ,u_n\rangle \right]\\
     \notag &\geq  \limsup_{n \to \infty} \int_{\RD} \widetilde{F}(x,u_n)\diff x + \l \L( \frac{\kappa^-}{\ell^+}-1\r) \limsup_{n \to \infty} \int_{\RD} \Phi_d (x,|u_n|)\diff x\\
\notag	&\geq   \limsup_{n \to \infty}\int_{A_n(0,a)} \widetilde{F}(x,u_n)\diff x + \l \L( \frac{\kappa^-}{\ell^+}-1\r) \limsup_{n \to \infty} \int_{\RD} \Phi_d (x,|u_n|)\diff x\\
&\geq  \limsup_{n \to \infty}\int_{A_n(0,a)} \widetilde{F}(x,u_n)\diff x + \L( \frac{\kappa^-}{\ell^+}-1\r)\l\left[\nu(\RD)+\nu_\infty\right].
\end{align}
Note that, by invoking \eqref{F}\textnormal{(i)} and applying Hölder's inequality, we obtain
$$
\begin{aligned}
\int_{A_n(0,a)}\L| \widetilde{F}(x,u_n)\r|\diff x&\leq  b^+ \int_{A_n(0,a)} \a \B(x,|u_n|) \diff x\\
& \leq  b^+ \int_{A_n(0,a)} \a \B(x,a) \diff x\leq C_a\|\hat{a}\|_{L^1(\RD)}\\& \leq \widetilde{C}.
\end{aligned}
$$
From the last inequality and \eqref{PS1.c.nu} we arrive at
\begin{equation}\label{araa}
  c\geq \L( \frac{\kappa^-}{\ell^+}-1\r)\l\left[\nu(\RD)+\nu_\infty\right] - \widetilde{C}.
\end{equation}
Taking  into account  \eqref{CVG4}, \eqref{PS1.mu_i.2} and \eqref{araa} we obtain
\begin{equation}\label{PL4.3-c}
	c \geq   \L( \frac{\kappa^-}{\ell^+}-1\r) \l \nu_i - \widetilde{C}
	\geq
	\L( \frac{\kappa^-}{\ell^+}-1\r)\frac{m^-}{\delta^+}\min\left\{S^{\tau_1},S^{\tau_2}\right\}\min \left\{\l^{-\sigma_1},\l^{-\sigma_2}\right\}- \widetilde{C}.
\end{equation}
This is in contrast to \eqref{PS1.c.cond}; in other words, $I=\emptyset$.

\textbf{Claim 2:} $\nu_\infty=\mu_\infty=0$. To this end, first we prove that
\begin{equation}\label{PS1.mu.nu.inf.1}
	\mu_\infty \leq \frac{\delta^+}{m^-} \l\nu_\infty.
\end{equation}
Indeed, let $\psi \in C^\infty(\mathbb{R}^d)$ be such that $0 \leq \psi \leq 1$, $|\nabla \psi| \leq 2$ in $\mathbb{R}^d$, $\psi \equiv 0$ on $B_{1}$, and $\psi \equiv 1$ on $B_2^c$. For each $R > 0$, define $\psi_R(x) := \psi\left(\frac{x}{R}\right)$ for $x \in \mathbb{R}^d$. It follows that
\begin{align}
		\langle J_\l'(u_n) ,\psi_Ru_n \rangle &\geq m^- \int_{\RD}\psi_R \L[\Phi(x,|\nabla u_n|)+V(x) \Phi(x,|u_n|)\r]\diff x \label{raar5} \\&\ \  -\l \int_{\RD}^{} \L[b^+ \a \B(x,|u_n|)  + \delta^+\Phi_d(x,|u_n|)\r] \psi_R \diff x + \int_{\RD} \phi(x,|\nabla u_n|) \frac{\nabla u_n}{|\nabla u_n|} \nabla (\psi_R) u_n \diff x .\notag \\\notag
\end{align}
From the boundedness of $\{\psi_R u_n\}_{n \in \N}$ in $\WV$ and ($\textnormal{C}_2$), it follows, from \eqref{F}(i) and \eqref{H}, that
\begin{align}\label{J'un}
	\lim_{ n \to \infty}\Big\langle J_{\lambda}'(u_n) ,\psi_R u_n \Big\rangle=0.
\end{align}
Using the same arguments leading to \eqref{T.ccp.inf.mu_inf} and \eqref{T.ccp.inf.nu_inf}, we obtain
\begin{equation}\label{PS1.mu.nu.inf.3}
	\mu_\infty=\lim_{R\to\infty}\limsup_{n\to\infty}\int_{\mathbb{R}^d} \psi_R \Big[\Phi(x,|\nabla u_n|)+\lambda V(x)\Phi(x,|u_n|)\Big]\diff x,
\end{equation}
and
\begin{equation}\label{PS1.nu.inf.est1}
	\nu_\infty=\lim_{R\to\infty}\limsup_{n \to \infty}\int_{\mathbb{R}^d}\psi_R \Phi_d(x,|u_n|) \diff x.
\end{equation}
In view of Theorem \ref{thm:Injcn}(ii), it follows from \eqref{CVG2}, that
\begin{equation} \label{PS1.ets.f}
	\lim_{R \rightarrow \infty}\lim_{ n \to \infty}\int_{\mathbb{R}^d} \a \B(x,|u_n|) \psi_R \diff x = \lim_{R \rightarrow \infty} \int_{\mathbb{R}^d} \a \B(x,|u|) \psi_R \diff x = 0.
\end{equation}
Similar arguments to those leading to \eqref{raar.b} give, for an arbitrary $\epsilon >0$,
\begin{align} \label{A.nabla.phi}
	\int_{\mathbb{R}^d} \Big|\phi(x,|\nabla u_n|) \frac{\nabla u_n}{|\nabla u_n|} \nabla (\psi_R) u_n \Big| \diff x
	& \leq \epsilon M + C_\epsilon \int_{\mathbb{R}^d}\Phi(x,| \nabla \psi_R u_n|)\diff x,
\end{align}
with $M$ given by \eqref{M} and $C_\epsilon>0$ independent of $n$ and $R$.\\
On the other hand, from \eqref{raar7}, we have
\begin{equation} \label{H.nabla.phiR}
	\lim_{R \to \infty}\limsup_{n \to \infty} 	\int_{\mathbb{R}^d} \Phi(x,|\nabla \psi_R u_n|) \diff x =0.
\end{equation}
Therefore, taking limit superior in \eqref{raar5} as $n \to \infty$ and then taking limit as $R \to \infty$ with taking into account \eqref{J'un}-\eqref{H.nabla.phiR}, we obtain
\begin{equation*}
	m^-\mu_\infty\leq\l \delta^+\nu_\infty + \epsilon M.
\end{equation*}
Hence,  \eqref{PS1.mu.nu.inf.1} holds since $\epsilon>0$ is small arbitrarily.
Now, suppose on the contrary that $\nu_\infty>0$. From \eqref{CVG7} and \eqref{PS1.mu.nu.inf.1},  we have
\begin{equation*}
	S\min\left\{ (\l^{-1}\mu_\infty)^{\frac{1}{\delta_\infty}}, (\l^{-1}\mu_\infty)^{\frac{1}{\kappa_\infty}} \right\} \leq \max\left\{ \mu_\infty^{\frac{1}{m_\infty}}, \mu_\infty^{\frac{1}{\ell_\infty}} \right\}.
\end{equation*}
This leads to
\begin{equation}\label{PS1.nu.inf.est4}
	\mu_\infty \geq S^{\frac{\eta_\infty\beta_\infty}{\beta_\infty-\eta_\infty}}\l^{-\frac{\eta_\infty}{\beta_\infty-\eta_\infty}},
\end{equation}
with $\eta_\infty \in \{m_\infty,\ell_\infty\}$ and $\beta_\infty \in \{\kappa_\infty,\delta_\infty\}$. Note that the assumptions on exponents yield $m_\infty\leq \ell_\infty< \kappa_\infty\leq \delta_\infty$. We have
\begin{equation*}
	\left(\frac{m(x)\delta(x)}{\delta(x)-m(x)}\right)^- \leq\frac{m(x) \delta(x) }{\delta(x) -m(x)}\leq \frac{\ell(x) \kappa(x) }{\kappa(x) -\ell(x) } \leq \left(\frac{\ell(x)\kappa(x)}{\kappa(x)-\ell (x)}\right)^+, \quad \forall x \in \mathbb{R}^d.
\end{equation*}
Thus,
\begin{equation}\label{PL4.3-1}
	\left(\frac{m(x)\delta(x)}{\delta(x)-m(x)}\right)^-  \leq\frac{m_\infty \delta_\infty }{\delta_\infty -m_\infty}\leq \frac{\ell_\infty \kappa_\infty }{\kappa_\infty -\ell_\infty } \leq \left(\frac{\ell(x)\kappa(x)}{\kappa(x)-\ell (x)}\right)^+.
\end{equation}
On the other hand, for $\eta_\infty \in \{m_\infty,\ell_\infty\}$ and $\beta_\infty \in \{\kappa_\infty,\delta_\infty\}$ we have
\begin{equation}\label{PL4.3-2}
	\frac{m_\infty \delta_\infty }{\delta_\infty -m_\infty }\leq \frac{\eta_\infty \beta_\infty }{\beta_\infty -\eta_\infty }\leq \frac{\ell_\infty \kappa_\infty }{\kappa_\infty -\ell_\infty }.
\end{equation}
Combining \eqref{PL4.3-1} with \eqref{PL4.3-2} gives
\begin{equation*}
	\left(\frac{m(x)\delta(x)}{\delta(x)-m(x)}\right)^-  \leq \frac{\eta_\infty \beta_\infty}{\beta_\infty-\eta_\infty} \leq \left(\frac{\ell(x)\kappa(x)}{\kappa(x)-\ell (x)}\right)^+.
\end{equation*}
Similarly, it holds that
\begin{equation*}
	\left(\frac{m(x)}{\ell_\ast(x)-m(x)}\right)^- \leq \frac{\eta_\infty }{\beta_\infty-\eta_\infty} \leq\left(\frac{\ell(x)}{m_\ast(x)-\ell (x)}\right)^+.
\end{equation*}
The last two estimates, together with \eqref{PS1.mu.nu.inf.1} and \eqref{PS1.nu.inf.est4}, imply that
$$\l\nu_\infty\frac{\delta^+}{m^-}\geq \mu_\infty \geq \min\left\{S^{\tau_1},S^{\tau_2}\right\}\min \left\{\l^{-\sigma_1},\l^{-\sigma_2}\right\},$$	
with $\tau_1$, $\tau_2$, $\sigma_1$ and $\sigma_2$ given in \eqref{PL4.3-0}. From this and \eqref{PS1.c.nu}, we obtain	
\begin{align*}
	c &\geq \L( \frac{\kappa^-}{\ell^+}-1\r)\l\nu_\infty
	\geq  \L( \frac{\kappa^-}{\ell^+}-1\r)\frac{m^-}{\delta^+}\min\left\{S^{\tau_1},S^{\tau_2}\right\}\min \left\{\l^{-\sigma_1},\l^{-\sigma_2}\right\},
\end{align*}
a contradiction. Thus, $\nu_\infty=\mu_\infty=0.$	

By this and the fact that $I=\emptyset$, we deduce, from \eqref{CVG4} and \eqref{CVG6}, that
\begin{equation}\label{PL4.3-lim}\underset{n\to\infty}{\lim\sup}\int_{\mathbb{R}^d}\Phi_d(x,|u_n|)\diff x=\int_{\mathbb{R}^d}\Phi_d(x,|u|)\diff x.
\end{equation}
By \eqref{CVG1} and Fatou's lemma, we obtain
$$
\int_{\mathbb{R}^d}\Phi_d(x,|u|)\diff x\leq \liminf_{n\to \infty}\int_{\mathbb{R}^d}\Phi_d(x,|u_n|)\diff x.
$$
Combining this with \eqref{PL4.3-lim}, we infer that
$$\lim_{n\to\infty}\int_{\mathbb{R}^d}\Phi_d(x,|u_n|)\diff x=\int_{\mathbb{R}^d}\Phi_d(x,|u|)\diff x.$$
From this and Lemma~\ref{L.brezis-lieb}, we obtain
$$\lim_{n\to\infty}\int_{\mathbb{R}^d}\Phi_d(x,|u_n-u|)\diff x=0,$$
and thus, in view of Proposition~\ref{zoo}, it holds
\begin{equation}\label{cngee}
	u_n \to u \ \  \text{in} \ \ L^{\Phi_d}(\mathbb{R}^d).
\end{equation}

Now, we calim to prove that
\begin{equation}\label{TZTM2}
u_n\longrightarrow u \text{ in } \WV.
\end{equation}
First, observe that by taking $v = u_n - u$ in \eqref{eq0001} and passing to the limit as $n \to +\infty$, we obtain \begin{equation}\label{TZTM3}
\left\langle J_\lambda '\left(u_n\right), u_n-u\right\rangle \longrightarrow 0 \quad \text {as } n \rightarrow +\infty.
\end{equation}
Observe that, based on \eqref{F}(i), and by applying Hölder's inequality, Lemma \ref{lm1}, and Theorem \ref{thm:Injcn}, we deduce that \begin{equation}\label{TZTM4}
\begin{aligned}
\int_{\RD}^{} \L| f(x,u_n)(u_n-u)\r|\diff x & \leq \int_{\RD} \a b(x,|u_n|) |u_n-u| \diff x \\
& \leq 2 \|b(\cdot, |u_n|) \|_{L^{\widetilde{\B}} (\RD)}  \| \hat{a} (\cdot) (u_n-u)\|_{L^{\B} (\RD)}\\
&\leq 2(b^+-1)  \|u_n\|_{L^{\B} (\RD)} \|u_n-u\|_{L^{\B}_{\hat{a}} (\RD)}\longrightarrow0, \text{ as } n \to +\infty.
\end{aligned}
\end{equation}

On the other hand, employing \eqref{ala11}, \eqref{cngee}, Hölder inequality, and Lemma \ref{lm1},
\begin{equation}\label{TZTM5}
\begin{aligned}
  \int_{\RD}^{} \L|\Phi_d (x,|u_n|)\frac{u_n}{|u_n|} (u_n-u) \r| \diff x &\leq 2 \|\Phi_d (\cdot,|u_n|)\|_{L^{\widetilde{\Phi}^\ast}(\RD)}  \|u_n-u\|_{L^{\Phi_d}(\RD)}\\
  &\leq 2\delta^+ \|u_n\|_{L^{\Phi_d}(\RD)}  \|u_n-u\|_{L^{\Phi_d}(\RD)} \\
  &  \leq 2c\delta^+ \|u_n\|_{\WV}  \|u_n-u\|_{L^{\Phi_d}(\RD)}\longrightarrow0, \text{ as } n \to +\infty.
\end{aligned}\end{equation}
Combining \eqref{TZTM4} and \eqref{TZTM5} with \eqref{TZTM3}, we conclude that
	\begin{align*}
		\left\langle I'\left(u_n\right), u_n-u\right\rangle \rightarrow 0 \quad \text {as } n \rightarrow +\infty.
	\end{align*}
	Since $I'$ satisfies the $\left(S_{+}\right)$-property, see Theorem \ref{op2}(ii), we get Lemma \ref{TZTM2}. The proof is complete.
\end{proof}

The next lemma confirms that the functional \( J \) exhibits Mountain Pass geometry.
\begin{lemma} \label{Lem.J(u).geo} Assume that the conditions of Theorem \ref{thm:exis} are satisfied. Let $\lambda > 0$ and $\theta > 0$. Then, the following assertions hold.
\begin{enumerate}
	\item[(i)] There exist $\beta \in (0,1)$ and $\rho>0$ such that $J_\l(u) \geq \rho $ if $\|u\|_{\WV} =\beta$.
	\item[(ii)] There exists $v\in \WV$ independent of $\l$  such that $\|v\|_{\WV} >1$ and $J_\l(v)<0$ for all $\l>0$.
\end{enumerate}
\end{lemma}
\begin{proof}
\begin{enumerate}
	\item[(i)]
Let $\lambda >0$. 
By \eqref{ala11} and \eqref{ala11.1}, we find $C_1>1$ such that
\begin{equation}\label{4norms}
	\max\{\|u\|_{L^\mathcal{B}_{\hat{a}(\cdot)}(\RD)},\|u\|_{L^{\Phi_d}(\RD)}\} \leq C_1 \|u\|_{\WV},~~ \forall u \in \WV.
\end{equation}
For any $u\in \WV$ with $\|u\|_{\WV} = \beta\in \left(0,C_1^{-1}\right)$, we apply Proposition~\ref{zoo}, together with \eqref{F}(i), \eqref{m-n} and \eqref{4norms}, to obtain
$$
\begin{aligned}
J_\l(u)&= I(u)-K(u)\\
	 &\geq   I(u) -\int_{ \mathbb{R}^d}  \a \B(x,|u|)  \diff x
	-\l\int_{ \mathbb{R}^d} \Phi_d(x,|u|) \diff x\notag \\
	&\geq  \|u\|_{\WV}^{\ell^+} -\max\curly{\|u\|_{L^{\B}_{\hat{a}(\cdot)}(\RD)}^{b^+},\|u\|_{L^{\B}_{\hat{a}(\cdot)}(\RD)}^{b^-}} - \l \max\curly{\|u\|_{L^{\Phi_d}(\RD)}^{\delta^+},\|u\|_{L^{\Phi_d}(\RD)}^{\kappa^-}} \\
& \geq \|u\|_{\WV}^{\ell^+} - \L( C_1\|u\|_{\WV}\r)^{b^-} -\l \L( C_1\|u\|_{\WV}\r)^{\kappa^-}\\
&\geq \beta^{\ell^+} -C_1^{b^-}\beta^{b^-} - C_1^{\kappa^-}\l\beta^{\kappa^-}:=\rho.
\end{aligned}$$	
Since $\ell^+<b^- \leq \kappa^-$, we can choose $\beta>0$ small enough such that $\rho>0$, and thus, $(i)$ has been shown.
\item[(ii)] In order to get $(ii)$, from \eqref{F}(ii), for any $A>0$ there exits $R_A>0$, such that
\begin{equation}\label{eq: for f}
  A |t|^{\ell^+}\leq F(x,t), \text{ for all } x \in \RD \text{ and } |t| >R_A.
\end{equation}
Let us fix $w \in C^\infty_c(\RD) \setminus \{0\}$. For any $\tau> R  _A $ we have, from \eqref{eq: for f} and \eqref{F}(i), that
\begin{equation*}\begin{aligned}
	J_\l(\tau w) &= I(\tau w) -K(\tau w)\\
& \leq \tau^{\ell^+} I(w) - \int_{\RD}^{} F(x,\tau w) \diff x \\
& \leq \tau^{\ell^+} I(w) - \L( \int_{\curly{x\in \RD,\ |\tau  w|> R_A}}^{} +\int_{\curly{x\in \RD,\ |\tau  w|\leq R_A}} \r) F(x,\tau w) \diff x \\
& \leq \tau^{\ell^+} I(w) - A \tau ^{\ell ^+} \int_{\RD}^{} | w|^{\ell ^+} \diff x +\int_{\curly{ |\tau  w|\leq R_A}}^{} \a \B(x, |\tau w|) \diff x \\
& \leq \tau^{\ell^+} I(w) - A \tau ^{\ell ^+} \int_{\RD}^{} | w|^{\ell ^+} \diff x +\int_{\RD} \a \B(x, R_A) \diff x \\
& \leq \tau^{\ell^+} I(w) - A \tau ^{\ell ^+} \int_{\RD}^{} | w|^{\ell ^+} \diff x +C_{1}\int_{\RD} \a  \diff x \\
& \leq \tau^{\ell^+} I(w) - A \tau ^{\ell ^+} \int_{\RD}^{} | w|^{\ell ^+} \diff x +C_{2} .\\
 \end{aligned}
\end{equation*} 
It follows, when $A \tau ^{\ell ^+} \int_{\RD}^{} | w|^{\ell ^+} \diff x    >  I(w)$, that
$$
J_\l(\tau w) \longrightarrow -\infty \text{, as } \tau \to +\infty.
$$

Thus, by taking $v=\tau w$ with $\tau>0$ large enough, it holds  $\|v\|_{\WV}>1$, and $J_\l(v)<0$ for all $\l>0$. Clearly, $v$ is independent of $\l$.
\end{enumerate}
This ends the proof.
\end{proof}
\subsection{Proof of existence result ( Theorem \ref{thm:exis})}
Let $\rho$ and $v$ be determined in Lemma~\ref{Lem.J(u).geo} $(i)$-$(ii)$. Define
\begin{equation}\label{c_lamb}
c_\l:=\inf _{\gamma \in \Gamma} \max _{0 \leqslant  \tau \leqslant 1} J_\l(\gamma(\tau)) ,
\end{equation}
where
\begin{equation}
\Gamma:=\left\{\gamma \in C\left([0,1],\, \WV\right): \gamma(0)=0, \gamma(1)=v\right\}.
\end{equation}
From Lemma~\ref{Lem.J(u).geo}\,(i), the definition of \(c_\l\) and since \(\gamma(\tau)=\tau v\in \Gamma\), we have
\begin{equation} \label{c_l}
0<\rho\leq c_\l\leq\max_{0\leq\tau\leq 1}J_\l(\tau v).
\end{equation}
Meanwhile, for all \( 0 \leq \tau \leq 1 \), applying Lemma~\ref{xit}, along with \eqref{eq: for f}, yields
\begin{align*}	
J_\l(\tau v) &\leq \tau^{m^-} I(v) -A\tau^{\ell^+}  \int_{\curly{|\tau v| >R_A}} |v|^{\ell^+}\diff x+C_2 \notag \\
&\leq  a_0\tau^{m^-}-b_0\tau^{\ell^+}+C_2:=g(\tau),
\end{align*}
where
$$a_0:=\displaystyle\int_{\mathbb{R}^d} \Big[ \Phi(x,|\nabla v|) + \lambda V(x) \Phi(x,|v|) \Big]\diff x >0\ \ \text{and} \ \ b_0:=A\displaystyle\int_{\curly{|\tau v| >R_A}}  |v|^{\ell^+}\diff x >0.$$
From this and \eqref{c_l}, we deduce
\begin{align}\label{g}
0<c_\lambda\leq \max _{\tau\in [0,1]}g(\tau)=g\left(d_0\right)=a_0d_0^{m^-}-b_0d_0^{\ell^+}+C_2,
\end{align}
where $d_0:=\min\left\{1, \left(\frac{a_0m^-}{b_0\ell^+}\right)^{\frac{1}{\ell^+ - m^-}}\right\}$.

\vskip5pt\textcolor[rgb]{1.00,0.00,0.00}{}

 Note that for the function $v$ obtained in Lemma~\ref{Lem.J(u).geo} (ii), the quantity $g(d_0)$ is independent of the parameter $\lambda$. Consequently, there exists $\lambda_0 > 0$ such that for all $\lambda \in (0, \lambda_0)$, we have
\begin{equation} \label{theta_0}
    0 < g(d_0) < \left( \frac{\kappa^-}{\ell^+} - 1 \right)\frac{m^-}{\delta^+} \min \left\{ S^{\tau_1}, S^{\tau_2} \right\} \min \left\{ \lambda^{-\sigma_1}, \lambda^{-\sigma_2} \right\} - \widetilde{C}.
\end{equation}

Now fix $\lambda \in (0, \lambda_0)$. Combining \eqref{g} with \eqref{theta_0} yields the following key estimate for the mountain pass level
\begin{equation} \label{Bound.level.c}
    0 < \rho \leq c_\lambda < \left( \frac{\kappa^-}{\ell^+} - 1 \right)\frac{m^-}{\delta^+} \min \left\{ S^{\tau_1}, S^{\tau_2} \right\} \min \left\{ \lambda^{-\sigma_1}, \lambda^{-\sigma_2} \right\} - \widetilde{C}.
\end{equation}

Using the Mountain Pass geometry of $J_\l$ established in Lemma~\ref{Lem.J(u).geo}, we obtain a Cerami sequence $\{u_n\}_{n\in\mathbb{N}}$ for $J_\l$ in $\WV$ at level $c_\lambda$. The estimate \eqref{Bound.level.c} together with Lemma~\ref{crmi} implies that, up to a subsequence, $u_n \to u$ in $\WV$. This limit satisfies $J'_\lambda(u) = 0$ and $J_\l(u) = c_\lambda > 0$, and therefore constitutes a nontrivial weak solution to problem \eqref{prb}. This completes the proof.

\vspace{0.5cm}
\noindent{\bf Availability of data and materials}: Not applicable.

	\noindent
{\bf Authors' contributions.}
The authors contributed equally to this paper.

 \bigskip
\noindent \textsc{\textsc{Ala Eddine Bahrouni}} \\
Mathematics Department, Faculty of Sciences, University of Monastir,
5019 Monastir, Tunisia\\
 ({\tt ala.bahrouni@fsm.rnu.tn})

\bigskip
\noindent \textsc{\textsc{Anouar Bahrouni}} \\
Mathematics Department, Faculty of Sciences, University of Monastir,
5019 Monastir, Tunisia\\
 ({\tt Anouar.Bahrouni@fsm.rnu.tn; bahrounianouar@yahoo.fr})
\end{document}